\documentclass[11pt]{amsart}

\usepackage{graphicx}        
\usepackage{multicol}        
\usepackage[bottom]{footmisc}
\usepackage{amscd,amsmath,amssymb,amsfonts}
\usepackage[cmtip, all]{xy}

\newcommand{\chapter}{\section}

\newtheorem{theorem}{Theorem}[section]
\newtheorem{proposition}[theorem]{Proposition}
\newtheorem{lemma}[theorem]{Lemma}

\theoremstyle{definition}
\newtheorem{definition}[theorem]{Definition}
\theoremstyle{remark}
\newtheorem{remark}[theorem]{Remark}
\newtheorem{remarks}[theorem]{Remarks}

\numberwithin{equation}{section}

\newcommand{\CH}{{\rm CH}}

\newcommand{\Spec}{{\rm Spec\,}}

\newcommand{\0}{\emptyset}

\newcommand{\sB}{{\mathcal B}}
\newcommand{\sC}{{\mathcal C}}
\newcommand{\sD}{{\mathcal D}}
\newcommand{\sE}{{\mathcal E}}
\newcommand{\sF}{{\mathcal F}}

\newcommand{\sI}{{\mathcal I}}

\newcommand{\sL}{{\mathcal L}}
\newcommand{\sM}{{\mathcal M}}

\newcommand{\sO}{{\mathcal O}}

\newcommand{\sR}{{\mathcal R}}

\newcommand{\sZ}{{\mathcal Z}}
\newcommand{\A}{{\mathbb A}}

\renewcommand{\L}{{\mathbb L}}

\renewcommand{\P}{{\mathbb P}}

\newcommand{\Z}{{\mathbb Z}}

\newcommand{\id}{{\operatorname{\rm Id}}}

\newcommand{\Sch}{{\operatorname{\mathbf{Sch}}}}

\newcommand{\<}{\langle}
\renewcommand{\>}{\rangle}

\renewcommand{\dim}{{\operatorname{\rm dim}}} 
\newcommand{\Div}{{\operatorname{div}}}

\newcommand{\uu}{\underline}
\newcommand{\oo}{\otimes}
\newcommand{\Sm}{{\mathbf{Sm}}}

\newcommand{\Proj}{{\operatorname{Proj}}}

\newcommand{\cn}{{\tilde{c}_1}}

\newcommand{\Sym}{{\operatorname{Sym}}}

\newcommand{\Om}{{\Omega}}

\newcommand{\lci}{l.c.i.\ }

\newcommand{\can}{\text{can}}

\newcommand{\ind}[1]{}

\newcommand{\inp}[1]{}

\newcommand{\res}{\text{res}}
\newcommand{\un}[1]{\underline{#1}}
\newcommand{\+}{\widehat{+}}

\begin{document}
\setcounter{tocdepth}{1}

\title{Intersection theory in algebraic cobordism}

\author[M.~Levine]{Marc~Levine}
\address{Universit\"at Duisburg-Essen,
Fakult\"at Mathematik, Campus Essen, 45117 Essen, Germany}
\email{marc.levine@uni-due.de}

\begin{abstract}
We give a  more detailed construction of the operation ``intersection with a pseudo-divisor'' in algebraic cobordism. Using arguments in \cite[\S 6.2, 6.3]{CobordismBook}, this gives a new proof of the contra variant functoriality of algebraic cobordism for l.c.i. morphisms of schemes of finite type over a field of characteristic zero.
\end{abstract}


\maketitle

\tableofcontents
\section*{Introduction}
One important aspect of algebraic cobordism, as constructed in \cite{CobordismBook}, is the existence of a good theory of pull-back maps for arbitrary \lci morphisms (always working over a field of characteristic zero). The key ingredient which makes this possible is the {\em intersection map}
\[
D(-):\Omega_*(X)\to \Omega_{*-1}(|D|)
\]
for an effective Cartier divisor (more generally a {\em pseudo-divisor}) $D$  with support  $|D|$ on a finite type $k$-scheme $X$. The construction of the intersection map in  \cite[\S 6.2, 6.3]{CobordismBook} is accomplished with the help of a  refined cobordism group $\Omega_*(X)_D$, which  admits an explicitly defined intersection map
\[
D(-)_D:\Omega_*(X)_D\to \Omega_{*-1}(|D|)
\]
and   a ``forgetful map'' $\Omega_*(X)_D\to \Omega_*(X)$. The forgetful map is shown to be an isomorphism \cite[theorem 6.4.12]{CobordismBook}; the map $D(-)$ is defined by composing $D(-)_D$ with the inverse of the forgetful map. Auxiliary refinements $\Omega_*(X)_{D|D'}$ of $\Omega_*(X)$ are also defined, and are used to prove properties of the intersection map, most importantly, the commutativity property $D(D'(x))=D'(D(x))$ in $\Omega_{*-2}(|D|\cap |D'|)$ for $x\in\Omega_*(X)$ and pseudo-divisors $D$, $D'$ on $X$.

This approach has been criticized, as it is not proven (nor is it necessary for the construction) that the auxiliary groups $\Omega_*(X)_{D|D'}$ are themselves isomorphic to $\Omega_*(X)$. Fulton has suggested that one should be able to define a series of refined groups $\Omega_*(X)_{D_1,\ldots, D_n}$ for pseudo-divisors $D_1,\ldots, D_n$ on $X$, related by intersection maps and forgetful maps, so that all forgetful maps are isomorphisms, and giving rise to all necessary relations among the intersection maps. 

Carrying out such a program is the purpose of this paper. We define groups $\Omega_*(X)_{D_1,\ldots, D_n}$, with intersection maps
\[
D(-)_{D_*}:\Omega_*(X)_{D, D_1,\ldots, D_n}\to \Omega_{*-1}(|D|)_{D_1,\ldots, D_n}
\]
and forgetful maps
\[
\res_{D_*, D/D_*}:\Omega_*(X)_{D_1,\ldots, D_n, D}\to \Omega_*(X)_{D_1,\ldots, D_n}.
\]
We show that the forgetful maps are isomorphisms and that the resulting intersection maps
\[
D(-):\Omega_*(X)\to \Omega_{*-1}(|D|)
\]
have all the properties needed to give rise to pullback maps in algebraic cobordism for \lci morphisms via the technique of deformation to the normal bundle.  The groups $\Omega_*(X)_D$ (for a single pseudo-divisor) and the intersection map $D(-)_D:\Omega_*(X)_D\to \Omega_{*-1}(|D|)$ are exactly the same as those defined in \cite{CobordismBook}. The passage from the intersection maps to the \lci pullbacks is not discussed here;  the arguments and constructions of \cite[\S6.5, \S6.6]{CobordismBook}, relying on the properties of the intersection map detailed in proposition~\ref{prop:IntersectionProperties}, will do that. 

The requirement that  each forgetful map should be an isomorphism has led to a definition of the groups $\Omega_*(X)_{D_1,\ldots, D_n}$ that for $n=2$ differs from the groups $\Omega_*(X)_{D_1|D_2}$ defined in \cite{CobordismBook}. We hope that the more uniform approach pursued here has made the arguments more transparent and natural.

We fix a base-field $k$,  let $\Sch_k$ denote the category of separated finite type $k$-schemes and let $\Sm_k$ denote the full subcategory of smooth quasi-projective $k$-schemes. We often drop the qualifier ``separated'' and refer to an object of $\Sch_k$  as a finite type $k$-scheme. Although many of the constructions do not require resolution of singularities, or characteristic zero, we will assume that $k$ has characteristic zero. 
It would be possible to work over a perfect field of positive characteristic, assuming resolution of singularities, but we have preferred to avoid this (at present illusory) generality as it allows us to shorten the argument for lemma~\ref{lem:StrongDim} by using Bertini's theorem. The corresponding result \cite[lemma 6.1.13]{CobordismBook} was proven without the use of Bertini's theorem and a similar proof would work in this setting. 

We will be using only the following properties of ``resolution of singularities over $k$'': The field $k$ should be perfect. For each integral $X\in \Sch_k$, there is a projective birational map $q:Y\to X$ with $Y$ in $\Sm_k$ such that $q$ is an isomorphism over the smooth locus in $X$. In addition, if $D$ is an effective Cartier divisor on $X$, whose restriction to a smooth open subscheme $U\subset X$ is a simple normal crossing divisor, there is a projective birational map $q:Y\to X$ with $Y$ in $\Sm_k$ such that $q$ is an isomorphism over $U$,  and  $q^*(D)+E$ is a simple normal crossing divisor on $Y$, where $E$ is the exceptional locus of $q$. For proofs of these facts over a field of characteristic zero, we refer the reader to \cite{BM, Cut, EV, Kollar}.

In section~\ref{sec:DivI} we define the refined cobordism groups $\Omega_*(X)_{D_1,\ldots, D_n}$. We introduce the notion of admissibility of an effective simple normal crossing divisor $E$ with respect to $D_*:=D_1,\ldots, D_n$ and define the divisor class $[E\to |E|]_{D_*}\in \Omega_*(|E|)_{D_*}$ for an admissible  simple normal crossing divisor $E$; these classes will play a central role in our definition of the intersection map.  The groups $\Omega_*(X)_{D_*}$ admit 1st Chern class operators for line bundles on $X$ and we discuss some basic properties of these operators and their relation with the divisor classes. In section~\ref{sec:Intersection} we construct the intersection map on generators for   $\Omega_*(X)_{D_*}$ and derive some of its important properties. In section~\ref{sec:Descent}, we show that intersection map descends to a map on $\Omega_*(X)_{D_*}$. We establish the relation of commutativity of intersection maps in section~\ref{sec:Relations} and show  under certain technical conditions the equality of intersection with linearly equivalent divisors. We use these relations and a construction of explicit ``distinguished liftings'' to prove that the forgetful maps are isomorphisms, theorem~\ref{thm:Moving}, in section~\ref{sec:MovLem}. We conclude in section~\ref{sec:IntersectionMap} with the definition of the intersection map on $\Omega_*(-)$ and a proof of its main properties. 

We wish to mention explicitly, that although we refer to our  theorem~\ref{thm:Moving}, stating that the forgetful maps are isomorphisms,  as a `moving lemma', the arguments rely for their algebro-geometric input on resolution of singularities. Techniques using projecting cones or any of the other aspects of the classical Chow's moving lemma or its more modern extension to moving lemmas for Bloch's cycle complexes do not appear. If one would return to the Chow groups via the isomorphism $\Omega^*\otimes_{\L}\Z\cong \CH^*$ of \cite[Theorem 4.5.1]{CobordismBook}, one recovers Fulton's definition of intersection with a pseudo-divisor, and our results say nothing about the earlier proofs of the contravariant functoriality of the Chow groups of smooth varieties that rely on Chow's moving lemma. The added difficulty here is due to the fact that the generators of algebraic cobordism are smooth varieties rather than integral closed subschemes and one needs a modified version of intersection with a pseudo-divisor that takes this into account. 

\section{Refined cobordism}\label{sec:DivI}  

\subsection{Pseudo-divisors}\ind{pseudo-divisor} Let $X$ be a finite
type $k$-scheme. Following Fulton \cite{Fulton}, a {\em pseudo-divisor} $D$ on $X$ is\ind{Fulton, William}
a triple $D:=(Z, \sL, s)$, where $Z\subset X$ is a closed subset, $\sL$ is an invertible sheaf on $X$, and $s$ is a section of $\sL$ on $X$, such that the subscheme $s=0$ has support contained in $Z$;  one identifies triples $(Z, \sL, s)$, $(Z, \sL', s')$ if there is an isomorphism $\phi:\sL\to \sL'$ with $s'=\phi(s)$. In particular, having fixed $\sL$, the section $s$ is determined exactly up to a global unit on $X$. If we have a morphism $f:Y\to X$, we define $f^*(Z,\sL,s):=(f^{-1}(Z), f^*\sL,f^*s)$; clearly $(fg)^*(D)=g^*(f^*D)$ for a pseudo-divisor $D$. Also, an effective Cartier divisor $D$ on $X$ uniquely determines a pseudo-divisor $(|D|,\sO_X(D), s_D)$, where
$s_D:\sO_X\to\sO_X(D)$ is the canonical section and $|D|$ is the support of $D$.

We call $Z$ the {\em support}  of a pseudo-divisor $D:=(Z,\sL,s)$, and write
$Z=|D|$.   We let $\Div D$ denote the subscheme $s=0$, and write $\sO_X(D)$ for $\sL$. If $X$ is in $\Sm_k$, if $|D|=|\Div D|$ and if this subset has pure codimension one on $X$, then we identify $D$ with the Cartier divisor $\Div D$.

We will not be needing the full flexibility of all pseudo-divisors. In this paper, we will always assume that the support $Z$ of a pseudo-divisor $D:=(Z,\sL,s)$ is given by the closed subset $s=0$. Thus,  a pseudo-divisor on a smooth irreducible $Y$ is either  $(|D|, \sO_Y(D), s_D)$ for an effective Cartier divisor $D$ on $Y$, or $(Y, \sL, 0)$ for some invertible sheaf $\sL$ on $Y$. 

The zero pseudo-divisor is $(\0,\sO_X,1)$. If we have pseudo-divisors $D=(Z,\sL,s)$ and $D'=(Z',\sL',s')$, define $D+D'=(Z\cup Z',\sL\otimes\sL',s\otimes s')$. A pseudo-divisor $C$ is a {\em sub-pseudo-divisor} of a pseudo-divisor $D$ if there is an effective Cartier divisor $E$ with $C+E=D$. We say that $C$ is {\em supported in $D$} if $C$ is a sub-pseudo-divisor of $mD$ for some integer $m\ge1$. If $f:Y\to X$ is a morphism of finite type $k$-schemes and $C$ and $D$ are pseudo-divisors on $X$ with $C$ supported in $D$, then $|f^*C|\subset |f^*D|$ and, if $f^*D$ is a Cartier divisor on $Y$, necessarily effective, then $f^*C$ is also an effective  Cartier divisor on $Y$, with support contained in  $|f^*D|=|\Div f^*D|$.

\subsection{Refined cobordism cycles} 
Let $X$ be a finite type $k$-scheme and $D_1$,$\ldots$, $D_r$  pseudo-divisors on $X$. We proceed to define a series of groups
\[
\sZ_*(X)_{D_*}\to\uu{\sZ}_*(X)_{D_*}\to\uu{\Om}_*(X)_{D_*}\to\L_*\oo\uu{\Om}_*(X)_{D_*}
\to\Omega_*(X)_{D_*}
\]
analogous to the sequence
\[
\sZ_*(X)\to\uu{\sZ}_*(X)\to\uu{\Om}_*(X)\to\L_*\oo\uu{\Om}_*(X)
\to\Omega_*(X)
\]
used to define $\Omega_*(X)$ in  \cite[\S2.4]{CobordismBook}; in this section we construct the group $\sZ_*(X)_{D_*}$.

Let  $E=\sum_{i=1}^mn_iE_i$ be an effective simple normal crossing on a scheme $W\in\Sm_k$  with irreducible components $E_1,\ldots, E_m$. For each $J\in \{0, 1\}^m$, $J=(j_1,\ldots, j_m)$, we have the {\em  face}
\[
E^J:=\cap_{j_i=1}E_i,
\]
which is smooth over $k$ and has codimension $|J|:=\sum_ij_i$ on $W$.   All Cartier divisors we will be using here will be effective and will often refer to an effective simple normal crossing divisor as a simple normal crossing divisor.

Let $X$ be a finite type $k$-scheme. We recall from \cite[\S 2.1.2]{CobordismBook} the set $\sM(X)$ of isomorphism classes of projective morphisms $f:Y\to X$, with $Y$ in $\Sm_k$ (where ``isomorphism" means isomorphism over $X$). $\sM(X)$ is a monoid under disjoint union and is the free monoid on the isomorphism classes of $f:Y\to X$ with $Y$ irreducible.

Let $\sM(X)_D$ be the submonoid of $\sM(X)$ generated by $f:Y\to X$, with $Y$ irreducible, and with either $f(Y)\subset|D|$, or with $\Div f^*D$ a simple normal crossing divisor on $Y$.  We extend this construction as follows.

\begin{definition} \label{def:AdmissibleMaps}Let $X$ be a finite type $k$-scheme, and let $D_1,\ldots, D_r$ be pseudo-divisors on $X$.\\
i) Let $f:Y\to X$ be a morphism, with $Y$ irreducible. Let $I_f=\{i\ |\ |f^*D_i|\neq Y\}$.  If $I_f\neq\0$, let $i_0$ be the smallest $i\in I_f$ and call $D_{i_0}$ the  {\em leading} pseudo-divisor for $f$.   If $I_f=\0$, we say that   $f$ has no leading pseudo-divisor.  \\[5pt]
ii) Let $f:Y\to X$ be a morphism with $Y\in\Sm_k$. Suppose $Y$ is irreducible. We say $f$ that is {\em admissible} with respect to $D_1,\ldots, D_r$ if  for each $s$, $1\le s\le r$, the scheme-theoretic intersection $\cap_{i=1}^s\Div f^*D_i$ is either $Y$, or is empty, or is a Cartier divisor with support a simple normal crossing divisor on $Y$. If $Y$ is not irreducible, $f$ is admissible if the restriction of $f$ to each irreducible component of $Y$ is so.  The submonoid $\sM(X)_{D_1,\ldots, D_r}$ of $\sM(X)$ is the subset consisting of those $f:Y\to X$ that are admissible with respect to $D_1,\ldots, D_r$.
\\
iii)  Let $f:Y\to X$ be a morphism, with $Y\in \Sm_k$. An effective Cartier divisor $E$ on $Y$ is called {\em admissible} with respect to $D_1,\ldots, D_r$ if $E$  is a simple normal crossing divisor on $Y$ and,  for each face $E^J$ of $E$, the composition $E^J\to Y\to X$ is  admissible with respect to $D_1,\ldots, D_r$.
\end{definition} 

When the collection of pseudo-divisors is understood we sometimes write $\sM(X)_{D_*}$ for $\sM(X)_{D_1,\ldots, D_r}$.  If $D_1,\ldots, D_r$ and $D'_1,\ldots, D'_s$ are two sequences of pseudo-divisors on $X$, we write $D_*, D_*'$ for the sequence $D_1,\ldots, D_r$, $D_1',\ldots, D_s'$. For an $X$-scheme $f:X'\to X$, we write $f^*(D_*)$ for the sequence of pseudo-divisors $f^*(D_1),\ldots, f^*(D_r)$ and often write $\sM(X')_{D_*}$ for $\sM(X')_{f^*(D_*)}$ if the map $f$ is understood.

\begin{remarks} (1) For $f:Y\to X$ in $\sM(X)$, $f$ is in $\sM(X)_{D_*}$ exactly when $\id_Y$ is in $\sM(Y)_{D_*}$.   \\[5pt]
(2) Suppose $X$ is irreducible, $D_1,\ldots, D_r$ pseudo-divisors on $X$. Write $I_{\id_X}=\{i_0, i_1,\ldots, i_s\}$ and let $D^{ess}_j=D_{i_j}$. Then $\sM(X)_{D_*}=\sM(X)_{D^{ess}_*}$. Also, $\sM(X)_\0=\sM(X)$.\\[5pt]
(3) Let $f:Y\to X$ be in $\sM(X)_{D_*}$ with $Y$ irreducible. Suppose that $f$ has leading pseudo-divisor $D_{i_0}$ and take $s$ with  $i_0\le s\le r$. Let $D_{i_0}^1,\ldots, D_{i_0}^m$ be the irreducible components of the simple normal crossing divisor $\Div f^*D_{i_0}$. Then the intersection  $E_s:=\cap_{i=i_0}^s\Div f^*D_i$ is an effective Cartier divisor of the form $\sum_jm_jD_{i_0}^j$; in particular, $E_s$  is  a simple normal crossing divisor on $Y$  with $0\le E_s\le \Div f^*D_{i_0}$.\\[5pt]
(4) The condition in definition~\ref{def:AdmissibleMaps}(ii) is equivalent to the following: For $f:Y\to X$ with $Y$ irreducible in $\Sm_k$, $f$ is admissible if $f$ has no leading pseudo-divisor or $f$ has leading pseudo-divisor $D_{i_0}$ and

\addtolength{\textwidth}{-15pt}
\ \hskip2pt\begin{minipage}[c]{\textwidth} 
\ \\
\hbox to0pt{\hss a)\ }$\Div f^*D_{i_0}$ is a simple normal crossing divisor on $Y$.\\
\hbox to0pt{\hss b)\ }For all $s$ with $i_0\le s\le r$, the scheme-theoretic intersection $\cap_{i=i_0}^s\Div f^*D_i$ is a Cartier divisor on $Y$, or is empty.
\end{minipage}
\addtolength{\textwidth}{15pt}
\end{remarks}

\begin{lemma}\label{lem:AdmissibleSupport} Let $D_1,\ldots, D_r$ be pseudo-divisors on $X\in \Sch_k$,  let $f:Y\to X$ be a morphism with $Y\in \Sm_k$, $Y$ irreducible, and  with $\id_Y\in \sM(Y)_{D_*}$.  Let $E$ be an effective Cartier divisor on $Y$. Suppose that $D_{i_0}$ is the leading pseudo-divisor for $f$ and that $E+\Div f^*D_{i_0}$ is a simple normal crossing divisor on $Y$. Then $E$ is admissible with respect to $D_1,\ldots, D_r$. \end{lemma}

\begin{proof} To simplify the notation, we may replace $X$ with $Y$, $D_*$ with $f^*(D_*)$ and $f$ with $\id_Y$. We may therefore assume that $X$ is in $\Sm_k$, is irreducible and $f=\id_X$.

Under our assumption, $E$ is a simple normal crossing divisor; so  let $F$  be an irreducible component of a face $E^J$ of $E$ and let $g:F\to X$ be the inclusion  $F\to  X$.  

Suppose that $F$ is not contained in $|D_{i_0}|$. Then $D_{i_0}$ is also the leading pseudo-divisor for $g$. Since $F$ is a component of a face of the simple normal crossing divisor $E+\Div D_{i_0}$, $\Div g^*D_{i_0}$ is a simple normal crossing divisor on $F$. Similarly, for all $s$, $i_0\le s\le r$ the scheme-theoretic intersection on $F$, $\cap_{i=i_0}^s\Div g^*D_i$ is just $g^{-1}(\cap_{i=i_0}^s\Div D_i)$. As $\cap_{i=i_0}^s\Div D_i$ is a Cartier divisor on $X$, with support in $\Div D_{i_0}$,  it follows that $\cap_{i=i_0}^s\Div g^*D_i$ is a Cartier divisor on $F$.

Suppose that $F$ is contained in $|D_{i_0}|$. If $g$ has no leading pseudo-divisor, then clearly $F\to X$ is in $\sM(X)_{D_*}$; if $g$ has leading pseudo-divisor $D_{i_1}$, then $i_0<i_1$ and $F\subset \Div D_i$ for $i_0\le i<i_1$. Thus,    $\Div g^*D_{i_1}=g^{-1}(\cap_{i=i_0}^{i_1}\Div D_i)$.  Since $\cap_{i=i_0}^{i_1}\Div D_i$ is a Cartier divisor on $X$ with $0\le \cap_{i=i_0}^{i_1}\Div D_i\le \Div D_{i_0}+E$, and $F$ is a face of the simple normal crossing divisor $\Div D_{i_0}+E$, it follows that $\Div g^*D_{i_1}$ is a simple normal crossing divisor on $F$. Similarly, for all $s$ with $i_1\le s\le r$, $\cap_{i=i_1}^s g^*\Div D_i$ is a Cartier divisor on $F$, and thus $g:F\to X$ is in $\sM(X)_{D_*}$.
\end{proof}

\begin{lemma} \label{lem:AdmissibleIntersection}  Let $D_1,\ldots, D_r$ be pseudo-divisors on $X\in \Sch_k$,  let $f:Y\to X$ be a morphism with $Y\in \Sm_k$, $Y$ irreducible, and with $\id_Y\in\sM(Y)_{D_*}$. Suppose that $D_{i_0}$ is the leading pseudo-divisor for $f$. Then $\Div f^*D_{i_0}$ is simple normal crossing divisor on $Y$, admissible with respect to $D_{i_0+1},\ldots, D_r$.
\end{lemma}

\begin{proof} As above, we reduce to the case $Y=X$, $f=\id_Y$. Since  $\id_Y$ is in $\sM(Y)_{D_*}$, it follows that $\Div D_{i_0}$ is a simple normal crossing divisor on $Y$. From lemma~\ref{lem:AdmissibleSupport} (with $E=\Div D_{i_0}$), it follows that $\Div D_{i_0}$ is admissible with respect to $D_1,\ldots, D_r$. If now $F$ is a face of $\Div D_{i_0}$, then the inclusion $i_F:F\to Y$ is in $\sM(Y)_{D_*}$. However, $Y=|D_i|$ for $i=1,\ldots, i_0-1$ and $F\subset |D_{i_0}|$, so either $i_F$ has no leading pseudo-divisor or $i_F$ has leading pseudo-divisor $D_{i_1}$ with $i_0<i_1\le r$. Therefore $i_F$ is in $\sM(Y)_{D_{i_0+1},\ldots, D_r}$; as $F$ was an arbitrary face of $\Div D_{i_0}$, it follows that $\Div D_{i_0}$ is admissible  with respect to $D_{i_0+1},\ldots, D_r$.
\end{proof}

Recall from  \cite[definition 2.1.6]{CobordismBook} the notion of a  cobordism cycle over $X$, namely, a tuple $(f:Y\to X, L_1,\ldots, L_m)$ with $Y\in\Sm_k$. $Y$ irreducible,  $f$ a morphism representing a class in $\sM(X)$ and the $L_i$ line bundles on $Y$. Two   cobordism cycles $(f:Y\to X, L_1,\ldots, L_m)$  and $(f':Y'\to X, L'_1,\ldots, L'_m)$   are isomorphic if there is an isomorphism $\phi:Y\to Y'$ over $X$ and a permutation $\sigma$ such that $\phi^*L'_i\cong L_{\sigma(i)}$ for $i=1,\ldots m$.  We have the free abelian group $\sZ_*(X)$ on the isomorphism classes of cobordism cycles, graded by giving $(f:Y\to X, L_1,\ldots, L_m)$ degree $\dim_kY-m$. 

\begin{definition} For $X\in \Sch_k$ with pseudo-divisors $D_1,\ldots, D_r$, we let $\sZ_*(X)_{D_*}$ be the graded subgroup of $\sZ_*(X)$ generated by the cobordism cycles $(f:Y\to X, L_1,\ldots, L_m)$ with $f:Y\to X$ in $\sM(X)_{D_*}$. 
\end{definition}

For $0\le s\le r$ and $D'_*$ the sequence $D_1, \ldots, D_s$, we have the inclusion
\[
\res_{D_*/D_*'}:\sM_*(X)_{D_*}\to \sM_*(X)_{D'_*}
\]
which defines the inclusion $\res_{D_*/D'_*}:\sZ_*(X)_{D_*}\to\sZ_*(X)_{D'_*}$.

\subsection{The dimension axiom} We define the group $\un{\sZ}_*(X)_{D_*}$.

\begin{definition} Let $X$ be in $\Sch_k$ and let $D_1,\ldots, D_r$ be pseudo-divisors on $X$. Let $\<\sR_*^{Dim}\>(X)_{D_*}$  be the subgroup of $\sZ_*(X)_{D_*}$ generated by cobordism cycles of the form
\[
(f:Y\to X, \pi^*(L_1),\ldots, \pi^*(L_m),M_1,\ldots, M_s),
\]
where $\pi:Y\to Z$ is a smooth quasi-projective morphism, $Z$ is in $\Sm_k$, $L_1,\ldots, L_m$ are line bundles on $Z$ and $m>\dim_kZ$. We set
\[
\uu{\sZ}_*(X)_{D_*}:= \sZ_*(X)_{D_*}/\<\sR_*^{Dim}\>(X)_{D_*}.
\]
\end{definition}

We have functoriality for smooth quasi-projective morphisms of relative dimension $d$, $f:X'\to X$:
\begin{align*}
&f^*:\sM(X)_{D_*}\to \sM(X')_{f^*(D_*)};\hskip5pt f^*(g:Y\to X):=p_1:X'\times_XY\to X',\\
&f^*:\sZ_*(X)_{D_*}\to \sZ_{*+d}(X')_{f^*(D_*)};\\&\hskip150pt f^*(g, L_1,\ldots, L_m):=(f^*g, p_2^*L_1,\ldots, p_2^*L_m)\\
&f^*:\uu{\sZ}_*(X)_{D_*}\to \uu{\sZ}_{*+d}(X')_{f^*(D_*)};\\&\hskip150pt f^*(g, L_1,\ldots, L_m):=(f^*g, p_2^*L_1,\ldots, p_2^*L_m),
\end{align*}
and push-forward maps for projective morphisms $f:X'\to X$:
\begin{align*}
&f_*:\sM(X')_{f^*(D_*)}\to \sM(X)_{D_*};&f_*(g:Y\to X'):=f\circ g:Y\to X,\\
&f_*:\sZ_*(X')_{f^*(D_*)}\to \sZ_*(X)_{D_*};&f_*(g, L_1,\ldots, L_m):=(f_*(g), L_1,\ldots, L_m),\\
&f_*:\uu{\sZ}_*(X')_{f^*(D_*)}\to \uu{\sZ}_*(X)_{D_*};&f_*(g, L_1,\ldots, L_m):=(f_*(g), L_1,\ldots, L_m).
\end{align*}
Also, for $L\to X$ a line bundle on $X$, we have the Chern class endomorphism
\[
\cn(L):\sZ_*(X)_{D_*}\to \sZ_{*-1}(X)_{D_*},
\]
defined as for $\sZ_*(X)$:
\[
\cn(L)((f:Y\to X,L_1,\ldots, L_r)):=(f:Y\to X,L_1,\ldots, L_r, f^*L).
\]
This descends to the locally nilpotent endomorphism
\[
\cn(L):\uu{\sZ}_*(X)_{D_*}\to\uu{\sZ}_{*-1}(X)_{D_*}.
\]
The operation of product over $k$ defines external products
\[
\times:\sZ_*(X)\otimes\sZ_*(X')_{D_*}\to
\sZ_*(X\times_kX')_{D_*},
\]
which descend to $\uu{\sZ}_*(-)_{-}$.  

For $0\le s\le r$ and for $D_*'$ the subsequence $D_1,\ldots, D_s$,  the map $\res_{D_*/D'_*}:\sZ_*(X)_{D_*}\to \sZ_*(X)_{D'_*}$ descends to a homomorphism
\[
\res_{D_*/D'_*}:\uu{\sZ}_*(X)_{D_*}\to \uu{\sZ}_*(X)_{D'_*}.
\]

\begin{remark}\label{rem:Injective} The homomorphism $\res_{D_*/D'_*}:\uu{\sZ}_*(X)_{D_*}\to \uu{\sZ}_*(X)_{D'_*}$ is a split injection. Indeed, we  have the identity
\[
\<\sR_*^{Dim}\>(X)_{D_*}= \<\sR_*^{Dim}\>(X)_{D'_*} \cap \sZ_*(X)_{D_*}
\]
and both inclusions $\<\sR_*^{Dim}\>(X)_{D_*}\subset \sZ_*(X)_{D_*}$, $\<\sR_*^{Dim}\>(X)_{D'_*}\subset {\sZ}_*(X)_{D'_*}$ are inclusions of free subgroups on subsets of the generators of the free abelian groups $\sZ_*(X)_{D_*}$, $\sZ_*(X)_{D'_*}$.
In particular, taking $D'_*$ to be the empty set of pseudo-divisors, we may identify  $\L_*\otimes \uu{\sZ}_*(X)_{D_*}$ with  this $\L_*$-submodule of $\L_*\otimes\uu{\sZ}_*(X)$; this allows us to speak of the intersection $\L_*\otimes \uu{\sZ}_*(X)_{D_*}\cap \L_*\otimes \uu{\sZ}_*(X)_{D'_*}$ for two (or more) sequences of pseudo-divisors on $X$, the intersection taking place in $\L_*\otimes\uu{\sZ}_*(X)$. We define intersections
$\uu{\sZ}_*(X)_{D_*}\cap\uu{\sZ}_*(X)_{D'_*}$, ${\sZ}_*(X)_{D_*}\cap\uu{\sZ}_*(X)_{D'_*}$ or $\sM_*(X)_{D_*}\cap \sM_*(X)_{D'_*}$ similarly, the intersections taking place in $\un{\sZ}_*(X)$, $\sZ_*(X)$ and $\sM_*(X)$, respectively.
\end{remark}

The maps $\res_{D_*/D'_*}$ are all natural with respect to the operations described above,  in particular, taking $D'_*$ to be the empty set of pseudo-divisors, the operations on $\sM_*(-)_{D_*}$,  $\sZ_*(-)_{D_*}$ and $\un{\sZ}_*(-)_{D_*}$ are compatible with the corresponding ones on $\sM_*$, $\sZ_*$ and $\un{\sZ}_*$ via $\res_{D_*/\0}$. Consequently, all relations and compatibilities among these operations that hold in $\sM(-)$, $\sZ_*(-)$ or $\un{\sZ}_*(-)$ hold for $\sM(-)_{D_*}$, $sZ_*(-)_{D_*}$ or $\un{\sZ}_*(-)_{D_*}$  as well.

\subsection{Good position} We define a notion of ``good position" of a divisor $E$ with respect to a sequence of pseudo-divisors $D_1,\ldots, D_r$.

\begin{definition}\label{Def:GoodPosition} 
Let $Y$ be in $\Sm_k$, irreducible, and let  $D_1,\ldots, D_r$ be pseudo-divisors on $Y$ such that $\id_Y$ is in $\sM(Y)_{D_*}$. Let $E$ be an effective simple normal crossing divisor on $Y$.  We say that $E$ is {\em in  good position with respect to $D_*$} if either
 
\addtolength{\textwidth}{-15pt}
\ \hskip5pt\begin{minipage}[c]{\textwidth}
\ \\
\hbox to0pt{\hss i)\ }$\id_Y$ has no leading pseudo-divisor,\\
\hbox to0pt{\hss or\hskip7pt} \\
\hbox to0pt{\hss ii)\ }$\id_Y$ has leading pseudo-divisor $D_{i_0}$ and $E+\Div D_{i_0}$ is a simple normal crossing divisor on $Y$.
\end{minipage}
\addtolength{\textwidth}{15pt}
\\
We extend this notion to Cartier divisors $E$ on a not necessarily irreducible  $Y\in\Sm_k$ by requiring that $E\cap Y_i$ is in good position with respect to $D_*$ for each irreducible component $Y_i$ of $Y$.
\end{definition}

\begin{remarks}\label{rem:GenPosition} (1) For  $f:Y\to X$ in $\sM(X)_{D_*}$ with  $E$ a simple normal crossing divisor on $Y$,  if $E$ is in good position with respect to $D_*$, then $E$ is admissible with respect to $D_*$; this is lemma~\ref{lem:AdmissibleSupport}.   \\
(2) If $E$ is in good position with respect to $D_*$ on some $Y\in\Sm_k$ and $C$ is an effective Cartier divisor with $|C|\subset |E|$, then $C$ is also in good position with respect to $D_*$.
\end{remarks}

\begin{lemma}\label{lem:GoodPosition} Let $D, D_1, \ldots, D_r$ be pseudo-divisors on $X$, let $f:Y\to X$ be morphism with $Y\in \Sm_k$, $Y$ irreducible, and with $f$ in $\sM(X)_{D, D_*}$. Let $E$ be a simple normal crossing divisor on $Y$, in good position with respect to $D, D_*$. Suppose that $f(Y)\not\subset |D|$. Then\\[5pt]
(1)  $\Div f^*D$ is a simple normal normal crossing divisor on $Y$, admissible for $D_*$ and in good position with respect to $D, D_*$. \\[5pt]
(2) For each irreducible component $F$ of a face of $\Div f^*D$, if $F\not\subset |E|$, then $F\cap E$ is a simple normal crossing divisor on $F$, in good position with respect to $D_*$.\\[5pt]
(3) Let $F$ be an irreducible component of a face of $E+\Div f^*D$ with $F\subset |f^*D|$. Then the inclusion $F\to Y$ is in $\sM(Y)_{D_*}$.\\[5pt]
(4) Suppose that $E$ is smooth over $k$, and let $F_1\subset F_2$ be irreducible components of faces of $\Div f^*D$, with $F_1\not\subset |E|$ and $F_1$ a codimension one closed subscheme of $F_2$. Then for $i=1,2$, $F_i\cap E$ is smooth, the inclusion $F_i\cap E\to Y$ is in $\sM(Y)_{D_*}$  and the divisor $F_1\cap E$  on $F_2\cap E$ is in good position with respect to $D_*$.\\[5pt]
(5) Let $F_1\subset F_2$ be the inclusion of irreducible components of faces of $E$, with $F_1$ a codimension one closed subscheme of $F_2$. Then the divisor $F_1$ on $F_2$ is in good position with respect to $D, D_*$. \\[5pt]
(6) Let $C$ and $S$ be effective Cartier divisors on $Y$ with support contained in $E$. Suppose that $S$ is smooth and $C$ and $S$ have no common components. Then $S\to Y\to X$ is in $\sM(X)_{D, D_*}$ and the simple normal crossing divisor $C\cap S$ on $S$ is in good position with respect to $D, D_*$.
\end{lemma}

\begin{proof} We write $D_0$ for $D$ when convenient. Under our assumptions, $D$ is the leading pseudo-divisor for $f$ with respect to $D, D_1,\ldots, D_r$. By  lemma~\ref{lem:AdmissibleIntersection}, $\Div f^*D$ is a simple normal crossing divisor on $Y$, admissible with respect to $D_*$.  As $\Div f^*D+\Div f^*D$ is a simple normal crossing divisor on $Y$, $\Div f^*D$ is in good position with respect to $D, D_*$,  proving (1).

For (2),  take $F$ to be an irreducible component of a face of $\Div f^*D$. Since $E+\Div f^*D$ is a simple normal crossing divisor on $Y$, it follows that $F\cap E$ is a simple normal crossing divisor on $F$ as long as $F\not\subset |E|$. Let $g:F\to X$ be the composition $f\circ i_F$, where $i_F:F\to Y$ is the inclusion. 

As $\Div f^*D$ is admissible for $D_*$, $g$ is in  $\sM(X)_{D_*}$.  Supposing that $g$ has no leading pseudo-divisor with respect to $D_*$, it follows that $i_F^*E$ is in good position with respect to $D_*$. Suppose then that $D_{i_0}$ is the leading pseudo-divisor for $g$ with respect to $D_*$. Then $F\subset \Div f^*D_i$ for all $i$, $0\le i<i_0$ and thus $\Div g^*D_{i_0}= i_F^*(\cap_{i=0}^{i_0}\Div f^*D_i)$. But since $\id_Y$ is in $\sM(Y)_{D, D_*}$,  $\cap_{i=0}^{i_0} \Div f^*D_i$  is a Cartier divisor on $Y$ and is a subdivisor of  $\Div f^*D=\Div f^*D_0$.  Thus  $\cap_{i=0}^{i_0} \Div f^*D_i+E$ is a simple normal crossing divisor on $Y$ and is a subdivisor of $\Div f^*D+E$. As $F$ is a component of a face of $\Div f^*D+E$, it follows that $i_F^*(E+ \Div f^*D_{i_0})$ is a simple normal crossing divisor on $F$, and thus $i_F^*E$ is in good position with respect to $D_*$. 

The assertion  (3) follows from remark~\ref{rem:GenPosition}(1): $E+\Div f^*D$ is in good position with respect to $D, D_*$, hence each irreducible component $F$ of a face of $E+\Div f^*D$ is admissible with respect to $D, D_*$. If $F$ is contained in $|f^*D|$, then $D$ cannot be the leading pseudo-divisor for $F\to Y$, hence $F\to Y$ is in $\sM(Y)_{D_*}$.

For (4), the fact that $F_i\cap E$ is smooth and $F_1\cap E$ has codimension one on $F_2\cap E$ follows from the fact that $E+\Div f^*D$ is a simple normal crossing divisor. The inclusions $F_i\cap E\to Y$ are in $\sM(Y)_{D_*}$ by (3). If $D_{i_0}$ is the leading pseudo-divisor for $F_2\cap E\to Y$, then $E+ \cap_{i=0}^{i_0} \Div f^*D_i$ is a simple normal crossing divisor on $Y$ and a sub-divisor of $E+\Div f^*D$. It follows that the divisor
\[
 F_1\cap E+ \Div f^*D_{i_0}\cap F_2\cap E= F_1\cap E +  \cap_{i=0}^{i_0} \Div f^*D_i)\cap F_2\cap E
\]
 on $F_2\cap E$  is a simple normal crossing divisor. 

For (5), $\id_{F_2}$ is in $\sM(F_2)_{D, D_*}$, since $E$ is admissible for $D, D_*$ by remark~\ref{rem:GenPosition}(1). Suppose that $D_{i_0}$ is the leading pseudo-divisor for $F_2$ with respect to $D, D_*$. As above, $\Div f^*D_{i_0}\cap F_2$ is equal to $C\cap F_2$ for $C$ the subdivisor $\cap_{i=0}^{i_0} \Div f^*D_i$ of $\Div D$. $F_1\subset F_2$ is an irreducible component of $F_2\cap E_i$ for $E_i$ some irreducible component of $E$. As $E_i+C$ is a subdivisor of the simple normal crossing divisor $E+\Div D$, the intersection $F_2\cap(E_i+C)$ is a simple normal crossing divisor on $F_2$. As this intersection contains $F_1+F_2\cap \Div D_{i_0}$ as a subdivisor, it follows that $F_1+F_2\cap \Div D_{i_0}$ is also a simple normal crossing divisor on $F_2$, and thus $F_1$ is in good position on $F_2$ with respect to $D, D_*$. In case there is no leading pseudo-divisor for $\id_{F_2}$, then  $F_1$ is in good position on $F_2$ with respect to $D, D_*$, as $F_1$ is a smooth divisor on $F_2$. 

Finally, for (6), $C+S$ is a simple normal crossing divisor and $C$ and $S$ have no common components, hence  $C\cap S$ is a simple normal crossing divisor on $S$. Since $S$ is a disjoint union of faces of $E$, $S\to X$ is in $\sM(X)_{D, D_*}$ by remark~\ref{rem:GenPosition}(1). Let $S'$ be an irreducible component of $S$; we need to check that $S'\cap C$ is in good position with respect to $D, D_*$.  Suppose $D$ is the leading pseudo-divisor for $S'\to X$. Then $C+\Div f^*D+S'$ is a simple  normal crossing divisor on $Y$ and $S'$ has no components in common with $C+\Div f^*D$, so $S'\cap(C+\Div f^*D)$ is a simple normal crossing divisor on $S'$ and thus $C\cap S'$ is in good position on $S'$ with respect to $D, D_*$. If $D_{i_0}$ is the leading pseudo-divisor for $S'\to X$ for some $i_0\ge1$, then $S'\subset |f^*D_i|$ for $0\le i<i_0$ and $S'\cap \Div f^*D_{i_0}=S'\cap \cap_{i=0}^{i_0}\Div f^*D_i$.  Since $Y\to X$ is in $\sM(X)_{D, D_*}$, $B:= \cap_{i=0}^{i_0}\Div f^*D_i$ is a subdivisor of $\Div f^*D$ and thus as above, $S'\cap(C+B)$ is a simple normal crossing divisor on $S'$. Thus, $C\cap S'$ is in good position on $S'$ with respect to $D, D_*$ in this case as well. In case $S'\to X$ has no leading pseudo-divisor, then $C\cap S'$ is trivially in good position with respect to $D, D_*$. 
\end{proof}

\subsection{Refined algebraic cobordism}
We complete the definition of $\Omega_*(-)_{D_*}$.

\begin{definition}\label{Def:preOmegaD} Let $X$ be in $\Sch_k$ and let $D_1,\ldots, D_r$ be 
pseudo-divisors on $X$. Let  $\<\sR^{Sect}_*\>(X)_{D_*}$ be the subgroup of $\uu{\sZ}_*(X)_{D_*}$ generated by elements of the form
\[
[f:Y\to X,L_1,\ldots, L_m]-[f\circ i:Z\to X,i^*(L_1),\ldots, i^*(L_{m-1})],
\]
with $m>0$, $[f:Y\to X,L_1,\ldots, L_m]$ a cobordism cycle in $\sZ_*(X)_{D_*}$ and $i:Z\to Y$ the closed immersion of the smooth subscheme defined by the vanishing of a section $s:Y\to L_m$ transverse to the zero-section, such that $Z$ is in good position with respect to $D_*$.

We define $\uu{\Om}_*(X)_{D_*}$ by
\[
\uu{\Om}_*(X)_{D_*}:=\uu{\sZ}_*(X)_{D_*}/\<\sR^{Sect}_*\>(X)_{D_*}.
\]
\end{definition}

Let $0\le s\le r$  and let $D_*'$ be the sequence $D_1,\ldots, D_s$. The map $\res_{D_*/D'_*}:\uu{\sZ}_*(X)_{D_*}\to\uu{\sZ}_*(X)_{D'_*}$ descends to $\res_{D_*/D'_*}:\uu{\Om}_*(X)_{D_*}\to\uu{\Om}_*(X)_{D'_*}$. The operations $f^*$, $f_*$ and $\cn(L)$, as well as the external products  also descend to the quotient $\uu{\Om}_*(-)_{D_*}$ of $\uu{\sZ}_*(-)_{D_*}$ and the maps $\res_{D_*/D'_*}$ are natural with respect to the operations and products.

We have the universal formal group law $(F_\L, \L_*)$ with coefficient ring $\L_*$ the  Lazard ring. If $T_1, T_2:B\to B$ are commuting locally nilpotent operators on an abelian group $B$, and $F(u,v)=\sum_{i,j}a_{ij}u^iv^j$ is a power series with $\L_*$-coefficients, we have the well-defined $\L_*$-linear operator $F(T_1,T_2):\L_*\otimes B\to\L_*\otimes B$ defined by
\[
F(T_1,T_2)(a\otimes b):=\sum_{i,j}aa_{ij}\otimes (T^i_1\circ T^j_2)(b).
\]

\begin{definition}\label{Def:tildeOmegaD} For $X$ in $\Sch_k$, let $\<\sR_*^{FGL}\>(X)_{D_*}$ be the $\L_*$-sub\-module of $\L_*\otimes\uu{\Om}_*(X)_{D_*}$ generated by elements of the form
\[
(\id\otimes f_*)\left(F_\L(\cn(L),\cn(M))(\eta)-\cn(L\otimes M)(\eta)\right),
\]
where $f:Y\to X$ is in $\sM(X)_{D_*}$, $L$ and $M$ are line bundles on $Y$, and $\eta$ is in $\uu{\Om}_*(Y)_{f^*(D_*)}$. We set 
\[
\Om_*(X)_{D_*}:=\L_*\otimes\uu{\Om}_*(X)_{D_*}/\<\sR_*^{FGL}\>(X)_{D_*}.
\]
\end{definition}

For $0\le s\le r$ and $D'_*$ the sequence $D_1,\ldots, D_s$, the natural transformation $\res_{D_*/D'_*}:\uu{\Om}_*(X)_{D_*}\to \uu{\Om}_*(X)_{D'_*}$ descends to an $\L_*$-linear transformation $\res_{D_*/D'_*}:\Om_*(X)_{D_*}\to\Om_*(X)_{D'_*}$.  The operations we have defined for $\uu{\Om}_*(-)_{-}$: $f^*$, $f_*$, $\cn(L)$ and external products, all descend to $\L_*$-linear or bi-linear operations on $\Om_*(-)_{-}$, and the maps  $\res_{D_*/D'_*}$ are natural with respect to these operations. 

If $f:Y\to X$ is an $X$-scheme, and $D_1,\ldots, D_r$ are  pseudo-divisors on $X$, we will often write $\Om_*(Y)_{D_*}$ for $\Om_*(Y)_{f^*(D_*)}$, and similarly for $\uu{\Om}_*(Y)_{D_*}$.

\subsection{Refined divisor classes}\label{subsec:RefDivClass}

The operators
\[
\cn(L):\Om_*(X)_{D_*}\to\Om_{*-1}(X)_{D_*}
\]
are locally nilpotent and commute with one another, thus, if we have line bundles $L_1,\ldots, L_m$ on $X$, and a power series $F(u_1,\ldots, u_m)$ with $\L_*$-coef\-fi\-cients (of total degree $d$), we have the $\L_*$-linear endomorphism
\[
F(\cn(L_1),\ldots, \cn(L_m)):\Om_*(X)_{D_*}\to \Om_{*+d}(X)_{D_*}.
\]
This lifts to the level of  $\L_*\otimes \un{\sZ}_*(X)_{D_*}$, giving us the $\L_*$-linear endomorphism
\[
\un{F}(\cn(L_1),\ldots, \cn(L_m)):\L_*\otimes \un{\sZ}_*(X)_{D_*}\to \L_*\otimes \un{\sZ}_{*+d}(X)_{D_*}.
\]
If we have a morphism $f:X'\to X$ and an element $\eta\in \Om_*(X')_{D_*}$, we often write $F(\cn(L_1),\ldots, \cn(L_m))(\eta)$ for 
$F(\cn(f^*L_1),\ldots, \cn(f^*L_m))(\eta)$, when the context makes the meaning clear; we similarly write  $\un{F}(\cn(L_1),\ldots, \cn(L_m))(\un{\eta})$ for  $\un{F}(\cn(f^*L_1),\ldots, \cn(f^*L_m))(\un{\eta})$ for an $\un{\eta}\in :\L_*\otimes \un{\sZ}_*(X')_{D_*}$.

If $f:Y\to X$ is in $\sM(X)_{D_*}$, we have the element $1_Y^{D_*}=[\id:Y\to Y]\in\sM_*(Y)_{D_*}$. Given line bundles $L_1,\ldots, L_m$ on $Y$ we define 
\[
\un{[Y; F(L_1,\ldots, L_m)]}_{D_*}:=\un{F}(\cn(L_1),\ldots, \cn(L_m))(1_Y^{D_*})\in\L_*\otimes \un{\sZ}_*(Y)_{D_*}
\]
and
\[
[Y; F(L_1,\ldots, L_r)]_{D_*}:=F(\cn(L_1),\ldots, \cn(L_m))(1_Y^{D_*})\in\Om_*(Y)_{D_*}.
\]
As above, if $f:Y'\to Y$ is a morphism, we often write $\un{[Y'; F(L_1,\ldots, L_m)]}_{D_*}$ for $\un{[Y'; F(f^*L_1,\ldots, f^*L_m)]}_{D_*}$ and similarly write $[Y'; F(L_1,\ldots, L_r)]_{D_*}$ for $[Y'; F(f^*L_1,\ldots, f^*L_r)]_{D_*}$.

We recall some notation from \cite[\S3.1.1]{CobordismBook}. Let $n_1$,$\ldots$, $n_m$ be non-negative integers. We have the power series with $\L_*$-coefficients $F^{n_1,\ldots, n_m}$   giving the sum in the universal group law $(F_\L,\L_*)$:
\[
F^{n_1,\ldots, n_m}(u_1,\ldots,u_m)=n_1\cdot_Fu_1+_F\ldots+_Fn_m\cdot_Fu_m.
\]
 We have  the canonical decomposition
\[
F^{n_1,\ldots, n_m}(u_1,\ldots,u_m)=\sum_{J}u^JF_J^{n_1,\ldots, n_m}(u_1,\ldots, u_m).
\]
Here $J=(j_1,\ldots, j_m)\in\{0,1\}^m$,   $u^J:=u_1^{j_1}\cdots u_m^{j_m}$ and $F_J^{n_1,\ldots, n_m}$ is a power series with $\L$-coefficients that only involves the $u_i$ with $j_i=1$; this property and the above identity uniquely characterize the $F_J^{n_1,\ldots, n_m}$.

If $E=\sum_{i=1}^mn_iE_i$ is a simple normal crossing divisor on a scheme $Y\in\Sm_k$, with  support $|E|$ and irreducible components $E_1,\ldots, E_m$, we have for $J\in\{0,1\}^m$ the face $E^J:=\cap_{j_i=1}E_i$, inclusions $\iota^J:E^J\to |E|$, $i_J:E^J\to Y$, and line bundles $O_Y(E_i)^J:=i_J^*(O_Y(E_i))$ on $E^J$. We have defined  the divisor class $[E\to|E|]$ of $\Omega_*(|E|)$ by the formula (see \cite[definition 3.1.5]{CobordismBook})
\[
[E\to|E|]:= \sum_J\iota^J_*([E^J;F^{n_1,\ldots, n_m}_J(O_Y(E_1)^J,\ldots,O_Y(E_m)^J)]).
\]

Suppose now that we have $f:Y\to X$ in $\sM(X)_{D_*}$, and a simple normal crossing divisor $E$ on $Y$, such that $E$ is admissible with respect to $D_*$. Write $E=\sum_{i=1}^mn_iE_i$, with the $E_i$ irreducible, as above.

Since  $E$ is admissible with respect to $D_*$,  $\id_{E^J}$  is in $\sM(E^J)_{D_*}$ for each index $J$, so we have the class
\[
[E^J; F^{n_1,\ldots, n_m}_J(O_W(E_1)^J,\ldots, O_W(E_m)^J)]_{D_*}\in\Omega_*(E^J)_{D_*},
\]
giving  the {\em refined divisor class} 
\[
[E\to|E|]_{D_*}:=\sum_J\iota^J_*[E^J; F^{n_1,\ldots, n_m}_J(O_Y(E_1)^J,\ldots, O_Y(E_m)^J)]_{D_*}
\]
in $\Omega_*(|E|)_{D_*}$. In fact, the same formula gives us a well-defined class
\[
\un{[E\to|E|]}_{D_*}:=\sum_J\iota^J_*\un{[E^J; F^{n_1,\ldots, n_m}_J(O_Y(E_1)^J,\ldots, O_Y(E_m)^J)]}_{D_*}
\]
in $\L_*\otimes \un{\sZ}_*(|E|)_{D_*}$ lifting $[E\to|E|]_{D_*}$.

If we have a pseudo-divisor $\tilde E$ on $Y$ such that  $E:=\Div \tilde{E}$ is a Cartier divisor on $Y$, admissible with respect to $D_*$, we often write $[\tilde E\to |\tilde E|]_{D_*}$ for   $[E\to |E|]_{D_*}$; if $f:|E|\to X'$ is a projective morphism of  finite type $X$-schemes, we write $[E\to X']_{D_*}$ for $f_*([E\to |E|]_{D_*})$ and $\un{[E\to X']}_{D_*}$ for $f_*(\un{[E\to |E|]}_{D_*})$.  These divisor classes and those for a truncated sequence $D'_*:=D_1,\ldots, D_s$ are compatible via the forget maps $\res_{D_*/D_*'}$.

\begin{lemma}\label{lem:Additivity} Let $D_1,\ldots, D_r$ be pseudo-divisors on some $Y\in \Sm_k$ and let $E$ be a simple normal crossing divisor on $Y$, admissible with respect to $D_*$. Write $E=\sum_{i=1}^mn_iE_i$ with each $E_i$ smooth, but not necessarily irreducible. We assume that $E_i$ and $E_j$ have no common components if $i\neq j$. For each index $J=(j_1,\ldots, j_m)\in\{0,1\}^m$ let $E^J=\cap_{j_i=1}E_i$, with inclusions $\iota^J:E^J\to |E|$, $i_J:E^J\to Y$. Let $L_i=O_Y(E_i)$, $L_i^J=i_J^*L_i$. Then
\[
[E\to |E|]_{D_*}=\sum_J\iota^J_*([E^J; F^{n_1,\ldots, n_m}_J(L_1^J,\ldots, L_m^J)]_{D_*})]).
\]
in $\Omega_*(|E|)_{D_*}$. 
\end{lemma}

\begin{proof}  For each $i$, write $E_i=\amalg_{j=1}^{p_i}E_{ij}$, with $E_{ij}$ irreducible (and smooth). Then $E=\sum_{i=1}^m\sum_{j=1}^{p_i}n_iE_{ij}$. Let $n_i^{(p_i)}$ be the sequence $n_i,\ldots, n_i$, with $p_i$ terms.  For $J_*=(J_1,\ldots, J_m)\in \{0,1\}^{p_1}\times\ldots\times\{0,1\}^{p_m}$, $J_i=(j_{i1},\ldots, j_{ip_i})$, we have the corresponding face $E^{J_*}$ of $E$, 
\[
E^{J_*}=\cap_{i=1}^m\cap_{j_{ik}=1}E_{ik}.
\]
We let $u_{i*}=u_{i1},\ldots, u_{ip_i}$. The formal group law sum 
\begin{multline*}
F^{n_1^{(p_1)},\ldots, n_m^{(p_m)}}(u_{1*},\ldots,u_{m*})\\:=n_1\cdot_F(u_{11}+_F\ldots+_Fu_{1p_1})+_F\ldots+_Fn_m\cdot_F
(u_{m1}+_F\ldots+_Fu_{mp_m})
\end{multline*}
decomposes following our usual conventions as
\[
F^{n_1^{(p_1)},\ldots, n_m^{(p_m)}}(u_{1*},\ldots,u_{m*})=\sum_{J_*}u^{J_*}F_{J_*}^{n_1^{(p_1)},\ldots, n_m^{(p_m)}}(u_{1*},\ldots, u_{m_*}),
\]
where $u^{J_*}=\prod_{ij}u_{ik}^{j_{ik}}$.  We let $L_{ij}=O_Y(E_{ij})$, $L_{ij}^{J_*}$ the restriction of $L_{ij}$ to $E^{J_*}$, and write $L^{J_*}_{i*}$ for the sequence $L^{J_*}_{i1},\ldots,  L^{J_*}_{ip_i}$ and let $\iota^{J_*}:E^{J_*}\to |E|$ be the inclusion. With these notations,  the divisor class $[E\to |E|]_{D_*}$ is  
\[
[E\to |E|]_{D_*}=\sum_{J_*}\iota^{J_*}_*([E^{J_*}; F_{J_*}^{n_1^{(p_1)},\ldots, n_m^{(p_m)}}(L^{J_*}_{1*},\ldots, L^{J_*}_{m*})]_{D_*}).
\]

Fix an index $J=(j_1,\ldots, j_m)\in \{0,1\}^m$. The divisor $E_i$ thus contains $E^J$ if and only if $j_i=1$.

To each $J=(j_1,\ldots, j_m)$ as above, the closed subscheme $E^J$ breaks up as a disjoint union of certain faces $E^{J_*}$, namely,  exactly for those $J_*=(J_1,\ldots, J_m)$ such that, if $j_i=0$, then $J_i=(0,\ldots, 0)=0^{p_i}$, and if $j_i=1$, then the index $J_i\in\{0,1\}^{p_i}$ contains exactly one 1; if this 1 appears in the $q$th spot, we write this as $J_i=e_q^{(p_i)}$. Since $E_{ij}\cap E_{ij'}=\0$ for $j\neq j'$, the face $E^{J_*}$ is empty if one $J_i$ contains more than one 1, and thus $E^J$ is the disjoint union of the faces $E^{J_*}$ with $J_i=0^{p_i}$ if $j_i=0$ and $J_i=e_{q_i}^{(p_i)}$ for some $q_i$ with $1\le q_i\le p_i$ if $j_i=1$. For such a $J_*$, we set $q_i=0$ if $j_i=0$, set $e_0^{(p_i)}=0^{p_i}$, set $q_*=(q_1,\ldots, q_m)$,  and write the index $J_*=(e_{q_1}^{(p_1)},\ldots, e_{q_m}^{(p_m)})$ as $J_*=J(q_*)$. Let $S(J)\subset \prod_{i=1}^m\{0,\ldots, p_i\}$ be the subset consisting of those $q_*$ such that $q_i=0$ if and only if $j_i=0$. In this notation, we have 
\[
E^J=\amalg_{q_*\in S(J)}E^{J(q_*)}.
\]

Take $q_*\in S(J)$. We claim that for each $i$ such that $j_i=1$, 
\[
\cn(L_i)(1^{D_*}_{E^{J(q_*)}})=\cn(L_{iq_i})(1^{D_*}_{E^{J(q_*)}})
\]
in $\Omega_*(E^{J(q_*)})_{D_*}$.
Indeed,  $E_{ij}\cap E^{J(q_*)}=\0$ for all $j\neq q_i$. Applying the relations $\<\sR^{Sect}_*\>(E^{J(q_*)})_{D_*}$ for the empty divisor, we have $\cn(L_{ij})(1^{D_*}_{E^{J(q_*)}})=0$ for all $j\neq q_i$. Noting that $L_i=\otimes_{j=1}^{p_i}L_{ij}$ and applying the formal group relations $\<\sR^{FGL}_*\>(E^{J(q_*)})_{D_*}$  we have
\[
\cn(L_i)(1^{D_*}_{E^{J(q_*)}})=F^{1,\ldots, 1}(\cn(L_{i1}),\ldots, \cn(L_{1p_i}))(1^{D_*}_{E^{J(q_*)}})=
\cn(L_{iq_i})(1^{D_*}_{E^{J(q_*)}}).
\]
Since $F_J^{n_1,\ldots, n_m}(u_1,\ldots, u_m)$ does not involve $u_i$ if $j_i=0$, this gives the relation
\begin{multline*}
F_J^{n_1,\ldots, n_m}(\cn(L_1),\ldots, \cn(L_m))(1^{D_*}_{E^{J(q_*)}})\\=
F_{J(q_*)}^{n_{1*},\ldots, n_{m*}}(\cn(L_{1*}),\ldots, \cn(L_{m*}))(1^{D_*}_{E^{J(q_*)}})
\end{multline*}
in $\Omega_*(E^{J(q_*)})_{D_*}$. 
Therefore
\begin{align*}
\sum_J\iota^J_*(&[E^J; F^{n_1,\ldots, n_m}_J(L_1^J,\ldots, L_m^J)]_{D_*})\\&=
\sum_J\iota^J_*(F^{n_1,\ldots, n_m}_J(\cn(L_1),\ldots, \cn(L_m))(1^{D_*}_{E^J}))\\
&=\sum_J\sum_{q\in S(J)} \iota^{J(q_*)}_*(F^{n_1,\ldots, n_m}_J(\cn(L_1),\ldots, \cn(L_m))(1^{D_*}_{E^{J(q_*)}}))\\
&=\sum_J\sum_{q\in S(J)} \iota^{J(q_*)}_*(F^{n_{1*},\ldots, n_{m*}}_{J(q_*)}(\cn(L_{1*}),\ldots, \cn(L_{m*}))(1^{D_*}_{E^{J(q_*)}}))\\
&=\sum_{J_*} \iota^{J_*}_*(F^{n_{1*},\ldots, n_{m*}}_{J_*}(\cn(L_{1*}),\ldots, \cn(L_{m*}))(1^{D_*}_{E^{J_*}}))\\
&=[E\to |E|]_{D_*}.
\end{align*}
\end{proof}

\begin{lemma}\label{lem:DivClass}  
Let $E$ be a simple normal crossing divisor on some $Y\in\Sm_k$. Let $D_1,\ldots, D_r$  be pseudo-divisors on $Y$. Suppose that $\id_Y$ is in $\sM(Y)_{D_*}$ and that $E$ is in good position with respect to $D_*$.
Then
\[
[E\to Y]_{D_*}=[Y; O_Y(E)]_{D_*}
\]
in $\Omega_*(Y)_{D_*}$.
\end{lemma}

\begin{proof} Write $E=\sum_{i=1}^mn_iE_i$ with the $E_i$ irreducible and let , $i:|E|\to Y$ be the inclusion. For $J=(j_1,\ldots, j_m)\in \{0, 1\}^m$, let $E^J$ be the corresponding face of $E$ with inclusions $\iota^J:E^J\to |E|$, $i^J:E^J\to Y$. If $j_i=1$ for some $i$, let $J'=(j_1,\ldots, j_i-1,\ldots, j_m)$. Then by  lemma~\ref{lem:GoodPosition}(5),  the subscheme   $E^J$ of  $E^{J'}$ is a divisor in good position with respect to $D_*$.  Repeated applications of the relations $\<\sR^{Sect}_*\>(-)_{D_*}$ thus give us the relation
\[
i^J_*\left(F(\cn(L^J_1),\ldots, \cn(L^J_m)))(1_{E^J}^{D_*})\right)=(\cn(L_*)^JF(\cn(L^J_1),\ldots, \cn(L^J_m)))(1^{D_*}_Y)
\]
for arbitrary power series $F(u_1,\ldots, u_m)\in \L_*[[u_1,\ldots, u_m]]$ and line bundles $L_1,\ldots, L_m$ on $Y$, where $L_i^J:=i_J^*L_i$ and $\cn(L_*)^J:=\cn(L_1)^{j_1}\circ\ldots\circ\cn(L_m)^{j_m}$. Applying this to the definition of $[E\to |E|]_{D_*}$ gives us
\begin{align*}
[E\to &Y]_{D_*}\\&=i_*([E\to |E|]_{D_*})\\
&=i_*(\sum_{J}\iota^J_*(F^{n_1,\ldots, n_m}_J(\cn(O_Y(E_1)^J),\ldots, \cn(O_Y(E_m)^J))(1^{D_*}_{E^J}))\\
&=\sum_{J} (\cn(O_Y(E_*))^J\circ F^{n_1,\ldots, n_m}_J(\cn(O_Y(E_1)),\ldots, \cn(O_Y(E_m))))(1^{D_*}_Y) \\
&=F^{n_1,\ldots, n_m}(\cn(O_Y(E_1)),\ldots, \cn(O_Y(E_m)))(1^{D_*}_Y)\\
&=\cn(O_Y(E))(1^{D_*}_Y) =[Y; O_Y(E)]_{D_*}.
\end{align*}
\end{proof}

\begin{lemma}\label{lem:ChernProd} Take  $Y$ in $\Sm_k$ with pseudo-divisors  $D_1,\ldots, D_r$
on $Y$ such that $\id_Y$ is in $\sM(Y)_{D_*}$. Let $Z_1$, $Z_2$ be smooth disjoint divisors on
$Y$. \\[5pt]
(1) Assume that at least one of $Z_1, Z_2$ is in good position with respect to $D_*$.   Then
\[
\cn(O_Y(Z_1))\circ\cn(O_Y(Z_2))(1_Y^{D_*})=0
\]
in $\Omega_*(Y)_{D_*}$.\\[5pt]
(2) Let $D$ be an effective  Cartier divisor on $Y$,  let $C$ be an effective Cartier divisor on $Y$ with support contained in $|D|$ and let $i:|C|\to Y$ be the inclusion. Suppose that $\id_Y$ is in $\sM(Y)_{D, D_*}$ and that $Z_1+Z_2$  is in good position with respect to $D, D_*$. Then   $\cn(O_Y(Z_1))\circ\cn(O_Y(Z_2))(1_{C^J}^{D_*})=0$ in $\Omega_*(C^J)_{D_*}$  for each face $C^J$ of $C$, and
\[
\cn(i^*O_Y(Z_1))\circ\cn(i^*O_Y(Z_2))([C\to|C|]_{D_*})=0
\]
in $\Omega_*(|C|)_{D_*}$.
\end{lemma}

\begin{proof} We first prove (2), assuming (1). From the hypotheses in (2)  it follows that $D$ is the leading pseudo-divisor for $\id_Y$ with respect to $D, D_*$. By  lemma~\ref{lem:AdmissibleIntersection}, $C$ is a simple normal crossing divisor on $Y$, admissible with respect to $D_*$; in particular, the divisor class $[C\to|C|]_{D_*}$ is defined. 

Write $C=\sum_{j=1}^mn_iC_j$ with each $C_i$ irreducible. For each $J\in\{0,1\}^m$, write $C^J$ as a disjoint union of
irreducible components, $C^J=\amalg_jC^J_j$, and let $\iota^J_j:C^J_j\to|C|$, $i^J_j:C^J_j\to Y$ be the inclusions. Then each inclusion $C^J_j\to Y$ is in $\sM(Y)_{D_*}$, so the unit elements $1_{C^J_j}^{D_*}\in\sM(C_j^J)_{D_*}$ are all defined.

Let   $\eta_{J,j}=\cn(i^{J*}_jO_Y(Z_1))\circ\cn(i^{J*}_jO_Y(Z_2))(1_{C^J_j}^{D_*})$.
In $\Omega_*(|C|)_{D_*}$ we have the identity
\begin{multline*}
\cn(i^*O_Y(Z_1))\circ\cn(i^*O_Y(Z_2))([C\to|C|]_{D_*})\\=
\sum_{J,j}F^{n_1,\ldots,n_m}_J(\cn(i^*O_X(C_1)),\ldots, \cn(i^*O_Y(C_m)))(\iota^J_{j*}(\eta_{J,j})).
\end{multline*}
If $C^J_j$ is not contained in $Z_1\cup Z_2$, then $Z_1\cap C^J_j$ and $Z_2\cap C^J_j$ are smooth disjoint divisors on $C^J_j$; by lemma~\ref{lem:GoodPosition}(2), these are both in good position with respect to $D_*$.  Thus (1) (for $Y=C^J_j$) implies that  $\eta_{J,j}=0$.

If $C^J_j$ is contained in say $Z_1$, then $C^J_j\cap Z_2=\0$ and thus $i^{J*}_jO_Y(Z_2)\cong O_{C^J_j}$. Using the relations $\<\sR^{Sect}_*\>(C^J_j)_{D_*}$ in the case of an empty divisor, we see that $\cn(i^{J*}_jO_Y(Z_2))(1_{C^J_j}^{D_*})=0$,  so $\eta_{J,j}=0$  in this case as well, and (2) follows.

For (1), let $i_j:Z_j\to Y$ be the inclusion. Since $\cn(O_Y(Z_1))$ and $\cn(O_Y(Z_2))$ commute, we may assume that $Z_2$ is in good position with respect to $D_*$. Using the relations $\<\sR^{Sect}_*\>(Y)_{D_*}$ gives
\[
\cn(O_Y(Z_1))\circ\cn(O_Y(Z_2))(1_{Y}^{D_*})=i_{Z_2*}\left(\cn(i_2^*O_X(Z_1))(1_{Z_2}^{D_*})\right)
\]
in $\Omega_*(Y)_{D_*}$. Using  $\<\sR^{Sect}_*\>(Z_2)_{D_*}$, again in the case of an empty
divisor,  gives $\cn(i_2^*O_X(Z_1))(1_{Z_2}^{D_*})= 0$ in $\Omega_*(Z_2)_{D_*}$.
\end{proof}

\begin{lemma}\label{lem:StrongDim} Let $f:X\to Z$ be a morphism in $\Sch_k$ with $Z$ in $\Sm_k$, and let $L_1,\ldots, L_m$ be line bundles on $Z$ with $m>\dim_kZ$. Let $D_1,\ldots, D_r$ be  pseudo-divisors on $X$. Then the operator $\cn(f^*L_1)\circ\ldots\circ\cn(f^*L_m)$ vanishes on $\Omega_*(X)_{D_*}$.
\end{lemma}

\begin{proof} We proceed by induction on $\dim_kZ$. Since the operators $\cn(L)$ are $\L_*$-linear and commute with each other, it suffices to show that the operator in question vanishes on elements $g:Y\to X$ of $\sM(X)_{D_*}$. The identity 
\[
\cn(f^*L_1)\circ \ldots\circ\cn(f^*L_m)(g:Y\to X)=g_*((\id_Y,(fg)^*L_1,\ldots, (fg)^*L_m))
\]
reduces us to the case $X=Y$,  $g=\id_Y$, that is, it suffices to show that
\[
(\id_Y,f^*L_1,\ldots, f^*L_m)=0\text{ in }\Omega_*(Y)_{D_*},
\]
assuming $\id_Y$ is in $\sM(Y)_{D_*}$ and $m>\dim_kZ$.

We may assume that $Y$ is irreducible. As before, we may assume that $|D_i|\neq Y$ for all $i$.  Using the formal group law, we reduce to the case of very ample line bundles $L_i$ (see for example the proof of \cite[lemma 3.2.6]{CobordismBook}).

Assume that $r\ge 1$. Then $D_1$ is the leading pseudo-divisor for $\id_Y$ and by lemma~\ref{lem:AdmissibleIntersection}, $\Div D_1$ is a simple normal crossing divisor on $Y$.

If $\dim_kZ=0$, all the line bundles are trivial, hence have a nowhere vanishing section. We may then use the relations in $\<\sR_*^{Sect}\>(Y)_{D_*}$ (for the empty divisor) to conclude that $(\id_Y,f^*L_1,\ldots, f^*L_m)=0$.

Suppose now that $\dim_kZ>0$. Let $s$ be a section of $L_m$. By Bertini's theorem, we may choose $s$ so that $s=0$ is a smooth divisor $\bar{i}:\bar{Z}\to Z$ on $Z$. Let $H$ be the divisor of $f^*s$ with inclusion $i:|H|\to Y$, and let $\bar{f}:|H|\to \bar{Z}$ be the induced morphism.

By Bertini's theorem again, we may choose $s$ so  that $H+D_1$ is a simple normal crossing divisor on $Y$. In other words, $H$ is in good position with respect to $D_*$. 
 
By lemma~\ref{lem:DivClass}, we have the identity in $\Omega_*(Y)_{D_*}$
\[
[H\to Y]_{D_*}=[Y; O_Y(H)]_{D_*}=[Y; f^*L_m]_{D_*}.
\]
Thus
\begin{align*}
(\id_Y,f^*L_1,&\ldots, f^*L_m)\\
&=
\cn(f^*L_1)\circ\ldots\circ\cn(f^*L_{m-1})([H\to Y]_{D_*})\\
&=i_*(\cn(\bar{f}^*(\bar{i}^*L_1))\circ\ldots\circ\cn(\bar{f}^*(\bar{i}^*L_{m-1}))
([H\to|H|]_{D_*})).
\end{align*}
As this last element is zero by our induction hypothesis, the lemma is proved in case $r\ge1$. If $r=0$, the same proof works, except that we need only assume that $\bar{Z}$ and $H$ are smooth.
\end{proof}

\section{Intersection with a pseudo-divisor} \label{sec:Intersection}
\subsection{The intersection map} We construct the intersection map and derive some of its basic properties.

Let $D, E_1,\ldots, E_r$ be pseudo-divisors on $X$.  Let $f:Y\to X$ be in $\sM(X)_{D, E_*}$ with $Y$ irreducible, and consider a cobordism cycle 
\[
\eta:=(f:Y\to X, L_1,\ldots, L_m)\in \sZ_*(X)_{D, E_*}. 
\]
Let $C$ be a pseudo-divisor  on $X$ with    $C$ supported in $D$. In particular, if $f(Y)\not\subset |D|$, then $\Div f^*C$ is an effective  Cartier divisor  with support contained in $|\Div f^*D|$. 

We define the element $\un{C}(\eta)_{E_*}\in\L_*\otimes\un{\sZ}_*(|D|)_{E_*}$ as follows: Suppose $f(Y)\subset |D|$. Let $f^D:Y\to |D|$ be the morphism induced by $f$; in this case $\sZ_*(Y)_{D, E_*}=\sZ_*(Y)_{E_*}$. Let $\eta_Y:=(\id_Y, L_1,\ldots, L_m)\in \sZ_*(Y)_{E_*}$.   We define
\[
\un{C}(\eta)_{E_*}:=1\otimes f_*^D(\cn(f^*\sO_X(C))(\eta_Y))\in\L_*\otimes\un{\sZ}_{*-1}(|D|)_{E_*}.
\]

If $f(Y)\not\subset |D|$, then $D$ is the leading pseudo-divisor for $f$, hence $\Div f^*D$ is a simple normal crossing divisor on $Y$ and thus $\tilde{C}:=\Div f^*C$ is also a simple normal crossing divisor on $Y$, with $|\tilde C|\subset |f^*D|$. By lemma~\ref{lem:AdmissibleIntersection}, $\Div D$ is admissible with respect to $E_*$. Since each face of $\tilde{C}$ is a face of $\Div f^*D$,  $\tilde{C}$ is also admissible and the divisor class $\un{[\tilde C\to|f^*D|]}_{E_*}\in \L_*\otimes\un{\sZ}_*(|f^*D|)_{E_*}$ is defined. We let $f^D:|f^*D|\to |D|$ be the restriction of $f$, $L_i^D$ the restriction of $L_i$ to $|f^*D|$, and define
\[
\un{C}(\eta)_{E_*}:=f^D_*(\cn(L_1^D)\circ\ldots\circ\cn(L_m^D)(\un{[\tilde C\to|f^*D|]}_{E_*}))\in\L_*\otimes\un{\sZ}_{*-1}(|D|).
\]
We extend this operation to a homomorphism  $\un{C}(-)_{E_*}:\L_*\otimes\sZ_*(X)_{D, E_*}\to\L_*\otimes\un{\sZ}_{*-1}(|D|)_{E_*}$ by $\L_*$-linearity.

\begin{definition} Let $D, E_1,\ldots, E_r$ be pseudo-divisors on $X$ and  let $C$ be a  pseudo-divisor on $X$ supported in  $D$.  Let $\can: \L_*\otimes\un{\sZ}_{*-1}(|D|)_{D, E_*}\to \Omega_*(|D|)_{D, E_*}$ be the canonical surjection. The homomorphism
\[
C(-)_{E_*}:\L_*\otimes\sZ_*(X)_{D, E_*}\to\Omega_{*-1}(|D|)_{E_*}
\] 
is defined to be the composition
\[
\L_*\otimes\sZ_*(X)_{D, E_*}\xrightarrow{\un{C}(-)_{E_*}} [\L_*\otimes\un{\sZ}_*(|D|)_{E_*}]_{*-1}\xrightarrow{\can} \Omega_{*-1}(|D|)_{E_*}.
\]

We will sometimes drop the subscript $E_*$, writing $C(-)$ for  $C(-)_{E_*}$ and $\un{C}(-)$ for  $\un{C}(-)_{E_*}$, if the context carries the meaning.  We will also often ignore the shift by -1 in the grading.
\end{definition}

The next two results follow directly from the definitions:

\begin{lemma}\label{lem:ProjectionFormula} Let $X$ be a finite type $k$-scheme with pseudo-divisors $D$, $E_1,\ldots, E_r$ and let $C$ be a pseudo-divisor on $X$ supported in $D$. Let  $g:X'\to X$ be a morphism of finite type and let $g_D:|g^*D|\to |D|$ be the restriction of $g$.
\\
(1) Suppose that $g$ is projective. Let   $\eta$ be in $\sZ_*(X')_{D, E_*}$.  Then $g_*\eta$ is in $\sZ_*(X)_{D, E_*}$, and
\[
g_{D*}(\un{g^*C}(\eta)_{E_*})=\un{C}(g_*\eta)_{E_*}
\]
in $\L_*\otimes\un{\sZ}_*(|D|)_{E_*}$.\\
(2) Suppose that $g$ is smooth and quasi-projective. Let $\eta$ be in $\sZ_*(X)_{D, E_*}$. Then $g^*\eta$ is in $\sZ_*(X')_{D, E_*}$, $g_D$ is smooth and quasi-projective, and
\[
g_D^*(\un{C}(\eta)_{E_*})=\un{g^*C}(g^*\eta)_{E_*}
\]
in $\L_*\otimes\un{\sZ}_*(|g^*D|)_{E_*}$.
\end{lemma}

\begin{lemma}\label{lem:ChernClassComp}  Let $X$ be a finite type $k$-scheme, with pseudo-divisors $D$, $E_1,\ldots, E_r$ and let $C$ be a pseudo-divisor supported in $D$. Let  $L$ be a line bundle on $X$, let $L^D$ be the restriction of $L$ to $|D|$ and take $\eta$ in $\sZ_*(X)_{D, E_*}$.  Then
\[
\cn(L^D)(\un{C}(\eta)_{E_*})=\un{C}(\cn(L)(\eta))_{E_*}
\]
in $\L_*\otimes\un{\sZ}_*(|D|)_{E_*}$.
\end{lemma}

\begin{lemma}\label{lem:WeakDim}  Let $f:Y\to X$ a morphism in $\sM(X)_{D, E_*}$, let $C$ be a pseudo-divisor on $X$ supported in $|D|$ and let $L_1,\ldots, L_m$ be line bundles on $Y$ with $m\ge\dim_kY$. Then 
\[
\un{C}((Y\to X,L_1,\ldots, L_m))_{E_*}=0
\]
 in $\L_*\otimes\un{\sZ}_*(|D|)_{E_*}$.
\end{lemma}

\begin{proof} We may suppose $Y$ to be irreducible. Write $\tilde{C}$ for $f^*C$, and let $f^D:|f^*D|\to |D|$ be the restriction of $f$. Using lemmas~\ref{lem:ProjectionFormula} and \ref{lem:ChernClassComp}, we have
\[
\un{C}((f:Y\to X,L_1,\ldots, L_m))_{E_*}=f^D_*\big(\cn(L_1)\circ\ldots\circ\cn(L_m)(\tilde{C}(1^{D, E_*}_Y)_{E_*})\big).
\]
If $f(Y)\subset |D|$, then  $\tilde{C}(1^{D, E_*}_Y)=\cn(O_Y(\tilde{C}))(1^{D, E_*}_Y)$, and thus
\[
\cn(L_1)\circ\ldots\circ\cn(L_m)(\tilde{C}(1^{D, E_*}_Y)_{E_*})=(\id_Y,L_1,\ldots, L_m, O_Y(\tilde{C}))=0
\]
in $\un{\sZ}_{*-1}(Y)_{E_*}$, using the relations $\<\sR_*^{Dim}\>(Y)_{E_*}$. If $f(Y)\not\subset|D|$, then $\tilde{C}(1^{D, E_*}_Y)$ is a sum of terms of the form 
\[
a\cdot\iota^J_*((\tilde{C}^J,M_1,\ldots, M_s)),
\]
with $a\in\L_*$, $\iota^J:\tilde{C}^J\to Y$ the inclusion of a face of $\tilde{C}$, and the $M_i$ line bundles on $C^J$. Thus $\cn(L_1)\circ\ldots\circ\cn(L_m)(\tilde{C}(1^{D, E_*}_Y))$ is a sum of terms of the form
\[
a\cdot\iota^J_*((\tilde{C}^J,M_1,\ldots, M_s,\iota^{J*}L_1,\ldots,\iota^{J*}L_m)), \ a\in \L_*.
\]
Since  $\dim_k\tilde{C}^J<\dim_kY$ for each face $\tilde{C}^J$, the terms
\[
(\tilde{C}^J,M_1,\ldots, M_s, \iota^{J*}L_1,\ldots,\iota^{J*}L_m)
\]
vanish in $\un{\sZ}_*(\tilde{C}^J)_{E_*}$ (using the relations $\<\sR_*^{Dim}\>(\tilde{C}^J)_{E_*}$), whence the result.
\end{proof}

Let $F(u_1,\ldots, u_m)$ be a power series with $\L_*$-coefficients,  let $L_1,\ldots, L_m$ be line bundles on $X$, and let $f:Y\to X$ be in $\sM(X)_{D, E_*}$. Let $F_N$ denote the truncation of $F$ after total degree $N$. By lemma~\ref{lem:WeakDim}, we have, for all $N\ge\dim_kY$ and all $n\ge0$,
\[
\un{C}(F_N(\cn(L_1),\ldots, \cn(L_m))([f]))_{E_*}=\un{C}(F_{N+n}(\cn(L_1),\ldots, \cn(L_m))([f]))_{E_*}.
\]
Thus, for $\eta\in\L_*\otimes\sZ_*(X)_{D, E_*}$, we may set
\[
\un{C}(F(\cn(L_1),\ldots, \cn(L_m))(\eta))_{E_*}:=\lim_{N\to\infty}\un{C}(F_N(\cn(L_1),\ldots, \cn(L_m))(\eta))_{E_*},
\]
as the terms in the limit are eventually constant in $N$.

With this definition,   lemma~\ref{lem:ChernClassComp} extends to power series in the Chern class operators.

\begin{lemma}\label{lem:ChernClassComp2} Let $D, E_1,\ldots, E_r$ be pseudo-divisors on $X$ and take $\eta$ in
$\L_*\otimes\sZ_*(X)_{D, E_*}$. Let $f^D:|f^*D|\to |D|$ be the restriction of $f$ and let $i:|f^*D|\to Y$ be the inclusion. Let $C$ be a pseudo-divisor on $X$ supported in $D$, let $F(u_1,\ldots, u_m)$ be a power series with $\L_*$-coefficients, and let $L_1,\ldots, L_m$ be line bundles on $X$. Then
\begin{multline*}
\un{C}(f_*(\un{F}(\cn(L_1),\ldots, \cn(L_m))(\eta))_{E_*}\\
=f^D_*\left(\un{F}(\cn(i^*L_1),\ldots,\cn(i^*L_m))(\un{f^*C}(\eta)_{E_*})\right)
\end{multline*}
in $\L_*\otimes\un{\sZ}_*(D)_{E_*}$.
\end{lemma}

\begin{lemma}\label{lem:DivClassCommutativity} Let $D, E_1,\ldots, E_r$ be pseudo-divisors on $X$, let $f:Y\to X$ be in $\sM(X)_{D, E_*}$ with $Y$ irreducible, and let $Z\to Y$ be a smooth codimension one closed subscheme of $Y$. We suppose that $Z$  is in good position with respect to $D, E_*$ and that $|f^*D|\neq Y$.   Let $C$ be an effective Cartier divisor on $Y$   supported in  $|f^*D|$.  \\
(1)   Suppose that no component of  $Z$ is contained in  $|C|$ and let  $i_Z:Z\to Y$,  $i_{CZ}:|i_Z^*C|\to |C|$, $i_C:|C|\to Y$ be the inclusions. Then 
\[
i_{CZ*}([i^*_ZC\to |i^*_ZC|]_{E_*})=\cn(i_C^*O_Y(Z))([C\to|C|]_{E_*})
\]
in $\Omega_*(|C|)_{E_*}$.\\
(2) Suppose that no component of $Z$ is contained in $|f^*D|$, and let $i:|f^*D|\to Y$ be the inclusion. Then
\[
C([Z\to Y]_{D, E_*})_{E_*}=\cn(i^*O_Y(Z))([C\to|f^*D|]_{E_*})
\]
in $\Omega_*(|f^*D|)_{E_*}$.
\end{lemma}

\begin{proof} We first check that all the terms in (1) and (2) are defined. By assumption, $f$ is in $\sM(X)_{D, E_*}$ and $Y\not\subset |f^*D|$, so $D$ is the leading pseudo-divisor for $f$ with respect to $D, E_*$ and $\Div f^*D$ is therefore a simple normal crossing divisor on $Y$.  Since $Z$ is in good position with respect to  $D, E_*$,  $Z$ is admissible with respect to $D, E_*$ by remark~\ref{rem:GenPosition}(1). The divisor $\Div f^*D$ is admissible with respect to $E_*$ by  lemma~\ref{lem:AdmissibleIntersection} and thus $C$ is also admissible with respect to $E_*$. By lemma~\ref{lem:GoodPosition}(2) each face of $C+Z$ contained in $D$ is admissible with respect to $E_*$ and thus the simple normal crossing divisor $i_Z^*C$ on $Z$ is admissible with respect to $E_*$.
 
Next, we note that (1) implies (2). Indeed, assuming that no component of $Z$ is contained in $|f^*D|$, we have
\[
C([Z\to Y]_{D, E_*})_{E_*}=(i_{DC*}\circ i_{CZ*})([i^*_ZC\to |i_Z^*C|]_{E_*}), 
\]
where $i_{DC}:|C|\to |f^*D|$ is the inclusion. Thus,  (2) follows by applying $i_{DC*}$ to the identity in (1). We now prove (1).

Write $C=\sum_{i=1}^mn_i C_i$, with each $C_i$ irreducible. The divisor class $[C\to|C|]_{E_*}$ is a sum over the faces $C^J$ of $C$,
\[
[C\to|C|]_{E_*}=\sum_J\iota^J_*([C^J; F_J^{n_1,\ldots,n_m}(L_1^J,\ldots,L_m^J)]_{E_*}),
\]
where $\iota^J:C^J\to |C|$ is the inclusion,  $L_i=O_W(C_i)$ and $L_i^J$ is the restriction of $L_i$ to $C^J$.

Since no component of $Z$ is contained in $|C|$, it follows that the intersection $C_Z^J:=Z\cap C^J$ is transverse, so $C_Z^J$ is a smooth codimension one closed subscheme of $C^J$.   It follows from lemma~\ref{lem:GoodPosition}(2) that $C_Z^J\subset C^J$ is in good position with respect to $E_*$. 

Thus, the relations in $\<\sR^{Sect}_*\>(|C|)_{E_*}$ imply that
\[
\cn(i^*O_Y(Z))([C\to|C|]_{E_*})=\sum_Ji^{ZJ}_*([C_Z^J; F_J^{n_1,\ldots,n_m}(L_1^{ZJ},\ldots, L_m^{ZJ})]_{E_*}),
\]
where $i^{ZJ}:C_Z^J\to |C|$ is the inclusion, and $L_i^{ZJ}$ is the restriction of $L_i$ to $C_Z^J$. By lemma~\ref{lem:Additivity}
\[
[i^*_ZC\to |i^*_ZC|]_{E_*}=\sum_J\iota^{ZJ}_*([C_Z^J; F_J^{n_1,\ldots,n_m}(L_1^{ZJ},\ldots, L_m^{ZJ})]_{E_*}),
\]
where $\iota^{ZJ}_*:C_Z^J\to |i_Z^*C|=|C|\cap Z$ is the inclusion, and since $i^{ZJ}_*=i_{CZ*}\circ \iota^{ZJ}_*$, (1) follows.
\end{proof}

We will need a strengthening of lemma~\ref{lem:DivClassCommutativity}.  

\begin{lemma}\label{lem:IntersectionFormula} Let $Y$ be in $\Sm_k$ and
irreducible. Let $i_Z:Z\to Y$ be a  codimension one closed subscheme, smooth over $k$. Let $D, E_1,\ldots, E_r$ be pseudo-divisors on $Y$ such that $\id_Y$ is in $\sM(Y)_{D, E_*}$ and  $Z$ is in good position with respect to $D, E_*$. Let $C$ be a pseudo-divisor  on $Y$ supported in $D$. We suppose that $|D|\neq Y$ and let $i_D:|D|to Y$ be the inclusion. Then
\begin{equation}\label{eqn:IntersectionFormula1}
C([Z\to Y]_{D,E_*})_{E_*}=\cn(i_D^*O_Y(Z))([C\to|D|]_{E_*})
\end{equation}
in $\Omega_*(|D|)_{E_*}$.
\end{lemma}

\begin{proof} The condition that $Z$ is in good position with respect to $D, E_*$ implies by remark~\ref{rem:GenPosition}(1) that $i_Z:Z\to Y$ is in $\sM(Y)_{D, E_*}$ and thus $C([Z\to Y]_{D,E_*})_{E_*}$ is  defined. Similarly, the simple normal crossing divisor $\Div D$ on $Y$ is admissible with respect to $E_*$ by lemma~\ref{lem:AdmissibleIntersection}, hence  $\Div C$   is admissible and  $[C\to|D|]_{E_*}$ is also  defined. 

We first reduce to the case of irreducible $Z$. Write $Z=\sum_{i=1}^mZ_i$ with $Z_i$ irreducible. Then each $Z_i$ is in good position with respect to $D, E_*$.   By lemma~\ref{lem:ChernProd}(2)  
\[
\cn(i_D^*O_Y(Z_i)\circ\cn(i_D^*O_Y(Z_j))([C\to|D|]_{E_*})=0
\]
for all $i\neq j$ and therefore, using the relations  $\<\sR_*^{FGL}\>(|D|)_{E_*}$, we have
\[
\cn(i_D^*O_Y(Z))([C\to|D|]_{E_*})=\sum_{i=1}^r\cn(i_D^*O_Y(Z_i))([C\to|D|]_{E_*}).
\]
As $[Z\to Y]_{D,E_*}=\sum_{i=1}^m[Z_i\to Y]_{D,E_*}$, it suffices to handle the case of irreducible $Z$.

 In case $Z$ is not contained in  $|D|$ the result follows from lemma~\ref{lem:DivClassCommutativity}, so we may suppose that $Z\subset |D|$. 
 As $Z$ is then a component of $\Div D$, $Z$ is admissible with respect to $E_*$ and with respect to $D, E_*$. 
 
We first consider the case in which $Z$ is not a component of $|\Div C|$.  By lemma~\ref{lem:GoodPosition}(2), the simple normal crossing divisor $i_Z^*\Div C$ on $Y$ is in good position with respect to $E_*$, and thus by lemma~\ref{lem:DivClass}, we have the identity
 \begin{equation}\label{eqn:Chern3}
 [\Div i_Z^*C\to Z]_{E_*}=\cn(i_Z^*O_Y(C))(1_Z^{E_*})
 \end{equation}
 in $\Omega_*(Z)_{E_*}$.  Let $\iota_Z:Z\to |D|$ be the inclusion. Applying $\iota_Z$ to the identity \eqref{eqn:Chern3} and using Lemma~\ref{lem:DivClassCommutativity}(1) gives us the identities
 \[
\iota_{Z*}(\cn(i_Z^*O_Y(C))(1_Z^{E_*}))=\iota_{Z*}([\Div i_Z^*C\to Z]_{E_*})=\cn(i_D^*O_Y(Z))([C\to|D|]_{E_*}).
\]
As $Z\subset |D|$, we have 
 \[
 C([Z\to Y]_{D, E_*})_{E_*}=\iota_{Z*}(\cn(i_Z^*O_Y(C))(1_Z^{E_*})), 
 \]
 which yields the result in this case.

We now assume that $Z$ is a component of $\Div C$.  Write $\Div C=\sum_{i=1}^mn_iC_i$, with the $C_i$ irreducible and with $Z=C_1$. For a face $C^J$ of $\Div C$, we have the following diagram of inclusions
\[
\xymatrixrowsep{8pt}
\xymatrix{
C^J\cap Z\ar[rr]^-{\tau^J}\ar[dd]_{\eta^J}&&C^J\ar[dl]^(.4){\iota^J}\ar[dd]^{i^J}\\
&|D|\ar[dr]^(.45){i_D}&\\
Z\ar[rr]_{i_Z}\ar[ru]^{\iota_Z}&&Y.
}
\]
Since $Z$ is in good position with respect to $D, E_*$, it follows from lemma~\ref{lem:GoodPosition} that the $Y$-schemes $C^J$, $C^J\cap Z$ are all in $\sM(Y)_{E_*}$, and if $C^J$ is not contained in $Z$, then $C^J\cap Z\subset C^J$ is in good position with respect to $E_*$. 
 
Since
\begin{align*}
n_1\cdot_Fu_1+_F+\ldots+_Fn_m\cdot_Fu_m&=F^{n_1,\ldots,n_m}(u_1,\ldots, u_m)\\
&=\sum_Ju^JF^{n_1,\ldots,n_m}_J(u_1,\ldots,u_m),
\end{align*}
we have
\[
\cn(O_Y(C))=\sum_J\cn(O_Y(C_*))^JF_J(\cn(O_Y(C_1)),\ldots, \cn(O_Y(C_m))),
\]
where $F_J:=F_J^{n_1,\ldots, n_m}$ and for $J=(j_1,\ldots, j_m)\in\{0,1\}^m$
\[
\cn(O_Y(C_*))^J=\cn(O_Y(C_1))^{j_1}\circ\ldots\circ\cn(O_Y(C_m))^{j_m}.
\]
Since we therefore have
\begin{align*}
C(&[Z\to Y])_{E_*}=\cn(i_D^*O_Y(C))([Z\to|D|]_{E_*})\\
&=\sum_J\iota_{Y*}\big(\cn(O_Y(C_*))^J([Z;F_J(O_Y(C_1),\ldots, O_Y(C_m))]_{E_*})\big)
\end{align*}
and
\begin{align*}
\cn(&O_Y(Z))([C\to|D|]_{E_*})\\
&=\sum_J\iota^J_*\big(\cn(O_Y(Z))([C^J;F_J(O_Y(C_1),\ldots, O_Y(C_m))]_{E_*})\big), 
\end{align*}
it suffices to prove that
\begin{multline}\label{multline:eqn1}
\iota^J_*\big(\cn(O_Y(Z))([C^J;F_J(O_Y(C_1),\ldots, O_Y(C_m))]_{E_*})\big)\\
=\iota_{Z*}\big(\cn(O_Y(C_*))^J([Z;F_J(O_Y(C_1),\ldots, O_Y(C_m))]_{E_*})\big)
\end{multline}
in $\Omega_*(|D|)_{E_*}$, for each index $J$. 

Suppose that $J=(0,j_2,\ldots, j_m)$, so no component of $C^J$ is contained in $Z$. Letting  $J_i=(0,1,\ldots, 1,0,\ldots,0)$, with $i$ 1's, we may assume that $J=J_s$ for some $s\ge1$. We have the sequence of closed subschemes
\[
Z\cap C^{J_s}\subset Z\cap C^{J_{s-1}}\subset \ldots\subset Z\cap C^{J_1}\subset Z,
\]
with $Z\cap C^{J_i}$ a smooth divisor on  $Z\cap C^{J_{i-1}}$, in good position with respect to  $E_*$ for each $i=1,\ldots, s$ (lemma~\ref{lem:GoodPosition}(4)).  Applying the relations $\<\sR^{Sect}_*(-)\>_{E_*}$ repeatedly,  we see that
\begin{multline*}
\cn(O_Y(C_*))^J\big([Z;F_J(O_Y(C_1),\ldots, O_Y(C_m))]_{E_*}\big)\\
=\eta^J_*[Z\cap C^J;F_J(O_Y(C_1),\ldots, O_Y(C_m))]_{E_*}.
\end{multline*}
Applying the same relations to the smooth divisor $Z\cap C^J$ on $C^J$, which by lemma~\ref{lem:GoodPosition}(2) is in good position with respect to $E_*$, we have
\begin{multline*}
\cn(O_Y(Z))\big([C^J;F_J(O_Y(C_1),\ldots, O_Y(C_m))]_{E_*}\big)\\
=\tau^J_*[Z\cap C^J;F_J(O_Y(C_1),\ldots, O_Y(C_m))]_{E_*}.
\end{multline*}
Since $\iota^J_*\circ \tau^J_*=\iota_{Y*}\circ\eta^J_*$, these two identities yield the equality \eqref{multline:eqn1} in this case. 

In case $J=(1,j_2,\ldots, j_m)$, then  $C^J$ is contained in $Z=C_1$. Letting $J'=(0,j_2,\ldots, j_m)$, we have $C^J=Z\cap C^{J'}$, and no component of $C^{J'}$ is contained in $Z$. Suppose that $J'\neq(0,\ldots,0)$.  As above, we have
\begin{multline*}
\iota^{J'}_*\big(\cn(O_Y(Z))([C^{J'};F_J(O_Y(C_1),\ldots, O_Y(C_m))]_{E_*})\big)\\=
\iota_{Z*}\big(\cn(O_Y(C_*))^{J'}([Z;F_J(O_Y(C_1),\ldots, O_Y(C_m))]_{E_*})\big)
\end{multline*}
in $\Omega_*(|D|)_{E_*}$. Also, the smooth divisor $C^J=Z\cap C^{J'}$ on $C^{J'}$ is in good position with respect to $E_*$, and
\begin{multline*}
\cn(O_Y(Z))([C^{J'};F_J(O_Y(C_1),\ldots, O_Y(C_m))]_{E_*})\\=
\tau^{J'}_*[C^J;F_J(O_Y(C_1),\ldots, O_Y(C_m))]_{E_*}
\end{multline*}
in $\Omega_*(C^{J'})_{E_*}$. As $\iota^J=\iota^J\circ \tau^{J'}$, this yields
\begin{align*}
\iota^{J}_*\big(\cn(O_Y(Z))(&[C^J;F_J(O_Y(C_1),\ldots)]_{E_*})\big)\\
&= \iota^{J'}_*\big(\cn(O_Y(Z))^2([C^{J'};F_J(O_Y(C_1),\ldots)]_{E_*})\big)\\
&= \cn(O_Y(Z))\big(\iota^{J'}_*\big(\cn(O_Y(Z))([C^{J'};F_J(O_Y(C_1),\ldots)]_{E_*})\big)\big)\\
&=\cn(O_Y(Z))\big(\iota_{Y*}(\cn(O_Y(C_*))^{J'}([Z;F_J(O_Y(C_1),\ldots)]_{E_*}))\big)\\
&=\iota_{Z*}\big((\cn(O_Y(C_*))^{J}([Z;F_J(O_Y(C_1),\ldots)]_{E_*})\big),
\end{align*}
verifying \eqref{multline:eqn1}.

If $J'=(0,\ldots,0)$, then $J=(1,0,\ldots,0)$, $C^J=C_1=Z$, $i^J=i_Z$, $\iota^J=\iota_Z$ and $\cn(O_Y(C_*))^J=\cn(O_Y(Z))$. The desired relation \eqref{multline:eqn1} becomes a simple identity in this case, finishing the proof.
\end{proof}

\section{Descent to $\Omega_*(X)_{D, E_*}$}\label{sec:Descent}  Let $X$ be a finite type $k$-scheme with pseudo-divisors $D, E_1,\ldots, E_r$ and a pseudo-divisor $C$ supported in $D$. We proceed in a series of steps to show that intersection with  $C$ descends to   $C(-)_{E_*}:\Omega_*(X)_{D, E_*}\to\Omega_{*-1}(|D|)_{E_*}$.

\medskip
\noindent
\emph{Step 1:} The descent to $\uu{\sZ}_*(X)_{D, E_*}$.   Let $\pi:Y\to Z$ be a  smooth morphism with $Z$ and $Y$ in $\Sm_k$ irreducible, $L_1,\ldots, L_m$ line bundles on $Z$ with $m>\dim_kZ$, and $f:Y\to X$ a morphism in $\sM(X)_{D, E_*}$. Using lemmas~\ref{lem:ProjectionFormula} and \ref{lem:ChernClassComp}, it suffices to show that $(f^*C)(\id_Y,\pi^*L_1,\ldots,\pi^*L_m)_{E_*}=0$ in $\Omega_*(|f^*D|)_{E_*}$. Changing notation, we may assume that $X=Y$, $f=\id_Y$, and that either $|D|\neq Y$ and  $\Div D$ is a simple normal crossing divisor on $Y$, or $|D|=Y$. We need to show that $C(\id_Y,\pi^*L_1,\ldots,\pi^*L_m)_{E_*}=0$ in $\Omega_*(|D|)_{E_*}$.

If $Y=|D|$, then 
\begin{align*}
C(\id_Y,\pi^*L_1,\ldots,\pi^*L_m)_{E_*}&=\cn(O_Y(C))(\id_Y,\pi^*L_1,\ldots,\pi^*L_m)\\
&=(\id_Y,\pi^*L_1,\ldots,\pi^*L_m,O_Y©),
\end{align*}
which is zero in $\Omega_*(|D|)_{E_*}=\Omega_*(Y)_{E_*}$ by the relations $\<\sR^{Dim}_*\>(Y)_{E_*}$.

If $Y\neq |D|$  with inclusion $i:|D|\to Y$ and $\Div D$ is a simple normal crossing divisor on $Y$, then 
\begin{multline*}
C\left(\id_Y,\pi^*L_1,\ldots,\pi^*L_m\right)_{E_*}\\=
\cn((\pi\circ i)^*L_1)\circ\ldots\circ\cn((\pi\circ i)^*L_m)([C\to|D|]_{E_*}).
\end{multline*}
To see that this class vanishes, apply  lemma~\ref{lem:StrongDim} to $\pi\circ i:|D|\to Z$.

\medskip
\noindent
\emph{Step 2:} The descent to $\uu{\Om}_*(X)_{D, E_*}$. Let $f:Y\to X$ be in $\sM(X)_{D, E_*}$, let $Z\to Y$ be a codimension one smooth closed subscheme in good position with respect to $D, E_*$ and let $C$ be a  pseudo-divisor on $X$ with support in $D$.  We may suppose that $Y$ is irreducible. As in step 1, we  reduce to the case $X=Y$ and $f=\id_Y$,  and it suffices to show that
\[
C([Y; O_Y(Z)]_{D, E_*})_{E_*}=C([Z\to Y]_{D,E_*})_{E_*}.
\]

If $Y=|D|$, then $\sM(Y)_{D, E_*}=\sM(Y)_{E_*}$, $[Z\to Y]_{D, E_*}=[Z\to Y]_{E_*}$ and the relations  $\<\sR^{Sect}_*\>(Y)_{E_*}$ yield
\begin{align*}
C([Y; O_Y(Z)]_{D, E_*})_{E_*}&=\cn(O_Y(C))\circ \cn(O_Y(Z))(1_Y^{E_*})\\
&=\cn(O_Y(C))([Z\to Y]_{E_*})\\
&=C([Z\to Y]_{D, E_*})_{E_*}.
\end{align*}
In case $Y\neq |D|$, then $\Div D$ is a simple normal crossing  divisor on $Y$, which  by lemma~\ref{lem:AdmissibleIntersection} is admissible with respect to $E_*$. The divisor $\Div C$ has support contained in $|D|$ and is therefore also  admissible with respect to $E_*$. Let $i:|D|\to Y$ be the inclusion.

Using lemma~\ref{lem:IntersectionFormula}, we have
\begin{align*}
C([Y; O_Y(Z)]_{D, E_*})_{E_*}&=\cn(i^*O_Y(Z))(C(1_Y^{D, E_*})_{E_*})\\
&=\cn(i^*O_Y(Z))([C\to|D|]_{E_*})\\
&=C([Z\to Y]_{D,E_*})_{E_*},
\end{align*}
as desired.

\medskip
\noindent
\emph{Step 3:} The descent to $\Omega_*(X)_{D, E_*}$. Let $f:Y\to X$ be in $\sM(X)_{D, E_*}$ and let $L$ and $M$ be line bundles on $Y$. As above, we may assume that $Y=X$, $f=\id_Y$, and it suffices  to show that
\[
C([Y; F_\L(L,M)]_{D,E_*})_{E_*}=C([Y; L\oo M]_{D, E_*})_{E_*}.
\]
 Using the relations
 $\<\sR_*^{FGL}\>(|D|)_{E_*}$ and lemma~\ref{lem:ChernClassComp2}, we
have
\begin{align*}
C\left([Y; F_\L(L,M)]_{D, E_*}\right)_{E_*}&=C\left(F_\L(\cn(L),\cn(M))(1^{D, E_*}_Y)\right)_{E_*}\\
&=F_\L(\cn(L),\cn(M))(C(1_Y^{D, E_*})_{E_*})\\
&=\cn(L\oo M)(C(1_Y^{D, E_*})_{E_*})\\
&=C([Y; L\oo M]_{D, E_*})_{E_*},
\end{align*}
as desired.

\section{Relations for intersections}\label{sec:Relations}

\subsection{Commutativity} We establish the 
commutativity of the intersection maps.  We begin with some preliminary results.

\begin{lemma}\label{lem:ClassPushforward} Let $E_1,\ldots, E_r$ be pseudo-divisors on some $Y\in\Sm_k$. Let $D$ and $B$ be effective Cartier divisors on $Y$. Suppose that $\id_Y$ is in $\sM(Y)_{D, E_*}\cap \sM(Y)_{B, E_*}$ and $D+B$ is a simple normal crossing divisor on $Y$. Let $C$ be an effective Cartier divisor on $Y$ with $|C|\subset |D|$. and let $i:|D|\cap |B|\to |B|$, $i_B:|B|\to Y$ be the inclusions. Then $B$ is admissible with respect to $D, E_*$ and $E_*$, and
\[
i_*\left(i_B^*C([B\to |B|]_{D, E_*})_{E_*}\right)=\cn(i^*_BO_Y(C))([B\to |B|]_{E_*}).
\]
\end{lemma}

\begin{proof}  As $B$ is in good position with respect to $D, E_*$, it follows from remark~\ref{rem:GenPosition}(1) that  $B$ is admissible with respect to $D, E_*$ . Since $\id_Y$ is in $\sM(Y)_{B, E_*}$, $B$ is admissible with respect to $E_*$ by lemma~\ref{lem:GoodPosition}(1). Thus the terms in the conclusion are all defined. 

We note that $C$ is in good position with respect to $B, E_*$. Let $F$ be an irreducible component of a face of $B$. Suppose that $F\not\subset |D|$. We apply lemma~\ref{lem:GoodPosition}(3) with $E=C$, which tells us that $C\cap F$ is a simple normal crossing divisor on $F$, in good position with respect to $E_*$. Thus, letting $i_F:F\to Y$, $\iota_F:|D|\cap F\to F$ be the inclusions, we have
\[
\iota_{F*}\left(i_F^*C(1_F^{D, E_*})_{E_*}\right)=\iota_{F*}\left([C\cap F\to |D|\cap F]_{E_*}\right)=\cn(i_F^*O_Y(C))(1_F^{E_*}),
\]
the first equality being the definition of the intersection and the second following from lemma~\ref{lem:DivClass}. In case $F\subset |D|$, then 
$\iota_F=\id_F$ and
\[
\iota_{F*}\left(i_F^*C(1_F^{D, E_*})_{E_*}\right)=i_F^*C(1_F^{D, E_*})_{E_*}=\cn(i_F^*O_Y(C))(1_F^{E_*}).
\]
Thus, for each face $B^J$ of $B$ with inclusion $\iota^J:B^J\to |B|$, we have
\[
i_*\left(i_B^*C(\iota^J_*(1_{B^J}^{D, E_*}))_{E_*}\right)=\cn(i_B^*O_Y(C))(\iota^J_*(1^{E_*}_{B^J})).
\]

For $B=\sum_{i=1}^mn_i B_i$, the divisor class $[B\to |B|]_{D, E_*}$ is
\[
[B\to |B|]_{D, E_*}=\sum_JF_J(\cn(L_1),\ldots, \cn(L_m))(\iota^J_*(1_{B^J}^{D, E_*})),
\]
where $L_j=i_B^*O_Y(B_j)$ and $F_J=F_J^{n_1,\ldots, n_m}$. Similarly,
\[
[B\to |B|]_{E_*}=\sum_JF_J(\cn(L_1),\ldots, \cn(L_m))(\iota^J_*(1_{B^J}^{E_*})).
\]
Thus
\begin{align*}
i_*(i_B^*C([B\to |B|&]_{D, E_*})_{E_*})\\&=\sum_JF_J(\cn(L_1),\ldots, \cn(L_m))\left(i_*(i_B^*C(\iota^J_*(1_{B^J}^{D, E_*}))_{E_*})\right)\\
&=\sum_JF_J(\cn(L_1),\ldots, \cn(L_m))\left(\cn(i_B^*O_Y(C))(\iota^J_*(1^{E_*}_{B^J}))\right)\\
&=\cn(i_B^*O_Y(C))\left(\sum_JF_J(\cn(L_1),\ldots, \cn(L_m))(\iota^J_*(1_{B^J}^{E_*}))\right)\\
&=\cn(i_B^*O_Y(C))([B\to|B|]_{E_*}).
\end{align*} 
\end{proof}

Write the universal formal group law as $F_\L(u,v)=u+v+uvF_{11}(u,v)$ and let $G_{11}(u,v)=vF_{11}(u,v)$.

\begin{lemma}\label{lem:addition} Let $W$ be in $\Sm_k$ with pseudo-divisors $E_1,\ldots, E_r$ and simple normal crossing divisor $D$ such that that $\id_W$ is in   $\sM(W)_{D, E_*}$.  Let $C$ be an effective Cartier divisor on $W$, with $|C|\subset |D|$ and suppose $C=C_0+C_1$, with $C_0>0$, $C_1>0$, and $C_1$  smooth. \\[5pt]
(1)  We have the identity in $\Omega_*(|D|)_{E_*}$
\begin{multline*}
[C\to|D|]_{E_*}=[C_0\to|D|]_{E_*}+[C_1\to|D|]_{E_*}\\
+G_{11}(\cn(O_W(C_1)),\cn(O_W(C_0)))([C_1\to|D|]_{E_*}).
\end{multline*}
\ \\[5pt]
(2) Let $f:Y\to W$ be in  $\sM(W)_{D, E_*}$.  Suppose that $Y$ is irreducible and either $f(Y)\subset |D|$ or $f^*C_1$ is a smooth divisor on $Y$. Then
\begin{align*}
(f^*C)(&1_Y^{D, E_*})_{E_*}\\
&= (f^*C_0)(1^{D, E_*}_Y)_{E_*}+( f^*C_1)(1_Y^{D, E_*})_{E_*}\\
&\hskip 20pt+ G_{11}\big(\cn(O_Y(f^*C_1)),\cn(O_Y(f^*C_0))\big)
((f^*C_1)(1_Y^{D, E_*})_{E_*})
\end{align*}
in $\Omega_*(|f^*D|)_{E_*}$.
\end{lemma}

\begin{proof}   For (2), suppose first that $|f^*D|=Y$.   Then $\sM(Y)_{D, E_*}=\sM(Y)_{E_*}$ and by definition
\begin{align*}
&(f^*C)(1_Y^{D, E_*})_{E_*}=\cn(O_Y(f^*C))(1_Y^{E_*}),\\
&(f^*C_i)(1_Y^{D, E_*})_{E_*}=\cn(O_Y(f^*C_i))(1_Y^{E_*});\quad i=0,1.
\end{align*}
Since $O_Y(f^*C)=O_Y(f^*C_0)\otimes O_Y(f^*C_1)$ we have 
\begin{multline*}
\cn(O_Y(f^*C))=\cn(O_Y(f^*C_0))+\cn(O_Y(f^*C_1))\\+ G_{11}(\cn(O_Y(f^*C_1)),\cn(O_Y(f^*C_0))\circ \cn(O_Y(f^*C_1)).
\end{multline*}
Thus (2) follows by applying this identity to $1_Y^{E_*}$.

If $|f^*D|\neq Y$, then (2)  is a consequence of (1).  Indeed, in this case, $f^*D$ is a simple normal crossing divisor on $Y$  by lemma~\ref{lem:AdmissibleIntersection} and
\begin{gather*}
(f^*C)(1_Y^{D, E_*})_{E_*}=[f^*C\to|f^*D|]_{E_*}\\
(f^*C_i)(1_Y^{D, E_*})_{E_*}=[f^*C_i\to|f^*D|]_{E_*},\ i=0,1.
\end{gather*}
Thus, applying the first assertion to the divisor  $f^*C=f^*C_0+f^*C_1$ on $Y$,  we have
\begin{align*}
(&f^*C)(1_Y^{D, E_*})_{E_*}=[f^*C\to|f^*D|]_{E_*}\\
&=[f^*C_0\to|f^*D|]_{E_*}+[f^*C_1\to|f^*D|]_{E_*}\\
&\hskip10pt+G_{11}(\cn(O_Y(f^*C_1)),\cn(O_Y(f^*C_0)))([f^*C_1\to|f^*D|]_{E_*})\\
&=(f^*C_0)(1^{D, E_*}_Y)_{E_*}+(f^*C_1)(1^{D, E_*}_Y)_{E_*}\\
&\hskip10pt+G_{11}(\cn(O_Y(f^*C_1)),\cn(O_Y(f^*C_0)))((f^*C_1)(1^{D_1, E_*}_Y)_{E_*}).
\end{align*}

We now prove (1).  By lemma~\ref{lem:AdmissibleIntersection}, $D$  is admissible with respect to $E_*$. Thus 
$C$, $C_0$ and $C_1$ are all admissible with respect to $E_*$, so the divisor classes $[C\to|D|]_{E_*}$, $[C_0\to|D|]_{E_*}$ and $[C_1\to|D|]_{E_*}$ are all defined.

We first reduce to the case of irreducible $C_1$.   In general,   suppose  $C_1=C_1'+C_1''$ with $C_1'>0$, $C_1''>0$; as  $C_1$ is admissible with respect to $E_*$, so are $C_1'$ and $C_1''$. 

Since $C_1'$ and $C_1''$ are disjoint,  applying the relations $\<\sR^{Sect}_*\>(-)_{E_*}$ gives
\[
\cn(O_W(C''_1))([C'_1\to |D|]_{E_*})=0=\cn(O_W(C'_1))([C''_1\to |D|]_{E_*}). 
\]
Noting that $O_W(C')=O_W(C_0)\otimes O_W(C'_1)$, this gives
\begin{align*}
G_{11}&(\cn(O_W(C_1'')),\cn(O_W(C'))([C_1''\to|D|]_{E_*})\\
&=G_{11}(\cn(O_W(C_1'')),F_\L(\cn(O_W(C_0),\cn(C'_1)))([C_1''\to|D|]_{E_*})\\
&=G_{11}(\cn(O_W(C_1'')),\cn(O_W(C_0))([C_1''\to|D|]_{E_*}),
\end{align*}
and
\begin{multline*}
G_{11}(\cn(O_W(C_1'')), \cn(O_W(C_1'))([C_1''\to |D|]_{E_*})\\=
F_{11}(\cn(O_W(C_1'')), \cn(O_W(C_1'))\circ \cn(C_1')([C_1''\to |D|]_{E_*})=0.
\end{multline*}

Additionally, the definition of the divisor class tells us that
\[
[C_1\to|D|]_{E_*}=[C'_1\to|D|]_{E_*}+[C''_1\to|D|]_{E_*}.
\]
We therefore have
\[
\cn(O_W(C''_1)\circ \cn(O_W(C'_1)([C_1\to|D|]_{E_*})=0,
\]
giving   the identities
\begin{align*}
G_{11}&(\cn(O_W(C_1)),\cn(O_W(C_0)))([C_1\to|D|]_{E_*})\\
&=G_{11}(\cn(O_W(C'_1))+\cn(O_W(C_1'')),\cn(O_W(C_0)))([C_1\to|D|]_{E_*})\\
&=G_{11}(\cn(O_W(C'_1),\cn(O_W(C_0)))([C'_1\to|D|]_{E_*})\\
&\hskip 30pt+G_{11}(\cn(O_W(C_1''),\cn(O_W(C_0)))([C_1''\to|D|]_{E_*}).
\end{align*}
With these formulas, one easily shows that (1) for the decompositions  $C'=C_0+C'_1$, $C=C'+C''_1$ and $C_1=C_1'+C''_1$ implies (1) for  $C=C_0+C_1$.  

We now assume $C_1$ is irreducible. Write $C=\sum_{i=1}^mn_iC_i$, with each $C_i$ irreducible. For each face
$C^J$ of $C$, let $\iota^J:C^J\to |D|$ be the inclusion; if $J=(1,j_2,,\ldots, j_m)$, then $C^J\subset C_1$ and we let $\iota^J_1:C^J\to C_1$ be the inclusion. We write $\iota^1:C_1\to|D|$ for the inclusion $\iota^{(1,0,\ldots, 0)}$.

Let $F_n$ denote the $n$-fold sum in the formal group $(F_\L,\L_*)$.  We have $F_{(0,j_2,\ldots,j_m)}^{0,n_2,\ldots, n_m}(u_1,u_2,\ldots, u_m)=F_{(j_2,\ldots, j_m)}^{n_2,\ldots, n_m}(u_2,\ldots, u_m)$ and $F_{(j_1,\ldots,j_m)}^{0,n_2,\ldots, n_m}=0$ if $j_1\neq0$.

The identity $F_\L(u_1, F_{n-1}(u_2,\ldots,u_n))=F_n(u_1,\ldots, u_n)$ gives us the identity
\begin{equation}\label{eqn:Splitting}
\sum_Ju^JF_J^{n_1,\ldots, n_m}(u_1,\ldots, u_m)=u_1+V+ u_1VF_{11}(u_1,V),
\end{equation}
where  
\[
V= F^{n_1-1,\ldots, n_m}(u_1,\ldots, u_m)=
\sum_{J'}u^{J'}F_{J'}^{n_1-1,n_2,\ldots, n_m}(u_1,u_2,\ldots,u_m).
\]
Here the sum $\sum_{J'}$ is over all faces of $C_0$ (with the convention of using indices $J'=(0, j_2,\ldots, j_m)\in\{0,1\}^m$ in case $n_1=1$).

Write $u_1VF_{11}(u_1,V)=u_1G_{11}(u_1,V)$ as the sum
\[
u_1VF_{11}(u_1,V)=\sum_Ku^KF'_K(u_1,\ldots,u_m),
\]
where the sum is over all indices $K=(1, k_2\ldots, k_m)\in\{0,1\}^m$ and we follow our usual convention of requiring that $F'_K$ does not involve $u_i$ if $k_i=0$.  The identity \eqref{eqn:Splitting} yields
\[
F^{n_1,\ldots, n_m}_{(1,0,\ldots,0)}=1+F^{n_1-1,\ldots, n_m}_{(1,0,\ldots,0)}+F'_{(1,0,\ldots,0)},
\]
and  for $K=(1,k_2,\ldots, k_m)\neq (1,0,\ldots,0)$ and for $J'=(0,j_2,\ldots, j_m)$
\begin{align*}
F^{n_1,\ldots, n_m}_K&=F^{n_1-1,\ldots, n_m}_K+F'_K\\
F^{n_1,\ldots, n_m}_{J'}&=F^{n_1-1,\ldots, n_m}_{J'}.
\end{align*}
Thus, comparing the definitions of  $[C\to|D|]_{E_*}$, $[C_1\to|D|]_{E_*}$ and $[C_0\to|D|]_{E_*}$ , we see that
\begin{multline*}
[C\to|D|]_{E_*}-[C_1\to|D|]_{E_*}-[C_0\to|D|]_{E_*}\\
={\sum_K}'\iota^K_*F'_K(\cn(O(C_1)),\ldots,\cn(O(C_m)))(1^{E_*}_{C^K}),
\end{multline*}
where $\sum'_K$ means the sum over all $K=(1,k_2,\ldots, k_n)\in\{0,1\}^m$. We therefore need to show
\begin{multline}\label{multline:ToShow}
{\sum_K}'\iota^K_*F'_K(\cn(O(C_1)),\ldots,\cn(O(C_m)))(1^{E_*}_{C^K})\\
=G_{11}(\cn(O(C_1)),\cn(O(C_0)))([C_1\to|D|]_{E_*}).
\end{multline}

We note that $C_i$ is in good position with respect to $D, E_*$ for  $i=0,1$. Let $K=(1, k_2,\ldots, k_m)$, with an index $k_i=1$, and let $K'=(1, k_2,\ldots, k_{i-1}, k_i-1, k_{i+1},\ldots, k_m)$.  By lemma~\ref{lem:GoodPosition}(2), the smooth divisor $C^K=C^{K'}\cap C_i$ on $C^{K'}$ is in good position with respect to $E_*$.  By repeated applications of the relations $\<\sR_*^{Sect}\>(-)_{E_*}$, we have
\begin{multline}\label{multline:Identity1}
\iota^K_{1*}\big(F'_K(\cn(O(C_1)),\ldots,\cn(O(C_m)))(1^{E_*}_{C^K})\big)\\
=\big(F'_K(\cn(O(C_1)),\ldots,\cn(O(C_m)))\circ\cn(O(C_*))^{K-1}\big)(1^{E_*}_{C_1}).
\end{multline}
where $K-1:=(0,k_2,\ldots, k_n)$. Note that 
\begin{equation}\label{eqn:SumRelation}
G_{11}(u_1,V)={\sum_K}'u^{K-1}F'_K(u_1,\ldots,u_m).
\end{equation}

The relations \eqref{multline:Identity1}, \eqref{eqn:SumRelation} together with the identities
\[
V(\cn(O(C_1)),\ldots,\cn(O(C_m)))=\cn(O(C_0)),\ \iota^K_*=\iota^1_*\circ \iota^K_{1*},
\]
imply
\begin{align*}
{\sum_K}'&\iota^K_*\left(F'_K(\cn(O(C_1)),\ldots,\cn(O(C_m)))(1^{E_*}_{C^K})\right)\\
&={\sum_K}'\iota_{1*}\left(\cn(O(C_*))^{K-1}F'_K(\cn(O(C_1)),\ldots,\cn(O(C_m)))(1^{E_*}_{C_1})\right)\\
&=\iota_{1*}\left(G_{11}(\cn(O(C_1)),V(\cn(O(C_1)),\ldots,\cn(O(C_m))))(1^{E_*}_{C_1})\right)\\
&=G_{11}(\cn(O(C_1)),\cn(O(C_0)))([C_1\to|D|]_{E_*}).
 \end{align*}
This verifies  \eqref{multline:ToShow}, completing the proof.
\end{proof}

\begin{proposition}[Commutativity]\label{prop:Commutativity} 
Let $T$ be in $\Sm_k$, irreducible, let $E_1,\ldots, E_r$ be pseudo-divisors on $T$, and let  $D, D'$ be effective Cartier divisors  on $T$. Suppose  that $D+D'$ is a simple normal crossing divisor on $T$, and that $\id_T$ is in $\sM(T)_{D, E_*}\cap \sM(T)_{D', E_*}$. Let $i_D:|D|\to T$ and $i_{D'}:|D'|\to T$ be the inclusions and let $C$ and $C'$ be effective Cartier divisors on $T$ with $|C|\subset |D|$, $|C'|\subset |D'|$. Then $C'$ is admissible with respect to $D, E_*$, $C$ is admissible with respect to $D', E_*$ and 
\[
(i_{D'}^*C)([C'\to |D'|]_{D, E_*})_{E_*}=(i_D^*C')([C\to|D|]_{D', E_*'})_{E_*}
\]
in $\Omega_*(|D|\cap|D'|)_{E_*}$.
\end{proposition}
 
\begin{proof}   Our hypotheses on $T$, $D$ and $D'$ imply that $D$ is in good position with respect to $D', E_*$,  and  $D'$ is in good position with respect to $D, E_*$. By remark~\ref{rem:GenPosition},   $C$ is in good position and admissible with respect to $D', E_*$,  and  $C'$ is in good position and admissible with respect to $D, E_*$, so all the terms in the conclusion are defined.  

Write $D=\sum_in_iD_i$, $D'=\sum_jn'_jD'_j$ with each $D_i$, $D_j'$ irreducible, and similarly $C=\sum_im_iD_i$ and $C'=\sum_jm'_jD'_j$ with $0\le m_i$, $0\le m_j'$. We proceed by induction on $m:=\sum_im_i$ and $m':=\sum_im'_i$.

Suppose $m=m'=1$; we may suppose $C=D_1$ and $C'=D'_1$.  In case $C=C'$, then $C\subset |D'|$ and $C'\subset |D|$, so
\begin{align*}
(i_{D'}^*C)([C'\to |D'|]_{D, E_*})_{E_*}&=\cn(\sO(C))([C'\to |D|\cap|D'|])\\&= (i_D^*C')([C\to|D|]_{D', E_*'})_{E_*}
\end{align*}
in $\Omega_*(|D|\cap|D'|)_{E_*}$. 

Suppose $C\neq C'$. By lemma~\ref{lem:GoodPosition}(2) the smooth divisor $C\cap C'$ on $C'$ is in good position with respect to $E_*$; similarly, the smooth divisor $C\cap C'$ on $C$ is also in good position with respect to $E_*$. Applying the relations $\<\sR^{Sect}_*\>(-)_{E_*}$ gives us the identities
\begin{align*}
&[C\cap C'\to |C|]_{E_*}=\cn(O_Y(C'))(1_C^{E_*})\text{ in }\Omega_*(|C|)_{E_*},\\
&[C\cap C'\to |C'|]_{E_*}=\cn(O_Y(C))(1_{C'}^{E_*})\text{ in }\Omega_*(|C'|)_{E_*}.
\end{align*}
Suppose $C\subset |D'|$ but $C'\not\subset |D|$. Then pushing forward the first identity to $|D|\cap |D'|$ gives
\[
[C\cap C'\to |D|\cap|D'|]_{E_*}=i^{C}_*(\cn(O_Y(C'))(1_{C}^{E_*}))
\]
in $\Omega_*(|D|\cap|D'|)_{E_*}$, where $i^C:C\to |D|\cap |D'|$ is the inclusion. This yields the identities in $\Omega_*(|D|\cap|D'|)_{E_*}$
\begin{align*}
(i_{D}^*C')([C\to |D|]_{D', E_*})_{E_*}&=i^C_*(\cn(O_Y(C'))(1_C^{E_*}))\\
&=[C\cap C'\to |D|\cap|D'|]_{E_*}\\
&=(i_{D'}^*C)([C'\to |D'|]_{D, E_*})_{E_*}.
\end{align*}
By symmetry, we have the desired identity if $C'\subset |D|$ but $C\not\subset |D'|$. If $C\cup C' \subset |D|\cap |D'|$, then 
 in $\Omega_*(|D|\cap|D'|)_{E_*}$
\begin{align*}
(i_{D}^*C')([C\to |D|]_{D', E_*})_{E_*}&=i^C_*(\cn(O_Y(C'))(1_C^{E_*}))\\
&=[C\cap C'\to |D|\cap|D'|]_{E_*}\\
&=i^{C'}_*(\cn(O_Y(C))(1_{C'}^{E_*}))\\
&=(i_{D'}^*C)([C'\to |D'|]_{D, E_*})_{E_*}.
\end{align*}
Finally, if $C\not\subset |D'|$ and $C'\not\subset|D|$, then 
\begin{align*}
(i_{D}^*C')([C\to |D|]_{D', E_*})_{E_*}&=[C\cap C'\to |D|\cap|D'|]_{E_*}\\
&=(i_{D'}^*C)([C'\to |D'|]_{D, E_*})_{E_*}.
\end{align*}

In the general case, we may assume that $D_1$ is a component of $C$. Let  $C_1=D_1$, $C_0=C-C_1$. By symmetry, it suffices to induct on
$m$ and assume the result for the pairs $C_0, C'$ and $C_1, C'$. Thus
\begin{align*}
&(i_{D'}^*C_0)([C'\to|D'|]_{D, E_*})_{E_*}=(i_{D}^*C')([C_0\to|D|]_{D', E_*})_{E_*}\\
&(i_{D'}^*C_1)([C'\to|D'|]_{D, E_*})_{E_*}=(i_{D}^*C')([C_1\to|D|]_{D', E_*})_{E_*}
\end{align*}
in $\Omega_*(|D|\cap|D'|)_{E_*}$.

The divisor class $[C'\to|D'|]_{D, E_*}$ is an $\L_*$-linear combination of cobordism cycles of the form $(g:Y\to |D'|, M_1,\ldots, M_s)$
with $g(Y)\subset |D|$, or  with $g^*(C_1)$ a smooth divisor on $Y$, so we may apply lemma~\ref{lem:addition}(2) to give
\begin{multline*}
(i_{D'}^*C)([C'\to|D'|]_{D, E_*})_{E_*}=\\
(i_{D'}^*C_0)([C'\to|D'|]_{D, E_*})_{E_*}+(i_{D'}^*C_1)([C'\to|D'|]_{D, E_*})_{E_*}\\
+G_{11}(\cn(O_T(C_1)),\cn(O_T(C_0)))((i_{D'}^*C_1)([C'\to|D'|]_{D,E_*})_{E_*}.
\end{multline*}
Using our induction hypothesis, together with lemma~\ref{lem:ChernClassComp2} and lemma~\ref{lem:addition}(1), we have the identities in $\Omega_*(|D|\cap|D'|)_{E_*}$:

\begin{align*}
(i_{D'}^*C)&([C'\to|D'|]_{D, E_*})_{E_*}\\
&=(i_D^*C')([C_0\to|D|]_{D', E_*})_{E_*}+(i_D^*C')([C_1\to|D|]_{D', E_*})_{E_*}\\
&\hskip 10pt +G_{11}(\cn(O_T(C_1)),\cn(O_T(C_0)))((i_D^*C')([C_1\to|D|]_{D', E_*})_{E_*})\\
&=(i_D^*C')\big([C_0\to|D|]_{D', E_*}+[C_1\to|D|]_{D', E_*}\\
&\hskip 10pt +G_{11}(\cn(O_T(C_1)),\cn(O_T(C_0)))([C_1\to|D|]_{D', E_*})\big)_{E_*}\\
&=(i_D^*C')([C\to|D|]_{D', E_*})_{E_*},
\end{align*}
as desired.
\end{proof}

\subsection{Linear equivalent pseudo-divisors} We show how to relate
the intersection  $D_0(-)_D$ with a naive intersection $D_1\cap(-)$ for linearly equivalent pseudo-divisors $D_0$, $D_1$, in a particular situation.

\begin{proposition}\label{prop:LinearEquiv2} 
Let $W$ be in $\Sm_k$ and let $E_1,\ldots, E_r$ be pseudo-divisors on $W$. Let $f:T\to W$ be a morphism in $\Sm_k$ and let $D_0, D_1$ and $B$ be simple normal crossing divisors on $T$. We suppose that $\id_T$ is in $\sM(T)_{D_0, E_*}\cap \sM(T)_{B, E_*}$, that $D_0+B$ is a simple normal crossing divisor on $T$ and that $O_T(D_0)\cong O_T(D_1)$. In addition, suppose that $T$ has an open subscheme $V$ with $|D_1|\subset V$ such that
\addtolength{\textwidth}{-15pt}

\ \hskip5pt\begin{minipage}[c]{\textwidth}
\ \\
\hbox to0pt{\hss (a)\ }The restriction of $f$ to $V$ is a smooth morphism $f_V:V\to W$ of relative dimension one and with geometrically irreducible fibers.\\
\hbox to0pt{\hss (b)\ }There is a simple normal crossing divisor $\bar{B}$ on $W$ such that $\bar{B}$ is admissible with respect to $E_*$ and such that $B\cap V=f_V^*(\bar{B})$. \\
\hbox to0pt{\hss (c)\ }$f_V:V\to W$ admits a section $s:W\to V$ and $D_1$ is the reduced divisor $s(W)$.\\
\end{minipage}
\addtolength{\textwidth}{15pt}
\\
Let $i_j:|D_j|\cap|B|\to |B|$ be the inclusion, $i=0,1$. Then the Cartier divisor $B$ on $T$ is admissible with respect to $D_0, E_*$, the Cartier divisor $B\cap D_1$ on $D_1$ is admissible with respect to $E_*$ and
\[
i_{0*}\left(D_0([B\to |B|]_{D_0, E_*})_{E_*}\right)=i_{1*}\left([B\cap D_1\to |D_1|\cap|B|]_{E_*}\right)
\]
in $\Omega_*(|B|)_{E_*}$.
\end{proposition}

\begin{proof} By lemma~\ref{lem:ClassPushforward}, $B$ is admissible with respect to $D_0, E_*$ and $E_*$, and
\[
i_{0*}\left(D_0([B\to |B|]_{D_0, E_*})_{E_*}\right)=\cn(O_T(D_0))([B\to |B|]_{E_*})
\]
in $\Omega_*(|B|)_{E_*}$.

On the other hand,  let $B^J$ be a face of $B$ with non-empty intersection with $V$. Then by our assumptions on $f_V$,  there is a corresponding face $\bar{B}^J$ of $\bar{B}$ such that $f_V: (B\cap V)^J\to \bar{B}^J$ is smooth with fibers of dimension one. Let $F$ be an irreducible component of $B^J$ with $F\cap V\neq\0$. Then $f(F)$ is a dense subset of an irreducible component $\bar{F}$ of $\bar{B}^J$ and  $F$ is the closure of $f_V^{-1}(\bar{F})$. If $E_{i_0}$ is the leading pseudo-divisor for $\bar{F}$ (with respect to $E_*$), then $E_{i_0}$ is also the leading pseudo-divisor for $F$. Since $F$ is admissible with respect to $E_*$, $f^*E_{i_0}\cap F$ is a simple normal crossing divisor on $F$.  As $D_1$ is equal to $s(W)$, we see that $(D_1+\Div f^*E_{i_0})\cap F$ is a simple normal crossing divisor on $F$, and thus $D_1\cap F$ is in good position on $F$ with respect to $E_*$.  If $\bar{F}$ has no leading pseudo-divisor for $E_*$, the same holds for $F$, and the smooth divisor $D_1\cap F$ on $F$ is again in good position on $F$ with respect to $E_*$. If $F'$ is an irreducible component of $B^J$ with $F'\cap V=\0$, then, as $D_1\cap F'=\0$, $D_1\cap F'$ is again in good position with respect to $E_*$. Thus the smooth divisor $D_1\cap B^J$ on $B^J$ is in good position with respect to $E_*$. 

Thus, by the relations $\<\sR^{Sect}_*\>(B^J)_{E_*}$, if $B^J$ is a face of $B$ with non-empty intersection with $V$, then 
\begin{equation}\label{eqn:*}
[B^J\cap D_1\to |B^J|]_{E_*}=\cn(O_T(D_1))(1_{B^J}^{E_*})
\end{equation}
in $\Omega_*(|B^J|)_{E_*}$. Similarly, if $B^K$ is a face of $B$ with empty intersection with $V$, then $B^K\cap D_1=\0$ and 
\[
\cn(O_T(D_1))(1_{B^K}^{E_*})=0,
\]
by applying  $\<\sR_*^{Sect}\>(B^K)_{E_*}$ in the case of an empty divisor on $B^K$.

From (b) and (c), the divisor $B\cap D_1$ on $D_1$  is admissible with respect to $E_*$. The divisor class $[B\cap D_1\to |B|\cap|D_1|]_{E_*}$ is given by a sum of the form
\[
[B\cap D_1\to |B|\cap|D_1|]_{E_*}=\sum_JF_J(\cn(L_1),\ldots)(\bar\iota^J_*(1_{B^J\cap D_1}^{E_*}))
\]
where the sum is over the faces $B^J$ of $B$ with $B^J\cap V\neq\0$,  and $\bar\iota^J:B^J\cap D_1\to |B|\cap|D_1|$ is the inclusion. We have a similar description of $[B\to |B|]_{E_*}$:
\[
[B\to |B|]_{E_*}=\sum_JF_J(\cn(L_1),\ldots)(\iota^J_*(1_{B^J}^{E_*}))+\sum_KF_K(\cn(L_1),\ldots)(\iota^K_*(1_{B^K}^{E_*}))
\]
where the first sum is over faces $B^J$ of $B$ with $B^J\cap V\neq\0$ and the second sum is over faces $B^K$ of $B$ with $B^K\cap V=\0$. 

Putting this all together  and using \eqref{eqn:*} gives us
\begin{align*}
i_{0*}\left(D_0([B\to |B|]_{D_0, E_*})_{E_*}\right)&=\cn(O_T(D_0))([B\to |B|]_{E_*})\\
&=\cn(O_T(D_1))([B\to |B|]_{E_*})\\
&=\sum_JF_J(\cn(L_1),\ldots)\cn(O_T(D_1))(\iota^J_*(1_{B^J}^{E_*}))\\
&=i_{1*}\left(\sum_JF_J(\cn(L_1),\ldots)(\bar\iota^J_*(1_{B^J\cap D_1}^{E_*}))\right)\\
&=i_{1*}([B\cap D_1\to |B|\cap|D_1|]_{E_*}).
\end{align*}
\end{proof}

\section{A moving lemma}\label{sec:MovLem} Let $E_1, \ldots, E_r, D$ be pseudo-divisors on a $k$-scheme $X$. We will show in this section that the forgetful map  
\[
\res_{E_*,D/E_*}:\Omega_*(X)_{E_*,D}\to\Omega_*(X)_{E_*}
\]
is an isomorphism and thus   $\res_{E_*/\0}:\Omega_*(X)_{E_*}\to \Omega_*(X)$ is an isomorphism. This enables us to the define the intersection map
\[
D(-):\Omega_*(X)\to \Omega_{*-1}(X)
\]
as the composition
\[
\Omega_*(X)\xrightarrow{\res_{D/\0}^{-1}} \Omega_*(X)_D\xrightarrow{D(-)}\Omega_{*-1}(X).
\]

\subsection{Blow-up diagrams}\label{subsec:BlowUp}   We discuss some properties of a generalized deformation diagram; this is a reformulation  and extension of the discussion in \cite[\S 3.2.1]{CobordismBook}, collecting most of what we will need in the following omnibus construction.

\begin{lemma}\label{lem:Blowup} Let $Y$ be in $\Sm_k$, let $Z\subset Y$ be a closed subscheme, let $\tau:Y'\to Y$ be the blowup of $Y$ along $Z$. Let $Y':=\Proj_{\sO_Y}\oplus_{n\ge0}\sI_Z^n$  and let $O(1)\to Y'$ be the tautological quotient line bundle, with zero section $s:Y'\to O(1)$. Let $\rho:T\to Y\times\A^1$ be the blowup of $Y\times\A^1$ along $Z\times0$, let $\<Y\times0\>, \<Z\times\A^1\>\subset T$ be the proper transforms of $Y\times0$, $Z\times\A^1$, respectively, and let $T^0=T\setminus\<Z\times\A^1\>$.   Let $E=\tau^{-1}(Z)\subset Y$ be the exceptional divisor of $\tau$  and let $\sE=\rho^{-1}(Z\times 0)\subset T$ be the exceptional divisor of $\rho$. Then
\\
(1) $\<Y\times0\>$ is contained in $T^0$.\\
(2)There is an isomorphism $\alpha:\<Y\times0\>\to Y'$ making the diagram
\[
\xymatrix{
\<Y\times0\>\ar[r]^-\alpha\ar[dr]_{p_1\circ\rho}&Y'\ar[d]^\tau\\&Y}
\]
commute.\\
(3) There is a morphism of $Y$-schemes $\pi:T^0\to Y'$, such that $\sE\cap T^0=\pi^{-1}(E)$ and with $\alpha=\pi_{|\<Y\times0\>}$.\\
(4) There is an isomorphism of $Y'$-schemes $\psi:T^0\to O(1)$ with   $\psi(\<Y\times0\>)=s(Y')\subset O(1)$.
\end{lemma}

\begin{proof}  Denote the sheaf of graded $\sO_Y$-algebras $\oplus_{n\ge0}\sI_Z^n$ by $\sB_\bullet$. Let $p:Y\times\A^1\to Y$ be the projection. For a sheaf $\sF$ on $Y$ we write $\sF[t]$ for the sheaf $p^*\sF$ on $Y\times\A^1$ and we identify $Y\times\A^1$ with $\Spec_{\sO_Y}p_*\sO_Y[t]$. Let $\sC_\bullet$ be the graded sheaf of $\sO_Y[t]$-algebras $\oplus_{n\ge0}(\sI_Z[t]+(t)\sO_Y[t])^n$; by definition, $T$ is the $Y\times\A^1$-scheme $\Proj_{\sO_Y[t]}\sC_\bullet$ and $Y'=\Proj_{\sO}\sB_\bullet$.

The proper transforms $\<Y\times0\>$ and $\<Z\times\A^1\>$ are defined by the homogeneous ideal sheaves in $\sC_\bullet$ generated in degree one by $(t)$ and $\sI_Z[t]$, respectively. As $\sI_Z[t]$ and $t$ generate $\sC_\bullet^+$ as a sheaf of ideals in $\sC_\bullet$, it follows that $\<Y\times0\>\cap\<Z\times\A^1\>=\0$ and thus $\<Y\times0\>\subset T^0$, proving (1). Writing $t_{(1)}$ for $t\in \sC_1$,  the quotient algebra $\sC_\bullet/(t_{(1)})$ is $\sO_Y[t]\oplus \oplus_{n\ge1}\sI_Z^n$, which is  isomorphic to $\sB_\bullet$ up to a bounded algebra, hence $p_1\circ\rho:\<Y\times\A^1\>\to Y$ is isomorphic  to $\tau:Y'\to Y$ as a $Y$-scheme, giving us the isomorphism $\alpha:\<Y\times\A^1\>\to Y'$ in (2). 

For (3), we have the evident inclusion of sheaves of graded $\sO_Y$-algebras $\sB_\bullet \to \sC_\bullet$. Let $Q_\bullet\subset \sC_\bullet$ be a sheaf of homogeneous prime ideals. Then  $Q_{\bullet}\cap \sB_{\bullet}\supset \sB_{\bullet}^+$  if and only if $Q_{1}\supset \sI_{Z}$ if and only if the corresponding  point $[Q_\bullet]$ of $T$ lies in $\<Z\times\A^1\>$. Thus the inclusion $\sB_\bullet \to \sC_\bullet$ induces a well-defined morphism of $Y$-schemes $\pi:T^0\to Y'$.  The closed subscheme $E\subset Y'$ is defined by the homogenous ideal sheaf in $\sB_\bullet$ generated by $\sI_Z$ (in degree 0), and $\sE\subset T$ is similarly defined by the homogeneous ideal sheaf in $\sC_\bullet$ generated by $\sI_Z[t]+(t)\sO_Y[t]$ in degree 0. But if we restrict to a principal open subset $U_f$  of $T$ formed by inverting a section $f$ of $\sI_Z\subset \sC_1$ over some open subset $V$ of $Y$, the ideal sheaves on $U_f$ corresponding to $\sI_Z\sC_\bullet$ and $(\sI_Z, t)\sC_\bullet$ agree, since $t_{(1)}/f_{(1)}$ is a regular function on $U_f$ and $t=f\cdot(t_{(1)}/f_{(1)})$. Thus $\sE\cap T^0$ is defined by the ideal sheaf $\sI_Z\sC_\bullet$, which shows that $\pi^{-1}(E)=\sE\cap T^0$. The composition $\sB_\bullet \to \sC_\bullet\to  \sC_\bullet/(t_{(1)})$ is the map used to defined $\alpha:\<Y\times\A^1\>\to Y'$, finishing the proof of (3).

The line bundle $O(1)\to Y'$ is the affine $Y'$-scheme $\Spec_{\sO_{Y'}}\Sym^*\sO(-1)$ and, as $\sO(-1)$ is an invertible sheaf, $\Sym^*\sO(-1)=\oplus_{m\ge0}\sO(-m)$. The invertible sheaf $\sO(-m)$ is the invertible sheaf associated to the graded $\sB_\bullet$-module $\sB_\bullet[-m]$,  $B_\bullet[-m]_p=\sB_{p-m}$, and so $\Sym^*\sO(-1)$ is the sheaf of algebras associated to the polynomial algebra  $\sB_\bullet[x]$ over $\sB_\bullet$ with generator $x$ in degree 1. The image of the zero section in $O(1)$ is defined by the ideal $(x)\sB_\bullet[x]\subset \sB_\bullet[x]$. 

For the proof of (4), we have the map of graded $\sB_\bullet$-algebras $\psi^*:\sB_\bullet[x]\to\sC_\bullet$, which sends $\sB_nx^m=\sI_Z^nx^m$ to $\sC_{m+n}=(\sI_Z[t]+t\sO_Y[t])^{m+n}$ by setting $\psi^*(a x^m)=at^m$, where $a$ is a section of $\sI_Z^n$. 

We claim that $\psi^*$ defines an isomorphism of $Y'$-schemes $\psi:T^0\to O(1)$ and sends $\<Y\times0\>$ to $s(Y')\subset O(1)$.  To show this, it suffices to  handle the case of affine $Y$, $Y=\Spec A$,  and to show in addition that the morphism $\psi$ we define is natural in $A$. 

Let $Z\subset Y$ be defined by an ideal $I\subset A$ and let $B_\bullet=\oplus_{n\ge0}I^n$, $C_\bullet=\oplus_{n\ge0}(I[t]+tA[t])^n$. For $a\in I^m$, we denote the corresponding element of $B_m$ by $a_{(m)}$. For $f\in I$, we have the principal open subschemes $V_f\subset Y'$ defined by $f_{(1)}\in B_1$ and  $U_f\subset T$ defined by $f_{(1)}\in C_1$. $Y'$ is covered by the $V_f$,  $T^0$ is covered by the $U_f$ and $\pi:T^0\to Y'$ restricts to $\pi_f:U_f\to V_f$. 

By definition 
\begin{align*}
&V_f=\Spec B_\bullet[1/f_{(1)}]_0\\
&O(1)_{|V_f}=\Spec  B_\bullet[1/f_{(1)}]_{0}[x/f_{(1)}]\\
&U_f=\Spec C_\bullet[1/f_{(1)}]_0,
\end{align*}
with projection $O(1)_{|V_f}\to V_f$ given by the inclusion 
\[
B_\bullet[1/f_{(1)}]_0\to  B_\bullet[1/f_{(1)}]_0[x/f_{(1)}]. 
\]
The map $\psi^*$ gives rise to the homomorphism of $B_\bullet[1/f_{(1)}]_0$-algebras
\[
\psi^*_f:B_\bullet[1/f_{(1)}]_0[x/f_{(1)}]\to C_\bullet[1/f_{(1)}]_0
\]
with $\psi^*_f(x/f_{(1)})=t_{(1)}/f_{(1)}$. If $g$ is another element of $I$, the element $x/f_{(1)}$ maps to $g\cdot [x/(fg)_{(1)}]$ under restriction map  for $O(1)_{|V_{fg}}\subset O(1)_{|V_f}$. As $g\cdot [t_{(1)}/(fg)_{(1)}]=t_{(1)}/f_{(1)}$, the
maps $\psi^*_f$ and $\psi^*_{fg}$ are compatible with the respective restriction maps for $V_{fg}\subset V_f$ and $U_{fg}\subset U_f$, so the family $(\psi_f^*)_{f\in I}$ gives a well-defined morphism of $Y'$-schemes $\psi:T^0\to O(1)$. Clearly $\psi$ is natural in $A$.  The inverse to $\psi^*_f$ is given by sending $t^ia_{(m)}t_{(1)}^{n-m}/f_{(1)}^n\in (I[t]+tA[t])^n/f_{(1)}^n$, $a\in I^m$, $i\ge0$, to $f^i\cdot (a_{(m)}/f_{(1)}^m)\cdot (x/f_{(1)})^{n-m+i}$ (use the relation $t=f\cdot(t_{(1)}/f_{(1)})$). Thus $\psi$ is an isomorphism. As $\<Y\times0\>\cap U_f$ is defined by the ideal $(t_{(1)}/f_{(1)})$ and the zero section $s(V_f)$ in $O(1)_{|V_f}$ is defined by the ideal $(x/f_{(1)})$, $\psi$ restricts to an isomorphism of $\<Y\times0\>$ with $s(Y')$. 
\end{proof}

\subsection{Distinguished liftings}\label{SubSec:DistinguishedLifting} Given a finite type
$k$-scheme $X$ with  pseudo-divisors $E_1, \ldots, E_r,D$, we describe method for lifting elements of
$\sZ_*(X)_{E_*}$ to $\Omega_*(X)_{E_*, D}$. 

\begin{lemma}\label{lem:DistinguishedLifting} Let $Y$ be in $\Sm_k$ and
 let $E_1,\ldots, E_r, E_{r+1}$ be effective (non-zero) divisors on $Y$; we allow the case $r=0$, that is, we have only the divisor $E_{r+1}$. We let $E_*$ denote the sequence $E_1,\ldots, E_r$. Suppose that  $\id_Y$ is in $\sM(Y)_{E_*}$. Then\\
 (A)  there is a projective birational morphism
$\rho:W\to Y\times\P^1$, with $W\in\Sm_k$,  such that,  letting $\sE$ be the exceptional divisor of $\rho$,   letting $\<Y\times0\>$  denote  the proper transform to $W$ of $Y\times0$,  and letting $\hat\tau:\<Y\times0\>\to Y$ be the restriction of $p_1\circ\rho$, 
we have
\begin{equation}\label{eqn:DistingLiftingConditions}
\end{equation}
\begin{enumerate}
\item The fundamental locus of $\rho$ is contained in $|E_1|\times0$.
\item $\<Y\times0\>$ is smooth, the morphism $\hat{\tau}:\<Y\times0\>\to Y$ is birational, with fundamental locus contained in $|E_1|\cap |E_{r+1}|$, and $\hat{\tau}$ is in $\sM(Y)_{E_*, E_{r+1}}$.
\item  The morphism $p_1\circ\rho:W\to Y$ is in $\sM(Y)_{E_*}$ and $\rho^*(Y\times0)$ is in good position with respect to $E_*$.
\item  The morphism $\rho:W\to Y\times \P^1$ is in $\sM(Y\times\P^1)_{Y\times0, E_*, E_{r+1}}$.
\end{enumerate}
(B) If $\tau:Y'\to Y$ is a projective birational morphism with $Y'\in\Sm/k$, with fundamental locus contained in $|E_1|\cap|E_{r+1}|$ and with $\tau$ in $\sM(Y)_{E_*, E_{r+1}}$, then there is a $\rho$ as above, satisfying (1)-(4),   with $\hat\tau:\<Y\times0\>\to Y$ isomorphic to the $Y$-scheme $\tau:Y'\to Y$, such that 
\begin{enumerate}
\item[(5)] $\<Y\times0\>\cap \<E_1\times\P^1\>=\0$,
\end{enumerate} 
where $\<|E_1|\times\P^1\>$ is the proper transform to $W$ of $|E_1|\times\P^1$.
\\
(C) Let $A\subset Y$ be an effective Cartier divisor, in good position with respect to $E_*$, and suppose that $\tau:Y'\to Y$ is a morphism satisfying the hypotheses in (B) such that $\tau^*A$ is in good position with respect to $E_*, E_{r+1}$. Then there is a $\rho:W\to Y\times\P^1$ satisfying (1)-(5) and with 
$\hat\tau:\<Y\times0\>\to Y$ isomorphic to   $Y'\to Y$ as $Y$-schemes, such that $\rho^*(A\times\P^1)$ is in good position with respect to $Y\times0, E_*, E_{r+1}$ and  $\rho^*(A\times\P^1+Y\times0)$ is in good position with respect to $E_*$.
 \end{lemma}

\begin{proof} We may assume that $Y$ is irreducible. We begin the proof of (A) by showing that there exists a morphism $\tau:Y'\to Y$ satisfying the hypotheses in (B). We first suppose $r>0$. 

Let $\tilde\tau:\tilde{Y}\to Y$ be the blowup of $Y$ along $E_{(r+1)}:=\cap_{i=0}^{r+1}E_i$. As $|E_{(r+1)}|\subset |E_1|\cap|E_{r+1}|$, we see that the fundamental locus $F$ of $\tilde\tau$ is contained in $|E_1|\cap|E_{r+1}|$. Let $\tilde E$ be the exceptional divisor of $\tilde\tau$. Since $\id_Y$ is in $\sM(Y)_{E_*}$, it follows that $\tilde{Y}\setminus\tilde{E}\to Y\setminus F$ is in $\sM(Y\setminus F)_{E_*}$, in particular,  $\tilde\tau^*E_1\setminus \tilde E$ is a simple normal crossing divisor on $\tilde Y\setminus\tilde E$.

By resolution of singularities, there is a projective birational morphism $\phi:Y'\to \tilde Y$ with $Y'\in \Sm/k$ and with fundamental locus contained in $|\tilde{E}|$, such that $\phi^*(\tilde\tau^*E_1)$ is a simple normal crossing divisor on $Y'$. Letting $\tau:Y'\to Y$ be the composition $\tilde\tau\circ\phi$,  it follows immediately that $\tau$ has the desired properties. We note that in particular, the exceptional divisor $E'$ of $\tau$ is supported in $|\tau^*E_1|$, and   $\tau^*E_1$ is a simple normal crossing divisor on $Y'$. 

If $r=0$, we simply take $\tau:Y'\to Y$ to be a blowup of $Y$ along a closed subscheme of $Y$ supported in $|E_1|$ so that $Y'$ is in $\Sm_k$ and $\tau^*E_1$ is simple normal crossing divisor on $Y'$. 

Having shown the existence of a projective birational morphism $\tau:Y'\to Y$ with fundamental locus contained in $|E_1|\cap |E_{r+1}|$, such that $\tau$ is in  $\sM(Y)_{E_*, E_{r+1}}$, we choose any such morphism and proceed with the construction of $\rho:W\to Y\times \P^1$. 

Since $\tau$ is  a projective birational morphism and $Y$ is smooth, there is a closed subschem $Z_0$ with support equal to the fundamental locus of $\tau$ (which is  contained in $|E_1|\cap|E_{r+1}|$) such that $\tau$ is the blow-up of $Y$ along $Z_0$. Let $Z$ be the closed subscheme of $Y$ with ideal sheaf $\sI_Z=\sI_{Z_0}\cdot\sI_{E_1}$. Then $|Z|=|E_1|$ and since $\sI_{E_1}$ is a locally principal ideal sheaf, $\tau$ is also isomorphic to the blowup of $Y$ along $Z$.  Let $q:O(1)\to Y'$ be the tautological quotient line bundle on $Y'$ corresponding to this identification.

Let $\rho_1:W_1\to Y\times\P^1$ be the blowup of $Y\times\P^1$ along $Z\times0$, let $\<|E_1|\times \P^1\>_1$ and $\<Y\times 0\>_1$ be the proper transforms of $|E_1|\times \P^1$ and $Y\times 0$ to $W_1$, respectively, and let $\sE_1\subset W_1$ be the exceptional divisor of $\rho_1$. Let $W_1^0=\rho_1^{-1}(Y\times\A^1)\cap (W_1\setminus \<|E_1|\times \P^1\>_1)$, $F_1=|\sE_1|\cap \<|E_1|\times \P^1\>_1$ and $U_1=W_1\setminus F_1\supset W_1^0$. By lemma~\ref{lem:Blowup} we have
\addtolength{\textwidth}{-15pt}

\ \hskip5pt\begin{minipage}[c]{\textwidth}
\ \\
\hbox to0pt{\hss (a)\ }$\<Y\times 0\>_1\subset W_1^0$.\\
\hbox to0pt{\hss (b)\ }$p_1\circ\rho_1:\<Y\times 0\>_1\to Y$ and $\tau:Y'\to Y$ are  isomorphic  as  $Y$-schemes.\\
\hbox to0pt{\hss (c)\ }There is a $Y$-morphism $\pi:W_1^0\to Y'$ and an isomorphism of $Y'$-schemes $\psi:W_1^0\to O(1)$, with $\psi$ sending 
$\<Y\times 0\>_1$ to the zero section of $O(1)$ and sending $\sE_1\cap W_1^0$ isomorphically onto $(\tau\circ q)^{-1}(Z)$.
\end{minipage}
\addtolength{\textwidth}{15pt}
\\
In particular, $U_1$ is smooth, $\<Y\times 0\>_1\subset U_1$,  $(\<Y\times 0\>_1+\sE_1+\rho_1^*(E_1\times\P^1))\cap U_1$ is a simple normal crossing divisor on $U_1$ and $\rho_1^{-1}(Y\times0\cap\cap_{i=1}^{r+1}E_i\times\P^1)\cap W_1^0$ is a Cartier divisor on $W_1^0$.

We claim that $\id_{U_1}$ is in $\sM(U_1)_{E_*}\cap \sM(U_1)_{Y\times 0, E_*, E_{r+1}}$. Indeed, let $U_1':=U_1\setminus |\rho_1^*(Y\times0)|$. Then $U_1'=Y\times(\P^1\setminus 0)$ as a $Y$-scheme and $Y\times0$ pulls back to the empty divisor on $U_1'$, so $\id_{U_1'}$ is in $\sM(U'_1)_{E_*}\cap \sM(U'_1)_{Y\times 0, E_*, E_{r+1}}$. Similarly, $\id_{W_1^0}$ is in $\sM(W^0_1)_{E_*}\cap  \sM(W^0_1)_{Y\times 0,  E_*, E_{r+1}}$ by properties (a)-(c), and as $U_1=W_1^0\cup U_1'$, our claim is verified.

By resolution of singularities, there is a projective birational morphism $\psi:W\to W_1$ with $W\in \Sm/k$ such that $\psi$ has fundamental locus contained in $F_1$,  $\psi^*(\sE_1+\<Y\times 0\>_1+\rho_1^*(E_1\times\P^1))$  is a simple normal crossing divisor on $W$ and $\id_W$ is in $\sM(W)_{E_*}\cap  \sM(W)_{Y\times 0, E_*, E_{r+1}}$. Indeed,  we may proceed as in the construction of $Y'$. First blow up the  subscheme
$\rho_1^{-1}(Y\times0\cap \cap_{i=1}^{r+1}E_i\times\P^1)$ of $W_1$, forming $\psi_1:W_2\to W_1$. Since $\rho_1^{-1}(Y\times0\cap \cap_{i=1}^{r+1}E_i\times\P^1)\cap U_1$ is a Cartier divisor on $U_1$, the fundamental locus of $\psi_1$ is contained in $F_1$.  Next, blow up $W_2$ along a closed subscheme lying over $F_1$, $\psi_2:W\to W_2$, so that $W$ is in $\Sm_k$ and, letting $\psi:=\psi_1\circ\psi_2$,  $\psi^*(\sE_1+\<Y\times 0\>_1+\rho_1^*(E_1\times\P^1))$ a simple normal crossing divisor. Letting $\rho:W\to Y\times\P^1$ be the composition $\rho_1\circ\psi$, we claim that $\rho$ satisfies our conditions (1)-(5). 

Property (1) is a direct consequence of our choice of $Z$.   Property (2) follows from (b).
Properties (3) and (4) are verified in the previous paragraph; to see that  $\rho^*(Y\times0)$ is in good position with respect to $E_*$, we  note that $\rho^*(Y\times0)+\rho^*(E_1\times\P^1)=\<Y\times 0\>+\sE+\rho^*(E_1\times\P^1)$, which is a simple normal crossing divisor on $W$. We have constructed $W$ starting with an arbitrary projective birational morphism $\tau:Y'\to Y$ with $Y'\in\Sm/k$, with fundamental locus contained in $|E_1|\cap |E_{r+1}|$ and with $\tau$ in $\sM(Y)_{E_*, E_{r+1}}$, and from (b), $\<Y\times0\>$ and $Y'$ are isomorphic $Y$-schemes. 

Finally, we see from (a) that the $\rho$ we have constructed satisfies (5). This completes the proof of (A) and (B).

For (C), we note that $\rho_1^*(A\times\P^1)\cap W_1^0=\pi^*(\tau^*(A))$, and $|\rho_1^*(Y\times0)\cap W_1^0|=\<Y\times0\>\cup|\pi^*(\tau^*E_1)|$. Since $\<Y\times0\>$ goes over to the zero-section of $O(1)\to Y'$, and $\tau^*(A+E_1)$ is a simple normal crossing divisor on $Y'$,  it follows  that $\rho_1^*(A\times\P^1+Y\times0+E_1\times \P^1)\cap W_1^0$ is a simple normal crossing divisors on $W_1^0$. Since $A$ is in good position with respect to $E_*$, we see that  $\rho_1^*(A\times\P^1+E_1\times \P^1)\setminus |\sE_1|$ is a simple normal crossing divisor on $W_1\setminus |\sE_1|$, and thus  $\rho_1^*(A\times\P^1+Y\times0+E_1\times \P^1)\cap U_1$ is a simple normal crossing divisor on $U_1$. Taking $\psi:W\to W_1$ as above, we may blow up $W$ further in a closed subscheme lying over $F_1$ and change notation so that $\psi^*\rho_1^*(A\times\P^1+Y\times0+E_1\times \P^1)$ is a simple normal crossing divisor on $W$, and all the properties of $\rho:W\to Y\times\P^1$ listed in (A) and (B) still hold. This proves (C).
\end{proof}

\begin{definition}\label{Def:DistinguishedLifting} Let $X$ be in $\Sch_k$ with pseudo-divisors $D_1,\ldots, D_n$, $D$.\\ 
(1) Let $f:Y\to X$ be in $\sM(X)_{D_*}$ with $Y$ irreducible.  Let $E_1,\ldots, E_r$ be the sequence of Cartier divisors  $\Div f^*D_{i_1},\ldots, \Div f^*D_{i_r}$ on $Y$, where $1\le i_1<\ldots<i_r\le n$ and $\{D_{i_1},\ldots, D_{i_r}\}$ is the set of pseudo-divisors $D_i$ such that $f(Y)\not\subset |D_i|$. If $f(Y)\not\subset |D|$, let $E_{r+1}=\Div f^*D$, if $f(Y)\subset |D|$, let $E_{r+1}=E_r$; in this latter case, if $r=0$, we take $E_1,\ldots, E_r, E_{r+1}$ to be the empty sequence.

Let  $\rho:W\to Y\times\P^1$ be a birational morphism satisfying the conditions \eqref{eqn:DistingLiftingConditions} for $f:Y\to X$, $E_1,\ldots, E_r, E_{r+1}$; in the case of an empty sequence we assume $\rho$ satisfies \eqref{eqn:DistingLiftingConditions} after replacing $|E_1|$ with $Y$. In either case, the element $\rho^*(Y\times0)(1_W^{Y\times 0, D, D_*})_{D_*,D}\in \Omega_*(|\rho^*(Y\times0)|)_{D_*, D}$ is defined. Let $\bar\rho:|\rho^*(Y\times0)|\to Y$ be the map induced by $\rho$. We call the element
\[
(f\circ \bar\rho)_*(\rho^*(Y\times0)(1_W^{Y\times 0, D, D_*})_{D, D_*})
\]
 of $\Omega_*(X)_{D_*, D}$ a {\em distinguished lifting} of $f\in\sM(X)_{D_*}$.\\
\\
(2) Let $\eta=(f:Y\to X, L_1,\ldots, L_m)$ be a cobordism cycle on $X$
with $f\in \sM(X)_{D_*}$. Choose $\rho:W\to Y\times\P^1$ as in (1), and let $\tilde L_i=(p_1\rho)^*L_i$. We call the element
\[
(f\circ \bar\rho)_*\big(\cn(\tilde L_1)\circ\ldots\circ\cn(\tilde L_m)(\rho^*(Y\times0)(1_W^{Y\times 0, D, E_*})_{D, E_*}))\big)
\]
of $\Omega_*(X)_{D_*, D}$ a distinguished lifting of $\eta$.  We extend this notion to arbitrary elements of $\L_*\otimes\sZ_*(X)_{D_*}$ by $\L_*$-linearity.
\end{definition}

\begin{remark}\label{rem:DistLiftRewrite} Take $(f:Y\to X, L_1,\ldots, L_r)\in \sZ(X)_{D_*}$ and let $\rho:W\to Y\times\P^1$
be a morphism satisfying the conditions  \eqref{eqn:DistingLiftingConditions} for the sequence of divisors $E_1,\ldots, E_{r+1}$ given in definition~\ref{Def:DistinguishedLifting}.  Then the distinguished lifting of  $(f:Y\to X, L_1,\ldots, L_r)$ associated to $\rho$ is given  by
\[
f_*\big(\cn(L_1)\circ\ldots\circ\cn(L_r)((p_1\circ\rho)_*
([\rho^*(Y\times0)\to W]_{D_*,D}))\big),
\]
since $\rho^*(Y\times0)(1_W^{Y\times0, E_*, E_{r+1}})_{E_*, E_{r+1}}=[\rho^*(Y\times0)\to |\rho^*(Y\times0)|]_{D_*,D}$.
\end{remark}

The term ``distinguished lifting'' is justified by the following result:

\begin{lemma}\label{lem:DistLifting} Let $D_1,\ldots, D_n, D$ be pseudo-divisors on $X$, and write $D_*:=D_1,\ldots, D_n$. Let $x$ be in  $\L_*\otimes\sZ_*(X)_{D_*}$ and let $x_D\in \Omega_*(X)_{D_*, D}$ be a distinguished lifting. Let $\can:\L_*\otimes\sZ_*(X)_{D_*}\to \Omega_*(X)_{D_*}$ be the canonical map and  $\res_{D_*,D/D_*}:\Omega_*(X)_{D_*, D}\to \Omega_*(X)_{D_*}$ the forgetful map. Then 
\[
\res_{D_*,D/D_*}(x_D)=\can(x)
\]
 in $\Omega_*(X)_{D_*}$.
\end{lemma}

\begin{proof} It suffices to handle of a cobordism cycle $\eta:= (f:Y\to X, L_1,\ldots, L_r)$; as the maps  $\sZ_*(X)_{D_*} \xrightarrow{\can} \Omega_*(X)_{D_*}\xleftarrow{\res} \Omega_*(X)_{D_*, D}$ are compatible with arbitrary 1st Chern class operators, we reduce to the case $x=(f:Y\to X)$ with $Y$ irreducible.  Let $\rho:W\to Y\times\P^1$ be the birational morphism used to define $x_D$ and let $i:|\rho^*(Y\times0)|\to W$ be the inclusion.   The divisors $B_0:=\rho^*(Y\times0)$ and $B_1:=\rho^*(Y\times 1)$ are both in good position with respect to $D_*$ and are linearly equivalent on $W$, and $\id_W$ is in $\sM(W)_{D_*}$.   Hence, by lemma~\ref{lem:DivClass}, we have
\begin{equation}\label{eqn:**}
[\rho^*(Y\times0)\to W]_{D_*}=[\rho^*(Y\times1)\to W]_{D_*}
\end{equation}
in $\Omega_*(W)_{D_*}$. In addition, $\id_W$ is in $\sM(W)_{B_0, D_*}$ and it follows directly from the definition of the operation $(Y\times0)(-)_{D_*, D}$  that 
\[
i_*(\rho^*(Y\times0)(1_W^{Y\times0, D_*, D})_{D_*, D})=[\rho^*(Y\times0)\to W]_{D_*, D}
\]
and thus $x_D=(f\circ p_1)_*([\rho^*(Y\times0)\to W]_{D_*, D})$. 

Pushing forward the identity \eqref{eqn:**} via $f\circ p_1\circ \rho$ yields
\begin{align*}
\res_{D_*,D/D_*}(x_D)&=\res_{D_*,D/D_*}([\rho^*(Y\times0)\to X]_{D_*,D})\\
&=[\rho^*(Y\times0)\to X]_{D_*}\\
&=[\rho^*(Y\times1)\to X]_{D_*}
\end{align*}
in $\Omega_*(X)_{D_*}$. As $\rho^*(Y\times 1)\to X$ is isomorphic as an $X$-scheme to $f:Y\to X$, the result follows.
\end{proof}

This yields the following similar result.

\begin{lemma}\label{lem:TrivDistLift} Let $\eta$ be in $\sZ_*(X)_{D_*, D}$. Then
\[
\Phi_{X, D_*, D}(\res_{D_*, D/D_*}(\eta))=\can(\eta)
\]
in  $\Omega_*(X)_{D_*, D}$.
\end{lemma}

\begin{proof} From the relations described in remark~\ref{rem:DistLiftProps}, we reduce to the case of $\eta=1_X^{D_*,D}\in\sM(X)_{D_*, D}\subset \sM(X)_{D_*}$, with $X$ irreducible and in $\Sm_k$.   Consider the new sequence of pseudo-divisors $D_*, D, D$. One sees directly that  $\sM(X)_{D_*, D}=\sM(X)_{D_*, D,D}$, so we have the element $1_X^{D_*, D,D}$ in $\sM(X)_{D_*, D,D}$ with  $\res_{D_*,D, D/D_*, D}(1_X^{D_*, D, D})=1_X^{D_*, D}$. If $\rho:W\to X\times\P^1$ satisfies the conditions \eqref{eqn:DistingLiftingConditions} for defining the distinguished lifting $\Phi_{X, D_*, D,D}(1_X^{D_*, D})$, then the same $\rho$ satisfies 
the conditions \eqref{eqn:DistingLiftingConditions} for defining the distinguished lifting $\Phi_{X, D_*, D}(1_X^{D_*})$. This implies that
\[
\res_{D_*, D, D/D_*,D}(\Phi_{X, D_*, D,D}(1_X^{D_*, D}))=\Phi_{X, D_*, D}(\res_{D_*, D/D_*}(1_X^{D_*, D})).
\]
By lemma~\ref{lem:DistLifting}, we have $\res_{D_*, D, D/D_*,D}(\Phi_{X, D_*, D,D}(1_X^{D_*, D}))=\can(1_X^{D_*,D})$ in $\Omega_*(X)_{D_*,D}$,  hence $\Phi_{X, D_*, D}(\res_{D_*, D/D_*}(1_X^{D_*, D}))=\can(1_X^{D_*,D})$, as desired. 
\end{proof}

Essential for the construction is the next result.

\begin{proposition}\label{prop:DistLifting} Let $\eta$ be in $\sZ_*(X)_{D_*}$, and let $\eta_1,\eta_2\in \Omega_*(X)_{D_*, D}$ be
distinguished liftings of $\eta$. Then $\eta_1=\eta_2$ in $\Omega_*(X)_{D_*, D}$.
\end{proposition}

\begin{proof} First of all, we may assume that $\eta$ is a cobordism cycle $(f:Y\to X,L_1,\ldots, L_r)$, with $f\in \sM(X)_{D_*}$. Next, it follows from the formula in remark~\ref{rem:DistLiftRewrite} that, if $\tau$ is a distinguished lifting of $(f,L_1,\ldots,L_r)$, then there is a distinguished lifting  $\tilde{1}^{D_*}_Y\in \Omega_*(Y)_{D_*, D}$ of    $1^{D_*}_Y\in\Omega_*(Y)_{D_*}$ with
\[
\tau=f_*(\cn(L_1)\circ\ldots\circ \cn(L_r)(\tilde{1}^{D_*}_Y)).
\]
Thus, it suffices to consider the case of $X\in\Sm_k$, and to show that two distinguished liftings of $1^{D_*}_X\in\sM(X)_{D_*}$ agree in $\Omega_*(X)_{D_*, D}$.

We may assume that $X$ is irreducible. If $|D|=X$, then $\sZ_*(X)_{D_*}=\sZ_*(X)_{D_*,D}$ and $\Omega_*(X)_{D_*}=\Omega_*(X)_{D_*,D}$. By lemma~\ref{lem:DistLifting}, each distinguished lifting of $1^{D_*}_X$ agrees with $1^{D_*}_X$ in $\Omega_*(X)_{D_*}$, hence each two distinguished liftings of $1^{D_*}_X$ agree in $\Omega_*(X)_{D_*,D}$.
 
Thus, we may assume that $D$ is a Cartier divisor on $X$. Let $\eta_1$ and $\eta_2$ be two
distinguished liftings of $1^{D_*}_X$. We may also suppose that $D_1,\ldots, D_n$ are Cartier divisors on $X$ (we allow the case $n=0$).

Let  $(E_1,\ldots,  E_r, E_{r+1})$ be the sequence $(D_1,\ldots, D_n, D)$; in particular, $E_1=D_1$ for $r\ge1$, while $E_1=D$ in case $r=0$. Suppose  for $i=1,2$ that $\eta_i$ is constructed via a birational morphism $\rho_i: W_i\to X\times\P^1$ satisfying \eqref{eqn:DistingLiftingConditions} for the sequence $(E_1,\ldots,  E_r, E_{r+1})$. Let  $Z_i$ be a subscheme of $X\times\P^1$, supported in $|E_1|\times0$, such that $W_i$ is the blow-up of $X\times\P^1$ along $Z_i$, $i=1,2$. Let $\bar\rho_i:|\rho_i^*(X\times0)|\to X$ be the map induced by $\rho_i$, $i=1,2$; thus
\[
\eta_i=\bar\rho_{i*}(\rho_i^*(X\times0)(1_{W_i}^{X\times0, D_*, D}));\quad i=1,2.
\]

Let $\phi_1:T_1\to X\times\P^1\times\P^1$ be the blow-up along $Z_1\times\P^1$,
let $\<Z_2\>$ denote the proper transform of $p_{13}^*(Z_2)$ to $T_1$ and let $\psi_1:T_2\to T_1$ be the blow-up of $T_1$ along $\<Z_2\>$, with structure morphism $\phi_2:T_2\to X\times\P^1\times\P^1$, $\phi_2=\phi_1\circ\psi_1$. Let $\sE_2$ be the exceptional divisor of $T_2\to X\times\P^1\times\P^1$.

We claim there is a blow-up $\psi_2:T\to T_2$ of $T_2$ at a closed subscheme $Z$ supported over $|E_1|\times0\times0$,  such that $T$ is smooth over $k$, $\id_T$ is in $\sM(T)_{X\times\P^1\times0, D_*, D}\cap \sM(T)_{X\times0\times\P^1, D_*, D}$ and the divisor $X\times\P^1\times0+X\times0\times\P^1$ pulls back to a simple normal crossing divisor on $T$.  To see this, define open subschemes $U_2$, $T_2^1$ and $T_2^2$ of $T_2$ by 
\begin{align*}
&U_2:=T_2\setminus\phi^{-1}(|E_1|\times0\times0),\\
&T_2^1:=T_2\setminus\phi^{-1}(|E_1|\times\P^1\times0),\\
&T_2^2:=T_2\setminus\phi^{-1}(|E_1|\times0\times\P^1).
\end{align*}
Let $U_2^1\subset T_2^1$ and $U_2^2\subset T_2^2$ be the open subschemes
\begin{align*}
&U_2^1:=T_2\setminus\phi^{-1}(X\times\P^1\times0),\\
&U_2^2:=T_2\setminus\phi^{-1}(X\times0\times\P^1).
\end{align*}
Finally, let $U=(X\setminus |E_1|)\times \P^1\times\P^1\subset X\times\P^1\times\P^1$.

Clearly  $U_2$ is in $\Sm_k$,  $\phi_2:\phi_2^{-1}(U)\to U$ is an isomorphism and $\id_U$ is in  $\sM(U)_{D_*}\cap\sM(U)_{X\times0\times\P^1, D_*, D}\cap \sM(U)_{X\times\P^1\times0, D_*, D}$. Next,  $U_2^1=W_1\times(\P^1\setminus\{0\})$, so $\id_{U_2^1}$ is in $\sM(U_2^1)_{D_*}\cap \sM(U_2^1)_{X\times0\times\P^1, D_*, D}$. In addition, $\id_{U_2^1}$  is in  $\sM(U_2^1)_{X\times\P^1\times0, D_*, D}$, since $X\times\P^1\times0$ pulls back to the empty divisor on $U_2^1$.  As $T_2^1\setminus U_2^1$ is contained in $\phi_2^{-1}(U)$, this implies that 
\[
\id_{T_2^1}\in \sM(T_2^1)_{D_*}\cap\sM(T_2^1)_{X\times0\times\P^1, D_*, D}\cap\sM(T_2^1)_{X\times\P^1\times0, D_*, D}.
\]
 By symmetry, we have
\[
\id_{T_2^2}\in \sM(T_2^2)_{D_*}\cap\sM(T_2^2)_{X\times0\times\P^1, D_*, D}\cap\sM(T_2^2)_{X\times\P^1\times0, D_*, D}.
\]
Since $U_2=T_2^1\cup T_2^2$, this shows that 
\[
\id_{U_2}\in \sM(U_2)_{D_*}\cap\sM(U_2)_{X\times0\times\P^1, D_*, D}\cap\sM(U_2)_{X\times\P^1\times0, D_*, D}.
\]
Thus, there is a blow up $T_2'\to T_2$ with fundamental locus contained in $T_2\setminus U_2$ so that the closed subschemes of $X\times\P^1\times\P^1$,
\begin{align*}
&X\times\P^1\times0\cap\cap_{i=1}^sE_i\times\P^1\times\P^1;\ s=1, \ldots, n+1,\\
&X\times0\times\P^1\cap\cap_{i=1}^sE_i\times\P^1\times\P^1;\ s=1, \ldots, n+1,
\end{align*}
all pull back to Cartier divisors on $T_2'$. Blowing up $T_2'$ further, again in closed subschemes lying over $T_2\setminus U_2$, we may resolve the singularities of $T_2'$ and achieve that $X\times\P^1\times0+X\times0\times\P^1$ and $D_1\times\P^1\times\P^1$ pull back to  simple normal crossing divisors. This gives us the desired blow-up $\psi_2:T\to T_2$.

Let  $\phi:T\to X\times\P^1\times\P^1$ be the composition $\phi_1\circ \psi_1\circ\psi_2$. We have the Cartier divisors $\sD_0:=\phi^*(X\times0\times\P^1)$, $\sD_1:=\phi^*(X\times1\times\P^1)$, $\sD'_0:=\phi^*(X\times\P^1\times0)$ and $\sD'_1:=\phi^*(X\times\P^1\times1$).  We claim that
\begin{align*}
&\sD_0\text{ is admissible with respect to }\sD_0', D_*, D,\\
&\sD'_0\text{ is admissible with respect to }\sD_0, D_*, D.
\end{align*}
Indeed, by proposition~\ref{prop:Commutativity}, $\sD_0$ is admissible with respect to $X\times \P^1\times0, D_*, D$. Since $\sD'_0=\phi^*(X\times\P^1\times0)$, it follows that $\sD_0$ is admissible with respect to $\sD_0', D_*, D$. The argument for $\sD_0'$ is the same.

By proposition~\ref{prop:Commutativity} we have
\begin{equation}\label{eqn:Com1}
\sD_0([\sD'_0\to |\sD'_0|]_{\sD_0, D_*, D})_{D_*, D}=\sD'_0([\sD_0\to |\sD_0|]_{\sD'_0, D_*, D})_{D_*, D}
\end{equation}
in $\Omega_*( |\sD_0|\cap |\sD'_0|)_{D_*, D}$. 

We note that $T_1$ is isomorphic to $W_1\times\P^1$. Via this isomorphism, let $f:T\to W_1$ be the composition $p_1\circ\psi_1\circ\psi_2$ and let  $V=\phi^{-1}(X\times\P^1\times(\P^1\setminus\{0\}))\subset T$. As $T\to T_1$ is an isomorphism over $\phi^{-1}(X\times\P^1\times(\P^1\setminus\{0\})$, we see that the restriction $f_V:V\to W_1$ of $f$ identifies $V$ with $W_1\times(\P^1\setminus\{0\})$. Let $s:W_1\to V$ be the 1-section. Letting $\bar{D}=\rho_1^*(X\times0)$, we have $\sD_0\cap V=f_V^{-1}(\bar{D})$, $\sD_1'\subset V$ and $\sD_1'$ is the reduced divisor $s(W_1)$. As $\id_{W_1}$ is in $\sM(W_1)_{X\times0, D_*, D}$ be construction, the simple normal crossing divisor $\rho_1^*(X\times0)$ is admissible with respect to $D_*, D$ (lemma~\ref{lem:GoodPosition}(1)). Finally, $O_T(\sD_0')\cong O_T(\sD_1')$. 

Thus, the hypotheses for proposition~\ref{prop:LinearEquiv2} are satisfied (with $E_*=D_*, D$, $B=\sD_0$, $D_0=\sD_0'$ and $D_1=\sD_1'$) and we may conclude that $\sD_0\cap\sD_1'$ is a simple normal crossing divisor on $\sD_1'$, admissible with respect to $D_*, D$, and
\begin{equation}\label{eqn:Com2}
i_{0*}(\sD_0'([\sD_0\to |\sD_0|]_{\sD'_0, D_*, D})_{D_*, D})=i_{1*}([\sD_0\cap\sD_1'\to |\sD_0|\cap|\sD_1'|]_{D_*, D})
\end{equation}
in $\Omega_*(|\sD_0|)_{D_*, D}$, where $i_0: |\sD_0|\cap|\sD_0'|\to |\sD_0|$, $i_1: |\sD_0|\cap|\sD_1'|\to |\sD_0|$ are the inclusions. Similarly, 
 $\sD_0'\cap\sD_1$ is a simple normal crossing divisor on $\sD_1$, admissible with respect to $D_*, D$, and
\begin{equation}\label{eqn:Com3}
i'_{0*}(\sD_0([\sD'_0\to |\sD'_0|]_{\sD_0, D_*, D})_{D_*, D})=i'_{1*}([\sD'_0\cap\sD_1\to |\sD'_0|\cap|\sD_1|]_{D_*, D})
\end{equation}
in $\Omega_*(|\sD'_0|)_{D_*, D}$, where $i'_0: |\sD'_0|\cap|\sD_0|\to |\sD'_0|$, $i'_1: |\sD'_0|\cap|\sD_1|\to |\sD'_0|$ are the inclusions.

Putting \eqref{eqn:Com1},  \eqref{eqn:Com2} and  \eqref{eqn:Com3} together and pushing forward to  $X$ gives the identity
\[
[\sD'_0\cap\sD_1\to X]_{D_*, D}=[\sD_0\cap\sD_1'\to X]_{D_*, D}
\]
in $\Omega_*(X)_{D_*, D}$. But via the isomorphisms $\sD'_1\cong W_1$, $\sD_1\cong W_2$, we have 
\begin{align*}
&[\sD_0\cap\sD'_1\to X]_{D_*, D}=(\bar\rho_1)_*(\rho_1^*(X\times0)(1_{W_1}^{X\times0, D_*, D}))=\eta_1,\\
&[\sD'_0\cap\sD_1\to X]_{D_*, D}=(\bar\rho_2)_*(\rho_2^*(X\times0)(1_{W_2}^{X\times0, D_*, D}))=\eta_2,
\end{align*}
and thus $\eta_1=\eta_2$, completing the proof.
\end{proof}

\begin{remark}\label{rem:DistLiftProps}\ind{distinguished lifting!uniqueness}
Via proposition~\ref{prop:DistLifting}, we may speak of {\em the} distinguished lifting of an
element of
$\L_*\otimes\sZ_*(X)_{D_*}$ to $\Omega_*(X)_{D_*, D}$. We have the following properties of the
distinguished lifting:
\begin{enumerate}
\item Sending $\eta\in\L_*\otimes\sZ_*(X)_{D_*}$ to its distinguished lifting
$\tilde\eta$ defines an $\L_*$-linear homomorphism $\Phi_{X, D_*, D}:\L_*\otimes\sZ_*(X)_{D_*}\to\Om_*(X)_{D_*, D}$, making the diagram
\[
\xymatrixcolsep{50pt}
\xymatrix{
\L_*\otimes\sZ_*(X)_{D_*}\ar[r]^-{\Phi_{X, D_*, D}}\ar[rd]_{\can}&\Om_*(X)_{D_*, D}\ar[d]^{\res_{D_*, D/D_*}}\\
&\Om_*(X)_{D_*}}
\]
commute.
 \item Given $f:X'\to X$ projective, we have
 \[
 \Phi_{X, D_*, D}\circ f_*=f_*\circ \Phi_{X', D_*, D}.
 \]
  
\item If $L$ is a  line bundle on $X$, then 
\[
\Phi_{X, D_*, D}\circ\cn(L)=\cn(L)\circ\Phi_{X, D_*, D}.
\]

\item For $f:X'\to X$ smooth, we have
  \[
 \Phi_{X', D_*, D}\circ f^*=f^*\circ \Phi_{X, D_*, D}.
 \]

\end{enumerate}
Property (1) is just lemma~\ref{lem:DistLifting} and proposition~\ref{prop:DistLifting}. The  properties (2) and (3) follow from the formula in remark~\ref{rem:DistLiftRewrite}. For (4), suppose $g:Y\to X$ is in $\sM(X)_{D_*}$ and $\rho:W\to Y\times\P^1$ is used to construct the distinguished lifting of $g$. As $f$ is smooth, it follows that we may use $W':=X'\times_XW\to X'\times_XY\times\P^1$ to construct the distinguished lifting of $f^*(g)$, from which (4) follows easily.
\end{remark}

\begin{lemma}\label{lem:ChernLifting}  Let $F$ be in $\L_*[[u_1,\ldots, u_m]]$, let $f:W\to X$ be in $\sM(X)_{D_*,D}$, and let $L_1,\ldots,L_m$ be line bundles on $W$. Take an element $\eta\in\sZ_*(W)_{D_*}$and let $F_N$ denote the truncation of $F$ after total degree $N$. Then 
\begin{multline*}
\Phi_{W, D_*, D}(f_*(F_N(\cn(L_1),\ldots,\cn(L_m))(\eta)))\\=f_*(F(\cn(L_1),\ldots,\cn(L_m))(\Phi_{W, D_*, D}(\eta)))
\end{multline*}
for all $N$ sufficiently large.
\end{lemma}

\begin{proof} This follows from remark~\ref{rem:DistLiftProps}, using the relations $\<\sR_*^{Dim}\>(X)_{D_*, D}$.
\end{proof}

\subsection{Lifting divisor classes}\label{subsec:LiftingDiv} We need some information on the distinguished lifting of divisor classes  before proving the main moving lemma. We fix pseudo-divisors $D_1,\ldots, D_n, D$ on $X$. 

Let $f:Y\to X$ be in $\sM(X)_{D_*}$.   We suppose  that $Y$ is irreducible,  $f(Y)\not\subset |D|$ and $f(Y)\not\subset |D_i|$ for $i=1,\ldots, n$.  Let $i:S\to Y$ be a smooth Cartier divisor on $Y$, in good position with respect to $D_*$. As in the proof of lemma~\ref{lem:DistinguishedLifting}, there is a blow-up $\tau:Y'\to Y$ with fundamental locus contained in $|\Div f^*D_1|\cap |\Div f^*D|$ such that $f\circ \tau:Y'\to X$ is in  $\sM(X)_{D_*, D}$. Blowing up further, we may assume that $\tau^*S$ is a simple normal crossing divisor on $Y'$, and is in good position with respect to $D_*, D$. Indeed, this is clearly the case after restriction to $Y'\setminus(|(f\circ\tau)^*D_1|\cap|(f\circ\tau)^*D|)$, and we need only blow $Y'$ up in smooth centers lying over $|(f\circ\tau)^*D_1|\cap|(f\circ\tau)^*D|$ so that $\tau^*S+\Div(f\circ\tau)^*D_1$ pulls back to a simple normal crossing divisor. 

We apply lemma~\ref{lem:DistinguishedLifting} with respect to the sequence of divisors 
\[
(E_1,\ldots, E_r, E_{r+1})=(\Div f^*D_1,\ldots, \Div f^*D_n, \Div f^*D)
\]
and the blow-up $Y'\to Y$,  forming a  projective birational morphism $\rho:W\to Y\times\P^1$  with $W\in \Sm_k$ satisfying the conditions (1)-(5) of that lemma, and such that the proper transform $\<Y\times0\>\subset W$ is isomorphic as $Y$-scheme to $Y'$. By 
part (C) of lemma~\ref{lem:DistinguishedLifting}, there is such a $W$ so that $\rho^*(S\times\P^1)$ is a simple normal crossing divisor on $W$, in good position with respect to $Y\times0, D_*, D$ and with 
$\rho^*(S\times\P^1+Y\times0)$ in good position with respect to $D_*$.

Let $\sE$ denote the exceptional divisor of $\rho$, $\<Y\times0\>$ the proper transform of $Y\times0$ and $\<S\times\P^1\>$ the proper transform of $S\times\P^1$.  Let $\tilde{S}=\rho^*(S\times\P^1)$, $\tilde{Y}=\rho^*(Y\times0)$ and let $i_S:|\tilde{S}|\to W$ be the inclusion.  Let $\bar\rho_S:|\rho_S^*(S\times0)|\to S$ be the map induced by $\rho_S$, and $\bar\rho:|\tilde Y|\to Y$ the map induced by $\rho$,

\begin{lemma}\label{lem:LiftingRelations} Let $q:|\tilde{Y}|\cap|\tilde{S}|\to X$ be the composition of the inclusion into $\tilde{Y}$ with $f\circ\bar\rho$. Then  \\[5pt]
(1) $f_*\bar\rho_*\left(\tilde{Y}(1_W^{Y\times0, D_*, D})_{D_*,D}\right)=\Phi_{X, D_*,D}([f:Y\to X]_{D_*})$.
\\[5pt]
(2) $q_*\left(i_S^*\tilde{Y}([\tilde{S}\to|\tilde S|]_{\tilde{Y}, D_*, D})_{D_*,D}\right)=\Phi_{X, D_*,D}(f_*([S\to Y]_{D_*}))$.
\end{lemma}

\begin{proof}  The assertion (1) follows from the fact that $\rho:W\to Y\times\P^1$ satisfies the conditions (1)-(4) of \eqref{eqn:DistingLiftingConditions}. 

For (2), we first claim that
\begin{equation}\label{eqn:PropTrans}
i_S^*\tilde{Y}([\tilde{S}\to|\tilde{S}|]_{\tilde{Y}, D_*,D})_{D_*,D}=i_S^*\tilde{Y}([\<S\times\P^1\>\to|\tilde{S}|]_{\tilde{Y}, D_*, D})_{D_*, D}
\end{equation}
in $\Omega_*(|\tilde{Y}|\cap|\tilde{S}|)_{D_*, D}$. Indeed, we may write the simple normal crossing divisor $\tilde{S}$ as
\[
\tilde{S} = \<S\times\P^1\> + A,
\]
where $A$ is an effective divisor, supported in $|\tilde{S}|\cap |\sE|$. Since the exceptional divisor $\sE$ is supported in $|\tilde{Y}|$,  $A$ is supported in $|\tilde{Y}|$. From the above decomposition of $\tilde{S}$, we have
\[
[\tilde{S}\to|\tilde{S}|]_{\tilde{Y}, D_*, D}=[\<S\times\P^1\>\to|\tilde{S}|]_{\tilde{Y}, D_*, D}+i^A_*\alpha,
\]
where $\alpha$ is a class in $\Omega_*(|A|)_{\tilde{Y}, D_*, D}$, and $i^A:|A|\to|\tilde{S}|$ is
the inclusion. Then
\begin{multline*}
i_S^*\tilde{Y}([\tilde{S}\to|\tilde{S}|]_{\tilde{Y}, D_*, D})_{D_*, D}\\=i_S^*\tilde{Y}([\<S\times\P^1\>\to|\tilde{S}|]_{\tilde{Y}, D_*, D})_{D_*, D}+
i^A_*\cn(i^{A*}O_W(\tilde{Y}))(\alpha).
\end{multline*}
But $\tilde{Y}=\rho^*(Y\times0)$, hence $O_W(\tilde{Y})\cong O_W$ in a neighborhood of $|A|$. Thus $\cn(i^{A*}O_W(\tilde{Y}))(\alpha)=0$, proving our claim.

We are thus reduced to showing that 
\[
q_*(i_S^*\tilde{Y}([\<S\times\P^1\>\to|\tilde{S}|]_{\tilde{Y}, D_*, D}))_{D_*, D}=\Phi_{X, D_*,D}(f_*([S\to Y]_{D_*})).
\]
We write $W_S$ for $\<S\times\P^1\>$; the restriction of $\rho$  defines a projective  morphism $\rho_S:W_S\to S\times\P^1$.

The fact that  $\rho^*(S\times\P^1)$ is a simple normal crossing divisor on $W$, in good position with respect to $Y\times0, D_*, D$ and with respect to $D_*$, implies that $\<S\times\P^1\>$ is a smooth Cartier divisor on $W$, in good position with respect to $Y\times0, D_*, D$ and with respect to $D_*$. By remark~\ref{rem:GenPosition}, $p_1\circ\rho_S:W_S\to S$ is in $\sM(S)_{D_*}$ and $\rho_S:W_S\to S\times\P^1$ is in $\sM(S\times\P^1)_{S\times0, D_*, D}$. Let $\<S\times0\>_S$ denote the proper transform of $S\times0\subset S\times\P^1$ to $W_S$ and let $S'\subset Y'$ denote the proper transform of $S$ to $Y'$. As $\<Y\times0\> \cong Y'$ as $Y$-schemes, it follows that $\<S\times0\>_S$ is isomorphic to $S'$ as $S$-schemes. Since $\tau^*(S)$ is a simple normal crossing divisor on $Y'$, and is in good position with respect to $D_*, D$, it follows again from remark~\ref{rem:GenPosition} that $S'\to S$ is in $\sM(S)_{D_*, D}$. Since $\rho^*(S\times\P^1+Y\times0)$ is in good position with respect to $D_*$,   lemma~\ref{lem:GoodPosition}(6) tells us that the divisor $\rho_S^*(S\times0)=\rho^*(Y\times0)\cap\<S\times\P^1\>$ on $W_S$ is in good position with respect to $D_*$. This shows that $\rho_S:W_S\to S\times\P^1$ satisfies the conditions (1)-(4) of \eqref{eqn:DistingLiftingConditions}, except possibly for the requirements about the support of the fundamental loci of $\<S\times0\>_S\to S$ and $W_S\to S\times\P^1$; these properties follow directly from the corresponding requirements for  $\<Y\times0\>\to Y$ and $W\to Y\times\P^1$. Thus, we may use $W_S$ to compute the distinguished lifting of $f_*([S\to Y]_{D_*})$, giving 
\begin{align*}
\Phi_{X,D_*, D}(f_*[S\to Y]_{D_*})&=\Phi_{X, D_*,D}((f\circ i)_*(1_S^{D_*}))\\
&=(f\circ i\circ\bar\rho_S)_*\left(\rho_S^*(S\times0)(1_{W_S}^{S\times0,D_*,D})_{D_*,D}\right)\\
&=(f\circ \bar\rho)_*\left(\rho^*(Y\times0)([\<S\times\P^1\>\to W]_{Y\times0, D_*, D})_{D_*, D}\right)\\
&=q_*\left(i_S^*\tilde{Y}([\<S\times\P^1\>\to|\tilde{S}|]_{\tilde{Y}, D_*, D})_{D_*, D}\right)\\
&=q_*\left(i_S^*\tilde{Y}([\tilde{S}\to|\tilde{S}|]_{\tilde{Y}, D_*, D})_{D_*, D}\right),
\end{align*}
as desired.
\end{proof}

\subsection{ The proof of the moving lemma} We are now ready to prove
the main result of this section.

\begin{theorem}\label{thm:Moving}
Let $X$ be a finite type $k$-scheme, and $D_1, \ldots, D_n, D$ 
pseudo-divisors on $X$. Then the forgetful  map
$\res_{D_*, D/D_*}:\Omega_*(X)_{D_*,D}\to\Omega_*(X)_{D_*}$ is an isomorphism.
\end{theorem}

\begin{proof}  By lemma~\ref{lem:DistLifting} and lemma~\ref{lem:TrivDistLift}, we have a commutative diagram
\[
\xymatrix{
\L_*\otimes\sZ_*(X)_{D_*, D}\ar[r]^-{\can}\ar[d]_{\res_{D_*, D/D_*}}&\Omega_*(X)_{D_*,D}\ar[d]^{\res_{D_*, D/D_*}}\\
\L_*\otimes\sZ_*(X)_{D_*}\ar[ur]_{\Phi_{X, D_*,D}}\ar[r]_-{\can}&\Omega_*(X)_{D_*}.
}
\]

Since $\can:\L_*\otimes\sZ_*(X)_{D_*}\to\Omega_*(X)_{D_*}$ is surjective, the surjectivity of $\res_{D_*, D/D_*}$ follows. As $\can:\L_*\otimes\sZ_*(X)_{D_*,D}\to \Omega_*(X)_{D_*,D}$ is surjective, so is $\Phi_{X, D_*,D}$ and the injectivity of $\res_{D_*, D/D_*}$ will follow if we show that $\Phi_{X, D_*,D}$ descends to an $\L_*$-linear homomorphism $\bar{\Phi}_{X, D_*,D}:\Omega_*(X)_{D_*}\to\Omega_*(X)_{D_*, D}$.

First, we show that $\Phi:=\Phi_{X, D_*, D}$ descends to $\Phi_1:\L_*\otimes\uu{\sZ}_*(X)_{D_*}\to\Omega_*(X)_{D_*,D}$. As $\Phi$ is $\L_*$-linear and is compatible with 1st Chern class operators, we need only consider an element of the form $\eta:=(f:Y\to X, \pi^*L_1,\ldots, \pi^*L_m)$ of $\<\sR_*^{Dim}\>(X)_{D_*}$. Here $f:Y\to X$ is in $\sM(X)_{D_*}$, $\pi:Y\to Z$ is a smooth morphism to some $Z\in\Sm_k$, $L_1,\ldots, L_m$ are line bundles on $Z$, and $m>\dim_kZ$. By lemma~\ref{lem:ChernLifting}, 
\[
\Phi(\eta)=f_*\left(\cn(\pi^*L_1)\circ\cdots\circ\cn(\pi^*L_m)(\Phi_{Y, D_*, D}(1_Y^{D_*}))\right).
\]
The operator $\cn(\pi^*L_1)\circ\ldots\circ\cn(\pi^*L_m)$ is zero on $\Omega_*(Y)_{D_*,D}$ by the relations $\<\sR_*^{Sect}\>(Y)_{D_*, D}$, so  $\Phi(\eta)=0$, as desired.

Next, we check that $\Phi_1$ descends to $\Phi_2:\L_*\otimes\uu{\Om}_*(X)_{D_*}\to\Omega_*(X)_{D_*, D}$. Since $\Phi$ is $\L_*$-linear,  intertwines the operators $\cn(L)$ on $\sZ_*(X)_{D_*}$ and on $\Om(X)_{D_*, D}$, and is compatible with pushforward, it suffices to check that $\Phi_1$ vanishes on elements of the form
\[
f_*(\cn(O_Y(S))(1_Y^{D_*})-[i:S\to Y]_{D_*}),
\]
for $f:Y\to X$ in $\sM(X)_{D_*}$, and $i:S\to Y$ the inclusion of a smooth divisor in good position with respect to $D_*$. 

By lemma~\ref{lem:ChernLifting}, 
\[
\cn(O_Y(S))(\Phi_{Y, D_*, D}(1_Y^{D_*}))=\Phi_{Y, D_*, D}(\cn(O_Y(S))(1^{D_*}_Y)).
\]

On the other hand, let $\rho:W\to Y\times\P^1$ be as constructed at the beginning of \S\ref{subsec:LiftingDiv} for the pair $(Y, S)$. Retaining the notation from that section, the Cartier divisor $\rho^*(S\times\P^1)$ is in good position with respect to $Y\times0, D_*, D$ and $[\rho:W\to Y\times \P^1]$ is in $\sM(Y\times\P^1)_{Y\times0, D_*, D}$, so by lemma~\ref{lem:DivClass} 
\[
[\rho^*(S\times\P^1)\to W]_{Y\times0, D_*, D}=\cn(\sO_W(\rho^*(S\times\P^1)))(1_W^{Y\times0, D_*, D}).
\]
Therefore, with $\bar\rho:|\rho^*(Y\times0)\to Y$ the map induced by $\rho$, we have
\begin{align*}
\Phi_{Y, D_*, D}(&[S\to Y]_{D_*})\\&=(\bar\rho)_*(\rho^*(Y\times0)([\rho^*(S\times\P^1)\to W]_{Y\times0, D_*, D})_{D_*, D})\\
&\hskip200pt \text{by lemma~\ref{lem:LiftingRelations}}\\
&=(\bar\rho)_*(\rho^*(Y\times0)(\cn(\sO_W(\rho^*(S\times\P^1)))(1_W^{Y\times0, D_*, D})))\\
&=\cn(O_Y(S))(\bar\rho_*(\rho^*(Y\times0)(1_W^{Y\times0, D_*, D})))\\
&=\cn(O_Y(S))(\Phi_{Y, D_*,D}(1_Y^{D_*})).
\end{align*}
Pushing forward to $X$ gives the desired identity.

Finally, we check that $\Phi_2$ descends to $\bar\Phi:\Omega_*(X)_{D_*}\to\Omega_*(X)_{D_*, D}$.  As above, it suffices to show that
\[
\Phi_2\left(f_*(F_\L(\cn(L),\cn(M))(1_Y^{D_*}) -\cn(L\otimes M)(1_Y^{D_*}))\right)=0
\]
for each $f:Y\to X$ in $\sM(X)_{D_*}$ and each pair of line bundles $L, M$ on $Y$, where $F_\L$ is the universal formal group law.  But 
\begin{align*}
\Phi_{Y, D_*, D}&(F_\L(\cn(L),\cn(M))(1_Y^{D_*}) -\cn(L\otimes M)(1_Y^{D_*}))\\
&=(F_\L(\cn(L),\cn(M))-\cn(L\otimes M))(\Phi_{Y, D_*, D}(1_Y^{D_*}))\\
&=0,
\end{align*}
using the relations $\<\sR_*^{FGL}\>(Y)_{D_*, D}$ in $\Omega_*(Y)_{D_*, D}$. Pushing forward to $X$ gives the desired identity. 

This  completes the descent and the proof of the theorem.
\end{proof}

\section{The intersection map}\label{sec:IntersectionMap} We are now in a position to define the intersection map and describe its main properties.

\begin{definition} Let $X$ be in $\Sch_k$, $D$ a pseudo-divisor on $X$. The {\em intersection map}
\begin{equation}\label{eqn:Intersection}
D(-):\Omega_*(X)\to \Omega_{*-1}(|D|)
\end{equation}
is defined as the composition
\[
\Omega_*(X)\xrightarrow{\res_{D/\0}^{-1}}\Omega_*(X)_D\xrightarrow{D(-)_D}\Omega_{*-1}(|D|);
\]
the map $\res_{D/\0}:\Omega_*(X)_D\to \Omega_*(X)$ being an isomorphism by theorem~\ref{thm:Moving}.
\end{definition}

As an aid to verifying properties of the intersection map, we prove the following technical result.

\begin{lemma} \label{lem:Generators} Take $X\in \Sch_k$ with pseudo-divisors $D_1, D_2$. Then the canonical map
\[
\L_*\otimes(\sZ_*(X)_{D_1, D_2}\cap\sZ_*(X)_{D_2, D_1}\cap \sZ_*(X)_{D_1+D_2})\to \Omega_*(X)
\]
is surjective. 
\end{lemma}

\begin{proof} As the canonical maps $\L_*\otimes\sZ_*(X)_{D_1, D_2}\to \Omega_*(X)_{D_1, D_2}$
and $\L_*\otimes\sZ_*(X)\to \Omega_*(X)$ are surjective, it follows from  lemma~\ref{lem:TrivDistLift} that  $\res\circ\can:\L_*\otimes\sZ_*(X)_{D_1, D_2}\to \Omega_*(X)$ is surjective.  To prove the result, we need only show that  for $f:Y\to X$ in $\sM(X)_{D_1, D_2}$,  $\res_{D_1, D_2/\0}(f)$ is in the image of 
$\L_*\otimes(\sZ_*(X)_{D_1, D_2}\cap\sZ_*(X)_{D_2, D_1}\cap \sZ_*(X)_{D_1+D_2})\to \Omega_*(X)$. 

If either $f(Y)\subset |D_1|$ or $f(Y)\subset |D_2|$, then $f$ is already in the intersection  $\sM(X)_{D_1,D_2}\cap\sM(X)_{D_2,D_1}\cap \sM(X)_{D_1+D_2}$, so we may assume that $f^*D_i$ is a Cartier divisor on $Y$ for $i=1,2$. Changing notation, we may replace $X$ with $Y$ and $D_1, D_2$ with effective Cartier divisors on $Y$, so that $\id_Y$ is in $\sM(Y)_{D_1, D_2}$. Thus $D_1$ is a simple normal crossing divisor on $Y$ and $D_1\cap D_2$ is a Cartier divisor on $Y$. If $g:Z\to Y$ is in $\sM(Y)_{D_1, D_2}\cap \sM(X)_{D_1+D_2}$, then automatically $g$ is in $\sM(Y)_{D_2, D_1}$, as, in case $g(Z)\not\subset |D_1|\cup|D_2|$, then $g^*(D_1+D_2)$ is a simple normal crossing divisors on $Z$, so $g^*D_2$ is a simple normal crossing divisor, and as $g^*(D_1)\cap g^*(D_2)$ is a Cartier divisor on $Z$, $g$ is in $\sM(Y)_{D_2, D_1}$. Thus we need only show that $\id_Y$ is in the image of $\L_*\otimes(\sZ_*(Y)_{D_1, D_2}\cap \sZ_*(Y)_{D_1+D_2})\to \Omega_*(Y)$.

Let $\tau:Y'\to Y$ be a blowup of $Y$, so that $Y'$ is in $\Sm_k$ and $\tau^*(D_1+D_2)$ is a simple normal crossing divisor on $Y$; this can be accomplished with the blowup of a closed subscheme $Z_0$ of $Y$ supported in $|D_2|$. As $D_1\cap D_2$ is a Cartier divisor on $Y$, the same holds for $\tau^*D_1\cap \tau^*D_2$, and clearly $\tau^*D_1$ and $\tau^*D_2$ are simple normal crossing divisors on $Y'$. Thus $Y'\to Y$ is in $\sM(Y)_{D_1, D_2}\cap \sM(Y)_{D_1+D_2}$, $\tau^*D_2$ is in good position on $Y'$ with respect to $D_1, D_2$ and $\tau^*D_1$ is in good position on $Y'$ with respect to $D_2, D_1$. 

Let $Z\subset Y$ be the subscheme with ideal sheaf $\sI_Z:=\sI_{Z_0}\cdot\sI_{D_1}\cdot \sI_{D_2}$ and let $\rho_1:W_1\to Y\times\P^1$ be the blowup of $Y\times\P^1$ along $Z\times0$. As in the proof of lemma~\ref{lem:DistinguishedLifting}, we let $\<Y\times0\>_1$, $\<Z\times\P^1\>_1$ be respective proper transforms of $Y\times0$, $Z\times\P^1$, $\sE_1$ the exceptional divisor of $\rho_1$, $W_1^0:=W_1\setminus|\<Z\times\P^1\>_1|$ and $U_1:=W_1\setminus(|\sE_1|\cap |\<Z\times\P^1\>_1|$. Just as in the proof of  lemma~\ref{lem:DistinguishedLifting},  $W_1^0$ is isomorphic to the line bundle $O(1)\to Y'$ (as $Y$-schemes), with the zero section of $O(1)$ going over to $\<Y\times0\>_1\subset W_1^0$ and with  $O(1)_{|D_1\cup D_2}$ going over to $\sE_1\cap W_1^0$. Thus  $\<Y\times0\>_1$ is isomorphic to $Y'$,
\[
\rho_1^*(Y\times0)\cap U_1=\rho_1^*(Y\times0)\cap W^0_1=\sE_1\cap W_1^0+\<Y\times0\>_1
\]
 is a simple normal crossing divisor  on $U_1$ and $U_1$ is smooth over $k$. Furthermore,
\[
\rho_1^*(Y\times0)\cap \rho_1^*p_1^*(D_1+D_2)\cap U_1=\rho_1^*(Y\times0)\cap \rho_1^*p_1^*(D_1+D_2)\cap W^0_1=\sE_1\cap W_1^0
\]
is a Cartier divisor. Thus  $\id_{U_1}$ is in $\sM(U_1)_{D_1, D_2}\cap \sM(U_1)_{Y\times0, D_1, D_2}\cap \sM(U_1)_{Y\times0, D_1+D_2}$. Thus, we may blow up $W_1$ in a closed subscheme supported in $W_1\setminus U_1$, $\psi:W\to W_1$, so that  $W$ is smooth over $k$,   $\rho:=\rho_1\circ \psi$ is in   
\[
\sM(Y\times\P^1)_{Y\times0, D_1, D_2}\cap \sM(Y\times\P^1)_{Y\times0, D_1+D_2}\cap \sM(Y\times\P^1)_{D_1, D_2}
\]
and $\rho^*(Y\times0)+\rho^*(D_1\times\P^1)$ is a simple normal crossing divisor on $W$. Thus $\rho^*(Y\times0)$ is in good position with respect to $D_1, D2$; clearly $\rho^*(Y\times1)$ is in good position with respect to $D_1, D_2$.

As $\rho^*(Y\times0)$ is in good position with respect to $D_1, D_2$, lemma~\ref{lem:DivClass} gives the identity
\[
[\rho^*(Y\times0)\to Y\times\P^1]_{D_1, D_2}=\rho_*(\cn(\rho^*O_{Y\times\P^1}(Y\times0))(1_W^{D_1, D_2}))
\]
in $\Omega_*(Y\times\P^1)_{D_1, D_2}$. As the smooth divisor $\rho^*(Y\times1)$ on $W$ is in good position with respect to $D_1, D_2$, the relations $\<\sR_*^{Sect}\>(W)_{D_1, D_2}$ give the identity 
\[
[\rho^*(Y\times1)\to Y\times\P^1]_{D_1, D_2})=\rho_*(\cn(O_W(\rho^*(Y\times1))(1_W^{D_1, D_2}))
\]
in $\Omega_*(Y\times\P^1)_{D_1, D_2}$. Since $O_W(\rho^*(Y\times1))\cong O_W(\rho^*(Y\times0))$ and the smooth divisor $\rho^*(Y\times1)$ is as a $Y$-scheme isomorphic to $Y$,  we may push forward to $Y$, giving
\[
[\rho^*(Y\times0)\to Y]_{D_1, D_2}= [\rho^*(Y\times1)\to Y]_{D_1, D_2}=1_Y^{D_1, D_2}
\]
in $\Omega_*(Y)_{D_1, D_2}\cong\Omega_*(Y)$. 

On the other hand,  as $\rho$ is also in $\sM(Y)_{Y\times0, D_1+D_2}$, and 
\[
\un{Y\times0}\left(\un{[\rho:W\to Y\times\P^1]}_{Y\times0, D_1+D_2}\right)_{D_1+D_2}=\un{[\rho^*(Y\times0)\to Y\times\P^1]}_{D_1+D_2}
\]
in $\L_*\otimes\un{\sZ}_*(Y\times\P^1)_{D_1+D_2}$. Pushing forward to $Y$, the two divisor classes 
$\un{[\rho^*(Y\times0)\to Y]}_{D_1+D_2}$ and $\un{[\rho^*(Y\times0)\to Y]}_{D_1,D_2}$ are equal in $\L_*\otimes\un{\sZ}_*(Y)$,  thus 
$\un{[\rho^*(Y\times0)\to Y]}_{D_1, D_2}$ is in 
\[
\L_*\otimes(\un{\sZ}_*(Y)_{D_1+D_2}\cap \un{\sZ}_*(Y)_{D_1,D_2})\subset \L_*\otimes\un{\sZ}_*(Y)
\]
and $\id_Y$ is in the image of $\L_*\otimes(\sZ_*(Y)_{D_1, D_2}\cap \sZ_*(Y)_{D_1+D_2})\to \Omega_*(Y)$, as desired.
\end{proof}

\begin{proposition}\label{prop:IntersectionProperties} Let $D$ be a pseudo-divisor on an $X\in\Sch_k$. 
The intersection map \eqref{eqn:Intersection} has the following properties. 
\\
(1) The map $D(-)$ is $\L_*$-linear.\\
(2) Let $f:X'\to X$ be a projective morphism, $f_D:|f^*D|\to |D|$ the restriction of $f$. Then $f_*\circ (f^*D)(-)=D(-)\circ f_*$.\\
(3) Let $f:X'\to X$ be a smooth quasi-projective morphism, $f_D:|f^*D|\to |D|$ the restriction of $f$. Then $f_D^*\circ D(-)=(f^*D)(-)\circ f^*$.\\
(4) Let $L$ be a line bundle on $X$. Then $\cn(L)\circ D(-)=D(-)\circ \cn(L)$.\\
(5) Take $X'\in \Sch_k$, and let $p:X'\times X\to X$ be the projection. Then for $\eta'\in \Omega_*(X)$, $\eta'\in \Omega_*(X')$, we have 
\[
(p^*D)(\eta'\times \eta)=\eta'\times D(\eta)
\]
in $\Omega_*(X'\times|D|)$.\\
(6) Let $D'$ be a second pseudo-divisor. Then the two maps 
\[
D(-)\circ D'(-), D'(-)\circ D(-):\Omega_*(X)\to \Omega_{*-2}(|D|\cap|D'|)
\]
are equal.\\
(7) Let $D'$ be a second pseudo-divisor. Suppose that $O_X(D)\cong O_X(D')$ and let $i:|D|\to X$, $i':|D'|\to X$ be the inclusions. Then 
\[
i_*\circ D(-)=i'_*\circ D'(-):\Omega_*(X)\to \Omega_{*-1}(X).
\]
\end{proposition}

\begin{proof} The properties (1)-(4) follow directly from lemma~\ref{lem:ProjectionFormula}, lemma~\ref{lem:ChernClassComp} and the definition of the map $D(-)$. For (5), we use (1)-(4) to reduce to the case $X, X'\in \Sm_k$,  $\eta=\id_X$, $\eta'=\id_{X'}$; in particular, $p:X'\times X\to X$ is smooth and quasi-projective. As $1_{X'}\times 1_X=1_{X'\times X}=p^*(1_X)$, we have
\[
p^*(D)(1_{X'}\times 1_X)=p^*(D)(p^*(1_X))=p_D^*(D(1_X))=1_{X'}\times D(1_X),
\]
the second equality following from (3).

For (6), we use lemma~\ref{lem:Generators} to reduce to showing that
\[
D(D'(\eta))=D'(D(\eta))
\]
for $\eta\in \L_*\otimes(\sZ_*(X)_{D, D'}\cap\sZ_*(X)_{D', D}\cap \sZ_*(X)_{D+D'})$. Since the maps $D(-), D'(-)$ are $\L_*$-linear and compatible with 1st Chern class operators, we reduce to showing $D(D'([f:Y\to X]))=D'(D([f:Y\to X])$ for $f\in 
\sM_*(X)_{D, D'}\cap\sM_*(X)_{D', D}\cap \sM_*(X)_{D+D'}$. As the intersection maps are compatible with $f_*$, we  reduce to  the case $X=Y$ and $f=\id_Y$. We may assume that $Y$ is irreducible. Since $\id_Y$ is in $\sM_*(X)_{D, D'}\cap\sM_*(X)_{D', D}$, it is not necessary to apply the inverse of any of the forgetful maps to compute the intersection maps.

If $|D|=|D'|=Y$, then
\begin{multline*}
D(D'(1_Y^{D, D'}))=\cn(O_Y(D))(\cn(O_Y(D'))(1_Y))\\=\cn(O_Y(D'))(\cn(O_Y(D))(1_Y))=D'(D(1_Y^{D', D})).
\end{multline*}
If $|D|=Y$, but $|D'|\neq Y$, then 
\begin{multline*}
D(D'(1_Y^{D, D'}))=\cn(O_Y(D))(D'(1_Y^{D, D'}))\\=D'(\cn(O_Y(D))(1^{D'}_Y))=D'(D(1_Y^{D, D'})).
\end{multline*}
By symmetry, we have the desired identity in case $|D'|=Y$, but $|D|\neq Y$ as well.

Suppose both $D$ and $D'$ are Cartier divisors on $Y$. Since $\id_Y$ is in $\sM_*(Y)_{D+D'}\cap \sM(Y)_{D,D'}\cap\sM(Y)_{D',D}$,  $D+D'$ is 
simple normal crossing divisor. We may then apply proposition~\ref{prop:Commutativity} to yield
\[
i_{D'}^*(D)([D'\to |D'|]_D)=i_D^*(D')([D\to |D|]_D)
\]
in $\Omega_*(Y)$. Thus
\begin{align*}
D(D'(1_Y^{D, D'}))&=i_{D'}^*(D)(D'(1_Y^{D, D'}))\\
&=i_{D'}^*(D)([D'\to |D'|]_D)\\
&=i_D^*(D')([D\to |D|]_{D'})\\
&=D'(D(1_Y^{D', D})).
\end{align*}

For (7), we may assume as above that $X=Y$ is in $\Sm_k$, irreducible and that $\id_Y$ is in $\sM_*(Y)_{D+D'}\cap \sM(Y)_{D,D'}\cap\sM(Y)_{D',D}$. It suffices to show that $i_*(D(\id_Y))=i'_*(D'(\id_Y))$. Our assumptions imply that $\id_Y$ is in $\sM(Y)_{D}\cap\sM(Y)_{D'}$.

Suppose $Y=|D|$. Then $i=\id_Y$ and $D(\id_Y)=\cn(O_Y(D))(\id_Y)$. Thus, if $|D'|=Y$, we have $i_*(D(\id_Y))=i'_*(D'(\id_Y))$. If $|D'|\neq Y$, then $D'$ is a simple normal crossing divisor on $Y$ and $D'(\id_Y)=[D'\to|D'|]$. Thus $D'$ is in good position with respect to the empty sequence $\0$ and by lemma~\ref{lem:DivClass}, 
\[
i'_*(D'(\id_Y))=\cn(O_Y(D'))(\id_Y)=\cn(O_Y(D))(\id_Y)=i_*(D(\id_Y)).
\]

If both $|D|\neq Y$ and $|D'|\neq Y$, then the same computation as above gives the desired result.
\end{proof}

 \end{document}